\documentclass[12pt]{article}
\usepackage[margin=1in]{geometry}  
\usepackage{graphicx}              
\usepackage{amsmath}               
\usepackage{amsthm}                
\usepackage[ansinew]{inputenc}
\usepackage{array}
\usepackage{xcolor}
\usepackage{amsxtra}
\usepackage{amstext}
\usepackage{latexsym}
\usepackage{dsfont}
\usepackage[all]{xy}
\usepackage{amsmath,amsfonts,amssymb}
\usepackage{soul}

\usepackage{marginnote}

\usepackage{enumitem}  
\usepackage{calc}
\setlist{labelindent=1pt,itemsep=.5em}
\setlist[itemize]{leftmargin=1.2cm}
\setlist[enumerate]{itemindent=0em,leftmargin=1.2cm}
\setlist[enumerate,1]{label={\upshape(\roman*)}}

\usepackage{cite}

\usepackage{hyperref}
\hypersetup{
colorlinks   = true,
citecolor    = black,
linkcolor =black,
urlcolor=black,
}

\hypersetup{hidelinks}

\usepackage{authblk}

\newcommand{\email}[1]{%
    \normalsize\href{mailto:#1}{\color{black}{#1} }}

\allowdisplaybreaks

\makeatletter
\newcommand{\subjclass}[2][2020]{%
  \let\@oldtitle\@title%
  \gdef\@title{\@oldtitle\footnotetext{#1 \emph{Mathematics subject classification}: #2}}%
}
\newcommand{\keywords}[1]{%
  \let\@@oldtitle\@title%
  \gdef\@title{\@@oldtitle\footnotetext{\emph{Keywords}: #1}}%
}
\makeatother

\newtheorem{thm}{Theorem}[section]
\newtheorem{cor}[thm]{Corollary}
\newtheorem{lem}[thm]{Lemma}
\newtheorem{prop}[thm]{Proposition}
\theoremstyle{definition}
\newtheorem{defn}[thm]{Definition}
\theoremstyle{remark}
\newtheorem{rmk}[thm]{Remark}
\theoremstyle{remark}

\newtheorem{ex}[thm]{Example}
\newtheorem{exes}[thm]{Examples}
\numberwithin{equation}{section}

\title{Hom-Leibniz bialgebras and BiHom-Leibniz dendriform algebras}

\author[1,2]{Ismail Laraiedh}
\author[3]{Sergei Silvestrov}
\affil[1]{\Affilfont Departement of Mathematics, Faculty of Sciences,
\authorcr \Affilfont Sfax University, Box 1171, 3000 Sfax, Tunisia
\authorcr \Affilfont
\email{ismail.laraiedh@gmail.com}}
\affil[2]{Departement of Mathematics,
\authorcr \Affilfont College of Sciences and Humanities Al Quwaiiyah,
\authorcr \Affilfont Shaqra University, Kingdom of Saudi Arabia
\authorcr \Affilfont
\email{ismail.laraiedh@su.edu.sa}}
\affil[3]{\Affilfont Division of Mathematics and Physics,
\authorcr \Affilfont School of Education, Culture and Communication,
\authorcr \Affilfont M\"{a}lardalen University, Box 883, 72123 V{\"a}ster{\aa}s, Sweden
\authorcr \Affilfont
\email{sergei.silvestrov@mdh.se}}


\date{\today}

\begin{document}
\maketitle

\abstract{
The notion of a Hom-Leibniz bialgebra is introduced and it is shown that matched pairs of Hom-Leibniz algebras, Manin triples of Hom-Leibniz algebras and Hom-Leibniz bialgebras are equivalent in a certain sense. The notion of Hom-Leibniz dendriform algebra is established, their bimodules and matched pairs are defined and their properties and theorems about their interplay and construction are obtained. Furthermore, the concept of BiHom-Leibniz dendriform algebras is introduced and discussed, their bimodules and matched pairs are constructed and properties are described. Finally, the connections between all these algebraic structures using $\mathcal{O}$-operators are shown.
}

\section{Introduction}

Hom-Lie algebras and quasi-Hom-Lie algebras were introduced first by Hartwig, Larsson, and Silvestrov in 2003 in \cite{HLSPrepr2003JA2006:deformLiealgsigmaderiv} devoted to a general method for construction of deformations and discretizations of Lie algebras of vector fields and deformations of Witt and Virasoro type algebras based on general twisted derivations ($\sigma$-derivations) obeying twisted Leibniz rule, and motivated also by the examples of $q$-deformed Jacobi identities in $q$-deformations of Witt and Visaroro algebras and in related $q$-deformed algebras discovered in 1990'th in string theory, vertex models of conformal field theory, quantum field theory and quantum mechanics, and $q$-deformed differential calculi and $q$-deformed homological algebra
\cite{aizawasato199091:qdefViralgcenext,
ChaiElinPop1990:qconfalgcentext,
ChaiIsLukPopPresn1991:Viralgconfdim,
ChaiKuLuk,
ChaiPopPres,
daskaloyannis1992generalized,
Hu1992:qWittalgqLie,
Kassel92,
LiuKQuantumCentExt,
LiuKQ1992:CharQuantWittAlg,
LiuKQPhDthesis}.
In 2005, Larsson and Silvestrov introduced quasi-Lie and quasi-Leibniz algebras  \cite{LarssonSilvestrov2005:QuasiLiealgebras} and color quasi-Lie and color quasi-Leibniz algebras  \cite{LarssonSilvestrov2005:GradedquasiLiealgebras}
incorporating within the same framework Hom-Lie algebras and quasi-Hom-Lie algebras, the color Hom-Lie algebras and Hom-Lie superalgebras, the color quasi-Hom-Lie algebras and quasi-Hom-Lie superalgebras, as well as the quasi-Leibniz algebras, the color quasi-Leibniz algebras and super quasi-Leibniz algebras.
The central extensions and cocycle conditions have been first considered for quasi-Hom-Lie algebras and Hom-Lie algebras in
\cite{HLSPrepr2003JA2006:deformLiealgsigmaderiv,LarssonSilvJA2005:QuasiHomLieCentExt2cocyid} and for graded color quasi-Hom-Lie algebras in \cite{SigurdssonSilvestrov2009:colorHomLiealgebrascentralext}.
In quasi-Lie algebras, the skew-symmetry and the Jacobi identity are twisted in a special way by linear maps, and in general the quasi-Lie Jacobi identity containes six twisted triple bracket terms. In Hom-Lie algebras, the bilinear product satisfies the non-twisted skew-symmetry property as in Lie algebras, and the Hom-Lie algebras Jacobi identity has three terms twisted by a single linear map. Lie algebras are a special case of Hom-Lie algebras when the twisting linear map is the identity map. For other twisting linear maps however the Hom-Lie algebras are different and in many ways richer algebraic structures with classifications, deformations, representations, morphisms, derivations and homological structures in the fundamental ways dependent on joint properties of the twisting map and bilinear product which are in the intricate way interlinked by Hom-Jacobi identity. Hom-Lie admissible algebras have been considered first in \cite{MakhoufSilvestrov:Prep2006JGLTA2008:homstructure}, where
the Hom-associative algebras and more general $G$-Hom-associative algebras including the Hom-Vinberg algebras (Hom-left symmetric algebras), Hom-pre-Lie algebras (Hom-right symmetric algebras), and some other new Hom-algebra structures have been introduced and shown to be Hom-Lie admissible, in the sense that the operation of commutator as new product in these Hom-algebras structures yields Hom-Lie algebras. Furthermore, in \cite{MakhoufSilvestrov:Prep2006JGLTA2008:homstructure}, flexible Hom-algebras have been introduced and connections to Hom-algebra generalizations of derivations and of adjoint derivations maps have been considered, investigations of the classification problems for Hom-Lie algebras have been initiated with constriction of families of the low-dimensional Hom-Lie algebras. Following the footsteps of \cite{HLSPrepr2003JA2006:deformLiealgsigmaderiv,
LarssonSilvJA2005:QuasiHomLieCentExt2cocyid,
LarssonSilvestrov2005:QuasiLiealgebras,
LarssonSilvestrov2005:GradedquasiLiealgebras,
LarssonSilv:quasidefsl2,
MakhoufSilvestrov:Prep2006JGLTA2008:homstructure},
many Hom-algebra structures were developed and shown to link on the fundamental axiomatic level within common framework various deformed and non-deformed associative and non-associative algebraic structures, super-algebras and color algebras, $n$-ary algebraic structures, their formal deformations and quantum deformations and related non-commutative differential and homological algebra structures.
In \cite{AbdaouiMabroukMakhlouf2017:rotabaxteropspreliesuper,
AmmarEjbehiMakhlouf2011:cohomhomdeformhomalg,
ArmakanFarhangdoost:2017IJGMMP:GeomaspectsextHomLiesuperalgs,
AttanLaraiedh:2020ConstrBihomalternBihomJordan,
Bakayoko:LaplacehomLiequasibialg,
Bakayoko:LmodcomodhomLiequasibialg,
BakBan:bimodrotbaxt,
BakyokoSilvestrov:HomleftsymHomdendicolorYauTwi,
BakyokoSilvestrov:MultiplicnHomLiecoloralg,
BenMakh:Hombiliform,
BenAbdeljElhamdKaygorMakhl201920GenDernBiHomLiealg,
ChtiouiMabroukMakhlouf1,
ChtiouiMabroukMakhlouf2,
DassoundoSilvestrov2021:NearlyHomass,
GrMakMenPan:Bihom1,
HassanzadehShapiroSutlu:CyclichomolHomasal,
HounkonnouDassoundo:centersymalgbialg,
HounkonnouHoundedjiSilvestrov:DoubleconstrbiHomFrobalg,
HounkonnouDassoundo:homcensymalgbialg,
kms:narygenBiHomLieBiHomassalgebras2020,
Laraiedh1:2021:BimodmtchdprsBihomprepois,
LarssonSigSilvJGLTA2008,
LarssonSilvJA2005:QuasiHomLieCentExt2cocyid,
LarssonSilv:quasidefsl2,
LarssonSigSilvJGLTA2008:QuasiLiedefFttN,
LarssonSilvestrovGLTMPBSpr2009:GenNComplTwistDer,
MaMakhSil:CurvedOoperatorSyst,
MaMakhSil:RotaBaxbisyscovbialg,
MaMakhSil:RotaBaxCosyCoquasitriMixBial,
MaZheng:RotaBaxtMonoidalHomAlg,
MabroukNcibSilvestrov2020:GenDerRotaBaxterOpsnaryHomNambuSuperalgs,
Makhl:HomaltHomJord,
Makhlouf2010:ParadigmnonassHomalgHomsuper,
MakhloufHomdemdoformRotaBaxterHomalg2011,
MakhSil:HomHopf,
MakhSilv:HomDeform,
MakhSilv:HomAlgHomCoalg,
MakYau:RotaBaxterHomLieadmis,
MandalMishra:HomGerstenhaberHomLiealgebroids, MishraSilvestrov:SpringerAAS2020HomGerstenhalgsHomLiealgds,
RichardSilvestrovJA2008,
RichardSilvestrovGLTbnd2009,
SaadaouSilvestrov:lmgderivationsBiHomLiealgebras,
ShengBai:homLiebialg,
Sheng:homrep,
ShengTang2017:LeibnizbialgsRotaBaxLeibnizYangBax,
ShengXiong:LMLA2015:OnHomLiealg,
SigurdssonSilvestrov2009:colorHomLiealgebrascentralext,
SilvestrovParadigmQLieQhomLie2007,
SilvestrovZardeh2021:HNNextinvolmultHomLiealg,
QSunHomPrealtBialg,
WangZhangWei2015:HomLeibnizsuperalg,
Yau:ModuleHomalg,
Yau:HomEnv,
Yau:homalghomol,
Yau:HombialgcomoduleHomalg,
Yau:HomYangBaHomLiequasitribial,
YauHomMalcevHomalternHomJord
}
Hom-modules and Hom-bimodules, Hom-Lie admissible Hom-coalgebras, Hom-coalgebras, Hom-bialgebras, Hom-Hopf algebras, $L$-modules, $L$-comodules and Hom-Lie quasi-bialgebras, $n$-ary generalizations of BiHom-Lie algebras and BiHom-associative algebras and generalized derivations, Rota-Baxter operators, Hom-dendriform color algebras, Rota-Baxter bisystems and covariant bialgebras, Rota-Baxter cosystems, coquasitriangular mixed bialgebras, coassociative Yang-Baxter pairs, coassociative Yang-Baxter equation and generalizations of Rota-Baxter systems and algebras, curved $\mathcal{O}$-operator systems and their connections with tridendriform systems and pre-Lie algebras, BiHom-algebras, BiHom-Frobenius algebras and double constructions, infinitesimal BiHom-bialgebras and Hom-dendriform $D$-bialgebras and Hom-omni-Lie algebras and Hom-Gerstenhaber algebras and Hom-Lie algebroids.

Dendriform algebras, introduced by Loday as algebras with two operations which dichotomize the notion of associative algebras \cite{JLodayDialgebras}, are connected to K-theory, Hopf algebras, homotopy Gerstenhaber algebra, operads, homology, combinatorics, and quantum field theory where they occur in the theory of renormalization of Connes and Kreimer. Then, the tridendriform algebras were
introduced by Loday and Ronco in their study of polytopes and Koszul duality \cite{LodayRonco2004:Trialgpolytops}. Connections of dendriform algebras to Rota-Baxter algebras have been considered in \cite{EbrahimiFardGuo08}.
Hom-tridendriform algebras, generalizing tridendriform algebras by twisting their axioms by arbitrary linear maps in a special ways, together with connections to Rota-Baxter Hom-algebras generalising Rota-Baxter algebras, have been introduced in \cite{MakhloufHomdemdoformRotaBaxterHomalg2011}.
Connections between Hom-bialgebras and Hom-coalgebras, Hom-dendiform and tridendiform systems and Rota-Baxter and Hom-Rota-Baxter Hom-algebra structures have been investigated in
\cite{MaMakhSil:CurvedOoperatorSyst,
MaMakhSil:RotaBaxbisyscovbialg,
MaMakhSil:RotaBaxCosyCoquasitriMixBial,
MaZheng:RotaBaxtMonoidalHomAlg,
MakhloufHomdemdoformRotaBaxterHomalg2011,
MakYau:RotaBaxterHomLieadmis}.

In this article, we introduce the notion of a Hom-Leibniz bialgebras and show that matched pairs of Hom-Leibniz algebras, Manin triples of Hom-Leibniz algebras and Hom-Leibniz bialgebras are equivalent in a certain sense. Next, we establish the notion of Hom-Leibniz dendriform algebra. Their bimodules and matched pairs are defined and the relevant properties and theorems are obtained. Furthermore, the concept of BiHom-Leibniz dendriform algebras is introduced and discussed. Their bimodules and matched pairs are also constructed, and related relevant properties are given. Finally, the connections between all these algebraic structures using $\mathcal{O} $-operators are shown.

The paper is organized as follows. In Section \ref{sec:prelimbasics},
we summarize the definition of Hom-Leibniz algebras and the concepts
of bimodule and matched pairs and we give some relevant properties. In Section \ref{sec:manintripplesbialgshomleibniz}, we introduce the notions of a Manin triple of Hom-Leibniz
algebras and a Hom-Leibniz bialgebra. We prove the equivalence between matched pairs of Hom-Leibniz
algebras, Manin triples of Hom-Leibniz algebras and Hom-Leibniz bialgebras. In Section \ref{sec:homleibnizdendrifalgs},
we establish definition of Hom-Leibniz dendriform algebra. Similar to the connection from Hom-pre-Lie algebras to Hom-Lie algebras and from Hom-dendriform algebras to Hom-associative algebras, we show that a Hom-Leibniz dendriform algebra gives rise to a Hom-Leibniz algebra. Next,
we develop some construction theorems. Their bimodule and matched pair are defined
and their related relevant properties are also given. Section \ref{sec:bihomleibniz}, is devoted to reminders of fundamental concepts for BiHom-Leibniz algebras.
Some results of bimodules and matched pair are given. In
Section \ref{sec:bileibnizdendriformalgs}, we introduce definition of BiHom-Leibniz dendriform algebra and we give some key of constructions. Their bimodule and matched pair are defined and their related relevant properties are also given. Also, we study the notion of $\mathcal{O}$-operator and we illustrate the relations existing between BiHom-Leibniz and BiHom-Leibniz dendriform algebras.

\section{Preliminaries on Hom-Leibniz algebras}   \label{sec:prelimbasics}
In this section definitions and basic facts on Hom-Leibniz algebras are reviewed.

Throughout this article all linear spaces are assumed to be over a field $\mathbb{K}$, and
multilinear maps $f\colon V_1 \times \cdots \times V_n \to W $ on finite direct products and linear maps $ F\colon V_1 \otimes \cdots \otimes V_n \to W\text{,}$ on finite tensor products of linear spaces are identified in the usual way via $F(v_1\otimes \cdots \otimes v_n) = f(v_1,\ldots,v_n)$.

Hom-modules are pairs $(M,\alpha_M)$ consisting of a
$\mathbb{K}$-module $M$ over a field $\mathbb{K}$ and a linear map $\alpha_M:
M\longrightarrow M$. A morphism $f: (M,\alpha_M)\rightarrow
(N,\alpha_N)$ of Hom-modules is a linear map
$f: M\longrightarrow N$ such that $f\alpha_M=\alpha_N f$.
Hom-algebras are triples $(A,\mu,\alpha)$ where $(A,\alpha)$ is
a Hom-module and $\mu : A \times A \rightarrow A$ is a bilinear map.
If, moreover, the linear map $\alpha: A\rightarrow A$ satisfies the multiplicativity,
$\alpha(\mu(x, y))=\mu(\alpha(x),\alpha(y)),$ that is if $\alpha$ is an algebra endomorphism of the algebra $(A,\mu)$, then the Hom-algebra $(A, \mu(\cdot,\cdot), \alpha)$ is called multiplicative.

\begin{defn}[\cite{LarssonSilvestrov2005:QuasiLiealgebras, MakhoufSilvestrov:Prep2006JGLTA2008:homstructure}]
\label{def:homleibnizalg}
Hom-Leibniz algebras are triples $( A, [\cdot,\cdot], \alpha)$ consisting of a linear space $ A$ over a field $\mathbb{K}$, a bilinear map $[\cdot,\cdot]$: $ A\times A\rightarrow  A$ and a linear map $\alpha:  A\rightarrow A$ satisfying for all $x,y,z \in A$,
\begin{equation}\label{Leibnizidentity}
[\alpha(x),[y, z]]=[[x,y],\alpha(z)]+[\alpha(y),[x,z]]. \quad \quad \text{(Hom-Leibniz identity)}
\end{equation}
If, moreover, the linear map $\alpha: A\rightarrow A$ is an algebra endomorphism, meaning that it satisfies also the multiplicativity property
\begin{eqnarray} \label{multipl:homleibnizhomlie}
\alpha([x, y])=[\alpha(x),\alpha(y)],
\end{eqnarray}
then the Hom-Leibniz algebra is called multiplicative.
\end{defn}

\begin{defn}[\cite{HLSPrepr2003JA2006:deformLiealgsigmaderiv,MakhoufSilvestrov:Prep2006JGLTA2008:homstructure}]
\label{def:homliealg}
Hom-Lie algebras are triples $( A, [\cdot,\cdot], \alpha)$ consisting of a linear space $A$ over a field $\mathbb{K}$, a bilinear map $[\cdot,\cdot]$: $ A\times A\rightarrow  A$ and a linear map $\alpha:  A \rightarrow A$ satisfying for all $x,y,z \in A$,
\begin{align}
\label{skewsymmetry}
[x,y]&=-[y,x] &\quad \quad \text{(Skew-symmetry identity)}\\
\label{homliejacobiidentity}
[\alpha(x),[y, z]] &+[\alpha(y),[z,x]]+[\alpha(z),[x,y]] = 0. &\quad \quad \text{(Hom-Jacobi identity)}
\end{align}
When, moreover, the linear map $\alpha: A\rightarrow A$ satisfies multiplicativity \eqref{multipl:homleibnizhomlie}, that is when $\alpha$ is an algebra endomorphism,
the Hom-Lie algebra $(A, [\cdot,\cdot], \alpha)$ is called multiplicative.
\end{defn}

\begin{rmk}
If $\alpha=id_{ A}$, then
$( A, [\cdot,\cdot])=( A, [\cdot,\cdot], id_{A})$ in Definition \ref{def:homleibnizalg} is a Leibniz algebra. So, Leibniz algebras are a special case of Hom-Leibniz algebras when the twisting linear map is the identity map, that is when the fundamental defining Leibniz identity is not deformed at all.
\end{rmk}
\begin{exes} Some basic examples of Hom-Leibniz algebras are as follows.
\begin{enumerate}[label=\upshape{\arabic*.},leftmargin=30pt]
\item
Any Hom-Lie algebra is automatically a Hom-Leibniz algebra as
in the presence of skew-symmetry and the Hom-Leibniz identity implies Hom-Jacobi identity.
\item
Given a Leibniz algebra $(A, [\cdot,\cdot])$ and a Leibniz algebra morphism
$\alpha: A\rightarrow A$, one always get a Hom-Leibniz algebra $(A, [\cdot,\cdot]_\alpha, \alpha)$, where $[x, y]_\alpha =
[\alpha(x), \alpha(y)].$
\item
Any Hom-module $(M, \alpha_M )$, together with the trivial zero bracket $[\cdot,\cdot]$, that is when $[x, y]= 0$ for all $x, y \in M$, is a Hom-Leibniz algebra, called abelian or commutative Hom-Leibniz algebra.
\item
The two-dimensional $\mathbb{K}$-linear space $A$ with basis $\{e_1, e_2\}$, the
bilinear operation given by $[e_2, e_2] = e_1$ and zero elsewhere, and the endomorphism $\alpha$ given by $\alpha(e_1)=e_1,~~\alpha(e_2)=e_1+e_2$ define a Hom-Leibniz algebra. This Hom-Leibniz algebra is multiplicative, since
$$\alpha([e_j,e_k])=
\left\{
\begin{array}{l}
\alpha(0)=0=[\alpha(e_j),\alpha(e_k)], \quad (j,k) \neq (2,2) \\
\begin{aligned}
\alpha(e_1)=e_1 = [e_2, e_2] =[e_1+e_2, e_1+e_2]
\\
=[\alpha(e_2), \alpha(e_2)], \quad j=k=2
\end{aligned}
\end{array}
\right.
$$
but it is not a Hom-Lie algebra since the skew-symmetry identity is not satisfied
$[e_2, e_2]=e_1\neq -e_1 = -[e_2, e_2].$
\end{enumerate}
\end{exes}
\begin{defn}[\cite{ChengSuAMS2011:CohomunivextHomLeibniz}]
Let $( A, [\cdot,\cdot], \alpha)$ be a Hom-Leibniz algebra and let $(V, \beta)$ be a Hom-module. Let $ l, r:  A \rightarrow gl(V) $ be two linear maps. The quadruple  $(l, r, \beta, V)$ is called a bimodule of $ A$ if for all $ x, y \in   A, v \in V $:
\begin{align} \label{Cond:HomBimod:lpb}
l(\alpha(x))l(y)v&=l([x,y])\beta(v)+l(\alpha(y))l(x)v,\\
 \label{Cond:HomBimod:rpb}
l(\alpha(x))r(y)v&=r(\alpha(y))l(x)v+r([x,y])\beta(v), \\
\label{Cond:HomBimod:lar}
 r([x,y])\beta(v)&=r(\alpha(y))r(x)v+l(\alpha(x))r(y)v,
\\
\label{Hombimodule:bl:eq1}
\beta(l(x)v) &= l(\alpha(x))\beta(v), \\
\label{Hombimodule:br:eq2}
\beta(r(x)v) &= r(\alpha(x))\beta(v).
\end{align}
\end{defn}

\begin{rmk}
By using \eqref{Cond:HomBimod:rpb} and \eqref{Cond:HomBimod:lar}, observe that for all $x,y\in A$ and $v\in V$,
\begin{equation}\label{Cond:HomBimod:rpblar}
    r(\alpha(y))l(x)v=-r(\alpha(y))r(x)v.
\end{equation}
\end{rmk}
\begin{prop}
Let $(l, r, \beta, V)$ is a bimodule of a Hom-Leibniz algebra $( A, [\cdot,\cdot], \alpha)$. Then the direct sum  of linear spaces, $ A \oplus V$,  is turned into a Hom-Leibniz algebra  by defining multiplication in $  A \oplus V $ for all $ x_{1}, x_{2} \in   A, v_{1}, v_{2} \in V$ by
\begin{eqnarray}
[x_{1} + v_{1}, x_{2} + v_{2}]' &=& [x_{1} , x_{2}] + (l(x_{1})v_{2} + r(x_{2})v_{1}),\cr
(\alpha\oplus\beta)(x_{1} + v_{1})&=&\alpha(x_{1}) + \beta(v_{1}).
\end{eqnarray}
\end{prop}
\begin{proof}
For all $v_{1}, v_{2}, v_{3}\in V$ and $x_{1}, x_{2},x_{3}\in  A$,
\begin{align*}\label{condit. du Bimod.}
& [(\alpha+\beta)(x_1+v_1),[x_2+v_2,x_3+v_3]']'\\&\quad\quad-
[[x_1+v_1,x_2+v_2]',(\alpha+\beta)(x_3+v_3)]'\\&\quad\quad-[(\alpha+\beta)(x_2+v_2),[x_1+v_1,x_3+v_3]']'\\
& \quad =[\alpha(x_1)+\beta(v_1),[x_2,x_3]+l(x_2)v_3+r(x_3)v_2]'\\
&\quad\quad-[[x_1,x_2]+l(x_1)v_2+r(x_2)v_1,\alpha(x_3)+\beta(v_3)]'\\
&\quad\quad-[\alpha(x_2)+\beta(v_2),[x_1,x_3]+l(x_1)v_3+r(x_3)v_1]'
\\
&\quad=[\alpha(x_1),[x_2,x_3]]+l(\alpha(x_1))l(x_2)v_3+l(\alpha(x_1))r(x_3)v_2+r([x_2,x_3])\beta(v_1)\\&\quad\quad-[[x_1,x_2],\alpha(x_3)]-l([x_1,x_2])\beta(v_3)-r(\alpha(x_3))l(x_1)v_2-r(\alpha(x_3))r(x_2)v_1\\&\quad\quad-[\alpha(x_2),[x_1,x_3]]-l(\alpha(x_2))l(x_1)v_3-r(\alpha(x_3))l(x_1)v_2-r(\alpha(x_3))r(x_2)v_1.
\end{align*}
Then, by \eqref{Leibnizidentity} and \eqref{Cond:HomBimod:lpb}-\eqref{Cond:HomBimod:lar},
\begin{multline*}
[(\alpha+\beta)(x_1+v_1),[x_2+v_2,x_3+v_3]']'-
[[x_1+v_1,x_2+v_2]',(\alpha+\beta)(x_3+v_3)]'\\
-[(\alpha+\beta)(x_2+v_2),[x_1+v_1,x_3+v_3]']'=0,
\end{multline*}
which completes the proof.
\end{proof}
We denote such a Hom-Leibniz algebra $( A\oplus V, [\cdot,\cdot]', \alpha + \beta),$
 or $ A\times_{l, r, \alpha, \beta}V.$
\begin{ex}
If $(A,[\cdot,\cdot],\alpha)$ is a Hom-Leibniz algebra, then $(L,0,\alpha,A)$ and $(L,R,\alpha,A)$ are bimodules of $(A,[\cdot,\cdot],\alpha)$, where $L(a)b=[a,b]$ and $R(a)b=[b,a]$ for all $a,b\in A$.
\end{ex}

\begin{prop}
If $( A, [\cdot,\cdot], \alpha)$ is a multiplicative Hom-Leibniz algebra, then the quadruple $(\alpha\otimes L ,\alpha\otimes R, \alpha\otimes \alpha,  A\otimes A)$, where for any $x, a, b\in  A,$
\begin{align*}
(\alpha\otimes L)(x)(a\otimes b)  &= (\alpha\otimes L(x))(a\otimes b) = \alpha(a)\otimes L(x)b=\alpha(a)\otimes [x,b], \cr
(\alpha \otimes R )(x)(a\otimes b)  &= (\alpha\otimes R(x) )(a\otimes b) = \alpha(a)\otimes R(x)b=\alpha(a)\otimes [b,x],
\end{align*}
is a bimodule of $ A.$
\end{prop}
\begin{proof}
By formulas for the maps and Hom-Leibniz identity \eqref{Leibnizidentity},
for all $x, y, v_{1}, v_{2}\in  A$:
\begin{align*}
(\alpha\otimes L)(\alpha(x))(\alpha\otimes L)(y)(v_1\otimes v_2)
&=(\alpha\otimes L(\alpha(x))(\alpha(v_1))\otimes L(y)v_2)\\
&=v_1\otimes L(\alpha(x))L(y)v_2\\
&=v_1\otimes[\alpha(x),[y,v_2]],\\
(\alpha\otimes L)([x,y])(\alpha\otimes\alpha)(v_1\otimes v_2)
&=(\alpha\otimes L)([x,y](\alpha(v_1)\otimes\alpha(v_2))\\
&=v_1\otimes L([x,y])\alpha(v_2)\\
&=v_1\otimes[[x,y],\alpha(v_2)],\\
(\alpha\otimes L)(\alpha(y))(\alpha\otimes L)(x)(v_1\otimes v_2)
&=(\alpha\otimes L)(\alpha(y))(\alpha\otimes L(x))(v_1\otimes v_2)\\
&=(\alpha\otimes L)(\alpha(y))(\alpha(v_1)\otimes L(x)v_2)\\
&=v_1\otimes L(\alpha(y))L(x)v_2\\
&=v_1\otimes[\alpha(y),[x,v_2]], \\
(\alpha\otimes L)(\alpha(x))(\alpha\otimes L)(y)(v_1\otimes v_2)
&= v_1\otimes[\alpha(x),[y,v_2]] \\
&\hspace{-1,5mm} \stackrel{\eqref{Leibnizidentity}}{=} v_1\otimes([[x,y],\alpha(v_2)]+[\alpha(y),[x,v_2]]) \\
&= v_1\otimes[[x,y],\alpha(v_2)]+ v_1\otimes [\alpha(y),[x,v_2]] \\
=(\alpha\otimes L)([x,y])(\alpha\otimes\alpha)(v_1\otimes v_2)
&+(\alpha\otimes L)(\alpha(y))(\alpha\otimes L)(x)(v_1\otimes v_2);
\\*[2mm]
(\alpha\otimes L)(\alpha(x))(\alpha\otimes R)(y)(v_1\otimes v_2)
&=(\alpha\otimes L(\alpha(x)))(\alpha\otimes R(y))(v_1\otimes v_2)\\
&=(\alpha\otimes L(\alpha(x)))(\alpha(v_1)\otimes R(y)v_2)\\
&=v_1\otimes[\alpha(x),[v_2,y]],\\
(\alpha\otimes R)(\alpha(y))(\alpha\otimes L)(x)(v_1\otimes v_2)
&=(\alpha\otimes R(\alpha(y)))(\alpha\otimes L(x))(v_1\otimes v_2)\\
&=(\alpha\otimes R(\alpha(y)))(\alpha(v_1)\otimes L(x)v_2)\\
&=v_1\otimes[[x,v_2],\alpha(y)],\\
(\alpha\otimes R)([x,y])(\alpha\otimes\alpha)(v_1\otimes v_2)
&=(\alpha\otimes R([x,y]))(\alpha(v_1)\otimes\alpha(v_2))\\
&=v_1\otimes[\alpha(v_2),[x,y]],  \\
(\alpha\otimes L)(\alpha(x))(\alpha\otimes R)(y)(v_1\otimes v_2)
&=v_1\otimes [\alpha(x),[v_2,y]] \\
&\hspace{-1,5mm} \stackrel{\eqref{Leibnizidentity}}{=}
v_1\otimes ([[x,v_2],\alpha(y)]+[\alpha(v_2),[x,y]]) \\
&= v_1\otimes [[x,v_2],\alpha(y)]+ v_1\otimes [\alpha(v_2),[x,y]]\\
=(\alpha\otimes R)(\alpha(y))(\alpha\otimes L)(x)(v_1\otimes v_2)
&+ (\alpha\otimes R)([x,y])(\alpha\otimes\alpha)(v_1\otimes v_2);\\*[2mm]
(\alpha\otimes R)([x,y])(\alpha\otimes\alpha)(v_1\otimes v_2)&=
(\alpha\otimes R([x,y]))(\alpha(v_1)\otimes\alpha(v_2))\\
&=v_1\otimes[\alpha(v_2),[x,y]],\\
(\alpha\otimes R)(\alpha(y))(\alpha\otimes R)(x)(v_1\otimes v_2)
&=(\alpha\otimes R(\alpha(y)))(\alpha\otimes R(x))(v_1\otimes v_2)\\
&=\alpha\otimes R(\alpha(y))(\alpha\otimes R(x))(v_1\otimes v_2)\\
&=(\alpha\otimes R(\alpha(y))(\alpha(v_1)\otimes R(x)v_2))\\
&=v_1\otimes[[v_2,x],\alpha(y)],\\
(\alpha\otimes L)(\alpha(x))(\alpha\otimes R)(y)(v_1\otimes v_2)
&=(\alpha\otimes L(\alpha(x))(\alpha\otimes R(y))(v_1\otimes v_2)\\
&=(\alpha\otimes L(\alpha(x))(\alpha(v_1)\otimes R(y)v_2)\\
&=v_1\otimes[\alpha(x),[[v_2,y]],
\\
(\alpha\otimes R)([x,y])(\alpha\otimes\alpha)(v_1\otimes v_2) &=v_1\otimes[\alpha(v_2),[x,y]] \\
& \hspace{-1,5mm} \stackrel{\eqref{Leibnizidentity}}{=} v_1\otimes([[v_2,x],\alpha(y)]+[\alpha(x),[v_2,y]]) \\
&= v_1\otimes [[v_2,x],\alpha(y)] + v_1\otimes [\alpha(x),[v_2,y]] \\
= (\alpha\otimes R)(\alpha(y))(\alpha\otimes R)(x)(v_1\otimes v_2)
&+(\alpha\otimes L)(\alpha(x))(\alpha\otimes R)(y)(v_1\otimes v_2).
\end{align*}
By formulas for the maps,
\begin{eqnarray*}
(\alpha\otimes\alpha)(\alpha\otimes L)(x)(v_1\otimes v_2)&=&(\alpha\otimes\alpha)(\alpha\otimes L(x))(v_1\otimes v_2))\\
&=&(\alpha\otimes\alpha)(\alpha(v_1)\otimes L(x)v_2)\\
&=&v_1\otimes L(\alpha(x))\alpha(v_2),\\
(\alpha\otimes L)(\alpha(x))((\alpha\otimes\alpha)(v_1\otimes v_2))
&=&(\alpha\otimes L)(\alpha(x))(\alpha(v_1)\otimes\alpha(v_2))\\
&=&v_1\otimes L(\alpha(x))\alpha(v_2).
\end{eqnarray*}
Therefore,
$
(\alpha\otimes\alpha)(\alpha\otimes L)(x)(v_1\otimes v_2)= (\alpha\otimes L)(\alpha(x))((\alpha\otimes\alpha)(v_1\otimes v_2)).
$
Similarly,
$
(\alpha\otimes\alpha)(\alpha\otimes R)(x)(v_1\otimes v_2)= (\alpha\otimes R)(\alpha(x))((\alpha\otimes\alpha)(v_1\otimes v_2)).
$
\qedhere
\end{proof}
\begin{rmk}
The quadruple $(L\otimes\alpha,  R\otimes\alpha, \alpha\otimes \alpha,  A\otimes A)$ is also a bimodule of $ A.$
\end{rmk}
\begin{thm}[\cite{Attan2021:Repsrelrotabaxthomleibnizpoissonalgs}]
\label{thm:matchedpairs}
Let $( A, [\cdot,\cdot], \alpha)$ and $(B, \{\cdot,\cdot\}, \beta)$ be two
Hom-Leibniz algebras. Suppose there are linear maps $l_{  A}, r_{
 A}:  A\rightarrow gl(  B)$ and $l_{B}, r_{
B} :  B\rightarrow gl(  A)$ such that the quadruple
 $(l_{  A}, r_{
 A},  \beta, B)$ is
a bimodule of $ A,$ and $(l_{ B}, r_{B}, \alpha,
 A)$ is a bimodule of $B,$
satisfying, for any $ x, y \in   A, a,b \in  B$, the following {\it  conditions}:
\begin{eqnarray} \label{match. pair1}
l_{A}(\alpha(x))(\{a,b\})=l_A(r_B(a)x)\beta(b)&+&r_A(r_B(b)x)\beta(a)\nonumber\\&+&\{l_A(x)a,\beta(b)\}+\{\beta(a),l_A(x)b\},\\
\label{match. pair2}
\{\beta(a),l_A(x)b\}+r_A(r_B(b)x)\beta(a)&=&\{r_A(x)a,\beta(b)\}\nonumber\\&+&l_A(l_B(a)x)\beta(b)+l_A(\alpha(x))(\{a,b\}),
\\ \label{match. pair3}
\{\beta(a),r_A(x)b\}+r_A(l_B(b)x)\beta(a)&=&r_A(\alpha(x))(\{a,b\})\nonumber\\&+&\{\beta(b),r_A(x)a\}+r_A(l_B(a)x)\beta(b),
\\\label{match. pair4}
l_{B}(\beta(a))([x,y])=l_B(r_A(x)a)\alpha(y)&+&r_B(r_A(y)a)\alpha(x)\nonumber\\&+&[l_B(a)x,\alpha(y)]+[\alpha(x),l_B(a)y],
\\ \label{match. pair5}
[\alpha(x),l_B(a)y]+r_B(r_A(y)a)\alpha(x)&=&[r_B(a)x,\alpha(y)]\nonumber\\&+&l_B(l_A(x)a)\alpha(y)+l_B(\beta(a))([x,y]),
\\\label{match. pair6}
[\alpha(x),r_B(a)y]+r_B(l_A(y)a)\alpha(x)&=&r_B(\beta(a))([x,y])\nonumber\\&+&[\alpha(y),r_B(a)x]+r_B(l_A(x)a)\alpha(y).
\end{eqnarray}
Then, $(A,B,l_A,r_A,\beta,l_B,r_B,\alpha)$ is called a matched pair of
Hom-Leibniz algebras. In this case, there is a Hom-Leibniz algebra structure on the direct sum $ A\oplus B$ of
the underlying linear spaces of $ A$ and $B$ given for all $ x, y \in   A, a,b \in  B$ by
\begin{equation} \label{matchedpairproduct}
\begin{array}{rcl}
[x + a , y + b]' &=& ([x, y] + l_{ B}(a)y + r_{ B}(b)x) + \{a, b\} +  l_{  A}(x)b +  r_{  A}(y)a),\\
(\alpha\oplus\beta)(x + a)&=&\alpha(x) + \beta(a).
\end{array}
\end{equation}
\end{thm}
We denote this Hom-Leibniz algebra by
$( A\bowtie B, [\cdot,\cdot]', \alpha + \beta)$ or $   A \bowtie^{l_{  A}, r_{  A}, \beta}_{l_{ B}, r_{ B}, \alpha} B.$
\begin{rmk}
Condition \ref{match. pair6}, can be reformulated equivalently as
\begin{equation}
[l_B(a)x,\alpha(y)]+l_B(r_A(x)a)\alpha(y)+[r_B(a)x,\alpha(y)]+l_B(l_A(x)a)\alpha(y)=0.\end{equation}
\end{rmk}
\section{Manin triples and bialgebras of Hom-Leibniz algebras}
\label{sec:manintripplesbialgshomleibniz}
In this section, we consider the multiplicative Hom-Leibniz algebra $( A, [\cdot,\cdot], \alpha)$ such that $\alpha$ is involutive, i.e, $\alpha^{2}= id_{ A}$.
\begin{defn}
Let $V_{1} $, $ V_{2} $ be two linear spaces. For a linear map $ \phi : V_{1} \rightarrow V_{2} $, we denote by $ \phi^{\ast} : V^{\ast}_{2} \rightarrow V^{\ast}_{1} $ the dual (linear) map given by
\begin{eqnarray*}
\langle v, \phi^{\ast}(u^{\ast})\rangle = -\langle \phi(v), u^{\ast} \rangle \mbox{ for all } v \in V_{1} ,  u^{\ast} \in V^{\ast}_{2}.
\end{eqnarray*}
\end{defn}
\begin{lem} \label{lem:bimodmultinvolhomleib}
Let $(l, r, \beta, V)$ be a bimodule of a multiplicative involutive Hom-Leibniz algebra $( A, [\cdot,\cdot], \alpha)$. Then, for linear maps
$ l^{\ast}, r^{\ast} $ : $   A \rightarrow gl(V^{\ast}) $ such that
\begin{align*}
\forall \ x \in   A,  u^{\ast} \in V^{\ast}, v \in V:  \langle l^{\ast}(x)u^{\ast}, v \rangle = -\langle l(x)v, u^{\ast} \rangle,
\langle r^{\ast}(x)u^{\ast}, v \rangle = -\langle r(x)v, u^{\ast} \rangle,
\end{align*}
\begin{enumerate}[label=\upshape{(\roman*)},leftmargin=30pt]
\item \label{ilem:bimodmultinvolhomleib} $(l^{\ast}, r^{\ast}, \beta^{\ast}, V^{\ast})$ is a bimodule of $( A, [\cdot,\cdot], \alpha)$;
\item \label{iilem:bimodmultinvolhomleib} $(l^{\ast}, -l^{\ast}-r^{\ast}, \beta^{\ast}, V^{\ast})$ is a bimodule of $( A, [\cdot,\cdot], \alpha)$;
\item \label{iiilem:bimodmultinvolhomleib} $(l^{\ast}, 0, \beta^{\ast}, V^{\ast})$ and $(0, r^{\ast}, \beta^{\ast}, V^{\ast})$ are bimodules of $( A,[\cdot,\cdot], \alpha)$.
\end{enumerate}
\end{lem}
\begin{proof}
{\rm \ref{ilem:bimodmultinvolhomleib}} For any $x,y\in A,~u^{\ast}\in A^{\ast}$ and $v\in V$,
\begin{eqnarray*}
\langle l^{\ast}(\alpha(x))l^{\ast}(y)u^{\ast},v\rangle  &=& \langle  l(y)l(\alpha(x))v,u^{\ast}\rangle\cr
&=&\langle l(\alpha^{2}(y))l(\alpha(x))v,u^{\ast}\rangle \cr
\mbox{(by~\eqref{Cond:HomBimod:lpb})}
&=&\langle l(\alpha^{2}(x))l(\alpha(y))v-l([\alpha(x),\alpha(y)])\beta(v),u^{\ast} \rangle \cr
&=&\langle l(x)l(\alpha(y))v-\beta(l([x,y])v),u^{\ast} \rangle\cr
&=&\langle l^{\ast}(\alpha(y))l^{\ast}(x)u^{\ast}+l^{\ast}([x,y])\beta^{\ast}(u^{\ast}),v \rangle, \\*[1mm]
\langle l^{\ast}(\alpha(x))r^{\ast}(y)u^{\ast},v\rangle
&=& \langle  r(y)l(\alpha(x))v,u^{\ast}\rangle \cr
&=&-\langle r(x)l(\alpha(y))v-r([\alpha(x),\alpha(y)])\beta(v),u^{\ast} \rangle \cr
&=& \langle  r^{\ast}(\alpha(y))l^{\ast}(x)u^{\ast}+r^{\ast}([x,y])\beta^{\ast}(u^{\ast}),v\rangle,\\*[1mm]
\langle r^{\ast}([x,y])\beta^{\ast}(u^{\ast}),v\rangle
&=& -\langle r([\alpha(x),\alpha(y)])\beta(v)u^{\ast}\rangle \cr
&=&\langle r(y)l(\alpha(x))v-l(x)r(\alpha(y))v,u^{\ast} \rangle \cr
&=& \langle l^{\ast}(\alpha(x))r^{\ast}(y)u^{\ast}-r^{\ast}(\alpha(y))l^{\ast}(x)u^{\ast},v \rangle.
\end{eqnarray*}
{\rm \ref{iilem:bimodmultinvolhomleib}} For any $x,y\in A,~u^{\ast}\in A^{\ast}$ and $v\in V$,
\begin{eqnarray*}
 \langle l^{\ast}(\alpha(x))l^{\ast}(y)u^{\ast},v\rangle
 &=& \langle  l(y)l(\alpha(x))v,u^{\ast}\rangle \cr
 &=&\langle l(\alpha^{2}(y))l(\alpha(x))v,u^{\ast}\rangle \cr
 \mbox{(by \eqref{Cond:HomBimod:lpb})} &=&\langle l(\alpha^{2}(x))l(\alpha(y))v-l([\alpha(x),\alpha(y)])\beta(v),u^{\ast} \rangle \cr
 &=&\langle l(x)l(\alpha(y))v-\beta(l([x,y])v),u^{\ast} \rangle\cr
 &=&\langle l^{\ast}(\alpha(y))l^{\ast}(x)u^{\ast}+l^{\ast}([x,y])\beta^{\ast}(u^{\ast}),v \rangle,\\*[1mm]
 \langle l^{\ast}(\alpha(x))(-l^{\ast}-r^{\ast})(y)u^{\ast},v\rangle
 &=& -\langle  (l+r)(y)l(\alpha(x))v,u^{\ast}\rangle \cr
 &=&-\langle (l+r)(\alpha^{2}(y)l(\alpha(x))v,u^{\ast} \rangle \cr
\mbox{(by \eqref{Cond:HomBimod:lpb})}
&=&\langle l(\alpha^{2}(x))l(\alpha(y))v-l([\alpha(x),\alpha(y)])\beta(v),u^{\ast} \rangle \cr
&&-\langle l(\alpha^{2}(x))r(\alpha(y))v-r([\alpha(x),\alpha(y)])\beta(v),u^{\ast}\rangle\cr
&=&-\langle l(x)(l+r)(\alpha(y))v-(l+r)([\alpha(x),\alpha(y)])\beta(v),u^{\ast}  \rangle\cr
&=& -\langle
(l^{\ast}+r^{\ast})(\alpha(y))l^{\ast}(x)u^{\ast}+(l^{\ast}+r^{\ast})[x,y]\beta^{\ast}(u^{\ast}),v\rangle\cr
&=& \langle  (-l^{\ast}-r^{\ast})(\alpha(y))l^{\ast}(x)u^{\ast}+(-l^{\ast}-r^{\ast})[x,y]\beta^{\ast}(u^{\ast}),v\rangle\rangle.
\end{eqnarray*}
Similarly, one can show $$(-l^{\ast}-r^{\ast})([x,y])\beta^{\ast}(u^{\ast})=
(-l^{\ast}-r^{\ast})(\alpha(y))(-l^{\ast}-r^{\ast})(x)u^{\ast}+l(\alpha(x))(-l^{\ast}-r^{\ast})(y)u^{\ast}.$$
{\rm \ref{iiilem:bimodmultinvolhomleib}} It is proved analogously.
\end{proof}
\begin{defn}
Let $( A, [\cdot,\cdot], \alpha)$ be a Hom-Leibniz algebra, and $B:  A\times A\rightarrow \mathbb{K}$ be a bilinear form on $ A.$ Then,
\begin{enumerate}
\item $B$ is said  to be nondegenerate if
\begin{eqnarray}
 A^{\bot}=\left\lbrace x\in A \mid B(y, x)=0, \forall y\in A \right\rbrace =\{0_A\};
\end{eqnarray}
\item  $B$ is said to be skew-symmetric if for all $x,y \in A$,
\begin{eqnarray}
B(x, y)=-B(y, x);
\end{eqnarray}
\item  $B$ is said to be  $\alpha$-invariant if for all $x,y \in A$,
\begin{eqnarray}
B([\alpha(x),\alpha(z)]+[\alpha(z),\alpha(x)], \alpha(y))= B(\alpha(x), [\alpha(y),\alpha(z)]).
\end{eqnarray}
\end{enumerate}
\end{defn}
\begin{defn}
 A quadratic Hom-Leibniz algebra is a Hom-Leibniz algebra equipped with a non-degenerate skew-symmetric invariant bilinear form.
\end{defn}
Recall that a quadratic Hom-Lie algebra is a Hom-Lie algebra $(\mathfrak{g},[\cdot,\cdot],\alpha)$ equipped with a non-degenerate symmetric bilinear form $B$ which is invariant in the sense that
$$B([\alpha(x),\alpha(y)],\alpha(z))=B(\alpha(x),[\alpha(y),\alpha(z)]),~~~\forall x,y,z\in\mathfrak{g}.$$
\begin{defn}
Let $(A,[\cdot,\cdot],\alpha)$ be a Hom-Leibniz algebras. A subalgebra $H$ of $A$ is called isotropic if $H\subseteq H^{\perp}$.
\end{defn}
\begin{defn}
Manin triples for Hom-Leibniz algebras are defined as quadruples $(A, A_1, A_2,\alpha)$ where
\begin{enumerate}
    \item $(A, \cdot, \alpha,B)$
is a quadratic Hom-Leibniz algebra;
\item
both $(A_1, \cdot_1, \alpha_1)$ and $(A_2, \cdot_2, \alpha_2)$ are isotropic
Hom-Leibniz subalgebras of $(A,\alpha)$ with $\alpha=\alpha_1+\alpha_2$;
\item
$A=A_1\oplus A_2$ as linear spaces.
\end{enumerate}
\end{defn}
Two Manin triples $( A,A_1, A_2,\alpha)$ and $(  C,   C_1,   C_2,\beta)$ with the bilinear forms $B_1$ and $B_2$ respectively are isomorphic if there exists an isomorphism of Hom-Leibniz algebras $f : A \rightarrow   C$ such that for all $x, y \in A$,
\begin{equation}
f\circ \alpha=\beta\circ f,~~f(A_1) =   C_1,~~ f(A_2) =   C_2, ~~B_1(x, y) = B_2(f(x), f(y)).
\end{equation}
\begin{ex}
If $(A,[\cdot,\cdot],\alpha)$ is a Hom-Leibniz algebra, then $(A\times_{L^{\ast},-L^{\ast}-R^{\ast},\alpha,\alpha^{\ast}}A^{\ast},A,A^{\ast})$ is a Manin triple of Hom-Leibniz algebras, where the natural non-degenerate skew-symmetric bilinear form $B$ on $A\oplus A^{\ast}$ is given by
\begin{equation}\label{quadreq}
    B(x+a^{\ast},y+b^\ast)=\langle a^{\ast},y\rangle-\langle b^{\ast},x\rangle,~~\forall x,y\in A,~a^\ast,b^\ast\in A^\ast,
\end{equation}
which is called the standard Manin
triple.
\end{ex}
\begin{prop}
 Every Manin triple is isomorphic to the standard Manin triple.
\end{prop}
For a Hom-Leibniz algebra $(A^{\ast},\{\cdot,\cdot\},\alpha^{\ast})$, let $\Delta:A\rightarrow\otimes^{2}A$ be the dual map of $\{\cdot,\cdot\}:\otimes^{2}A^\ast\rightarrow A^{\ast}$, that is
$$\langle \Delta x,a^\ast\otimes b^\ast\rangle=\langle x,\{a^\ast,b^\ast\}\rangle.$$
\begin{defn}\label{defi:Leibnizbialgebra}
Let $(A,[\cdot,\cdot],\alpha)$ and $(A^\ast,\{\cdot,\cdot\},\alpha^\ast)$ be a Hom-Leibniz algebras. Then $(A,A^\ast)$ is called a Hom-Leibniz bialgebra if for all $x,y\in A$,
  \begin{align}
  \label{bialgebra1}
   &(R_{[,]}(x)\otimes id)\Delta\alpha(y)=\sigma((R_{[,]}(\alpha(y))\otimes id)\Delta x), \\ \label{bialgebra2}
  &(\alpha\otimes id)\Delta[x,y]=-(id\otimes R_{[,]}(\alpha(y)))\Delta x-(id\otimes L_{[,]}(\alpha(x)))\Delta y\nonumber\\
  &\quad +\Big(L_{[,]}(y)\otimes id+R_{[,]}(y)\otimes id)+\sigma(id\otimes L_{[,]}(y)+id\otimes R_{[,]}(y))\Big)\Delta\alpha(x)\nonumber\\
  &\quad -(L_{[,]}(x)\otimes id+R_{[,]}(x)\otimes id)\Delta\alpha(y),
  \end{align}
   where  $\sigma:A\otimes A\rightarrow A\otimes A$ is the exchange operator defined by
$
\sigma(x\otimes y)=y\otimes x.
$
\end{defn}
Now we are ready to give the main results of this section.
\begin{thm}\label{Frobenius theorem}
Let $( A,[\cdot,\cdot], \alpha) $ be an involutive Hom-Leibniz algebra. Suppose that there is an involutive Hom-Leibniz algebra structure $"\{\cdot,\cdot\}"$ on its
dual space $  A^{\ast} $. Then, there is a standard Manin triple of $( A, [\cdot,\cdot], \alpha)$
and $( A^{\ast}, \{\cdot,\cdot\}, \alpha^{\ast})$ if and only if $( A,  A^{\ast}, L^{\ast}_{[,]}, -L^{\ast}_{[,]}-R^{\ast}_{[,]}, \alpha^{\ast}, L^{\ast}_{\{,\}},-L^{\ast}_{\{,\}}-R^{\ast}_{\{,\}}, \alpha)$
is a matched pair of involutive Hom-Leibniz algebras.
\end{thm}
\begin{proof}
Let us consider the four maps such that for all $ x, v, u \in  A $, $ x^{\ast}, v^{\ast}, u^{\ast} \in  A^{\ast} $,
\begin{align*}
& L^{\ast}_{[,]}:  A \rightarrow gl( A^{\ast}), && \langle L^{\ast}_{[,]}(x)u^{\ast}, v \rangle = -\langle
L_{[,]}(x)v, u^{\ast} \rangle =- \langle [x, v], u^{\ast} \rangle,\cr
& R^{\ast}_{[,]} :  A \rightarrow gl( A^{\ast}), && \langle R^{\ast}_{[,]}(x)u^{\ast}, v \rangle = -\langle
R_{[,]}(x)v, u^{\ast} \rangle = -\langle [v, x], u^{\ast} \rangle,\cr
& L^{\ast}_{\{,\}} :  A^{\ast} \rightarrow gl( A), && \langle L^{\ast}_{\{,\}}(x^{\ast})u,
 v^{\ast} \rangle =- \langle L_{\{,\}}(x^{\ast})v^{\ast}, u \rangle = -\langle \{x^{\ast} , v^{\ast}\},
  u\rangle,\cr
& R^{\ast}_{\{,\}} :  A^{\ast} \rightarrow gl( A), && \langle R^{\ast}_{\{,\}}(x^{\ast})u,
v^{\ast} \rangle = -\langle R_{\{,\}}(x^{\ast})v^{\ast}, u \rangle = -\langle \{v^{\ast} , x^{\ast}\}, u\rangle.
\end{align*}
If $( A,  A^{\ast}, L^{\ast}_{[,]}, -L^{\ast}_{[,]}-R^{\ast}_{[,]}, \alpha^{\ast}, L^{\ast}_{\{,\}},  -L^{\ast}_{\{,\}}-R^{\ast}_{\{,\}}, \alpha)$ is a matched pair of multiplicative Hom-Leibniz algebras, then $( A\bowtie  A^{\ast}, [\cdot,\cdot]', \alpha + \alpha^{\ast})$ is a multiplicative Hom-Leibniz algebra with the product $ [\cdot,\cdot]' $  given by \eqref{matchedpairproduct}, and the bilinear form $ {B}(\cdot, \cdot) $ defined by \eqref{quadreq} is $(\alpha\oplus\alpha^{\ast})$-invariant, that is
\begin{multline*} {B}([\alpha(x) + \alpha^{\ast}(a^{\ast}), \alpha(z)+\alpha^{\ast}(c^{\ast})]'+[\alpha(z)+\alpha^{\ast}(c^{\ast}),
\alpha(x)+\alpha^{\ast}(a^{\ast})]',\alpha(y)+\alpha^{\ast}(b^{\ast})) \\
= {B}((\alpha(x) + \alpha^{\ast}(a^{\ast}), [\alpha(y) + \alpha^{\ast}(b^{\ast})),
(\alpha(z) + \alpha^{\ast}(c^{\ast}))]')
\end{multline*}
for all $ x, y \in  A^{\ast},  a^{\ast}, b^{\ast} \in  A^{\ast},$ and $$ [(x + a^{\ast}),
(y + b^{\ast})]' = ([x,y] + l_{{A^{\ast}}}(a^{\ast})y + r_{{A^{\ast}}}(b^{\ast})x ) + (\{a^{\ast}, b^{\ast}\} +
l_{ A}(x)b^{\ast} + r_{ A}(y)a^{\ast} ), $$ with $l_{ A} = L^{\ast}_{[,]}, r_{ A} = -L^{\ast}_{[,]}-R^{\ast}_{[,]}, l_{{A^{\ast}}} = L^{\ast}_{\{,\}}, r_{{A^{\ast}}} = -L^{\ast} _{\{,\}} - R^{\ast} _{\{,\}}$. Indeed,
\begin{eqnarray*}
&&{B}(\alpha(x) + \alpha^{\ast}(a^{\ast}),[\alpha(y)+\alpha^{\ast}(b^{\ast}),\alpha(z)+\alpha^{\ast}(c^{\ast})]')\\
&&\quad=B\Big(\alpha(x)+\alpha^{\ast}(a^{\ast}),[\alpha(y),\alpha(z)]+l_{A^{\ast}}(\alpha^{\ast}(b^{\ast}))\alpha(z)+r_{A^{\ast}}(\alpha^{\ast}(c^{\ast}))\alpha(y)\\
&& \quad\quad +\{\alpha^{\ast}(b^{\ast}),\alpha^{\ast}(c^{\ast})\}+l_A(\alpha(y))\alpha^{\ast}(c^{\ast})+r_A(\alpha(z))\alpha^{\ast}(b^{\ast})\Big)\\
&&\quad=\langle \alpha^{\ast}(a^{\ast}),[\alpha(y),\alpha(z)]\rangle-\langle\{\alpha^{\ast}(b^{\ast}),\alpha^{\ast}(a^{\ast})\},\alpha(z)\rangle+\langle \{\alpha^{\ast}(c^{\ast}),\alpha^{\ast}(a^{\ast})\},\alpha(y)\rangle\\
&&\quad\quad +\langle\{\alpha^{\ast}(a^{\ast}),\alpha^{\ast}(c^{\ast})\},\alpha(y)\rangle-\langle\{\alpha^{\ast}(b^{\ast}),\alpha^{\ast}(c^{\ast})\},\alpha(x)\rangle+\langle [\alpha(y),\alpha(x)],\alpha^{\ast}(c^{\ast})\rangle\\
&&\quad\quad
-\langle [\alpha(z),\alpha(x)],\alpha^{\ast}(b^{\ast})\rangle-\langle [\alpha(x),\alpha(z)],\alpha^{\ast}(b^{\ast})\rangle,
\\*[3mm]
&&{B}([\alpha(x) + \alpha^{\ast}(a^{\ast}), \alpha(z)+\alpha^{\ast}(c^{\ast})]'+[\alpha(z)+\alpha^{\ast}(c^{\ast}),\alpha(x)+\alpha^{\ast}(a^{\ast})]',\alpha(y)+\alpha^{\ast}(b^{\ast}))\\
&&\quad=B\Big([\alpha(x),\alpha(z)]+l_{A^{\ast}}(\alpha^{\ast}(a^{\ast}))\alpha(z)+r_{A^{\ast}}(\alpha^{\ast}(c^{\ast}))\alpha(y)+\{\alpha^{\ast}(a^{\ast}),\alpha^{\ast}(c^{\ast})\}\\
&&\quad\quad +l_A(\alpha(x))\alpha^{\ast}(c^{\ast})+r_A(\alpha(z))\alpha^{\ast}(a^{\ast})
+[\alpha(z),\alpha(x)]+l_{A^{\ast}}(\alpha^{\ast}(c^{\ast}))\alpha(x)\\
&&\quad\quad +r_{A^{\ast}}(\alpha^{\ast}(a^{\ast}))\alpha(z)+\{\alpha^{\ast}(c^{\ast}),\alpha^{\ast}(a^{\ast})\}+l_A(\alpha(z))\alpha^{\ast}(a^{\ast})+r_A(\alpha(y))\alpha^{\ast}(b^{\ast})\Big)\\
&&\quad=-\langle[\alpha(x),\alpha(z)],\alpha^{\ast}(b^{\ast})\rangle
+\langle\alpha(z),\{\alpha^{\ast}(a^{\ast}),\alpha^{\ast}(b^{\ast})\}\rangle-\langle\alpha(x),\{\alpha^{\ast}(c^{\ast}),\alpha^{\ast}(b^{\ast})\}\rangle\\
&&\quad\quad +\langle\alpha(y),\{\alpha^{\ast}(a^{\ast}),\alpha^{\ast}(c^{\ast})\}\rangle+\langle[\alpha(y),\alpha(x)],\alpha^{\ast}(c^{\ast})\rangle+\langle[\alpha(z),\alpha(y)],\alpha^{\ast}(a^{\ast})\rangle\\
&&\quad\quad +\langle[\alpha(y),\alpha(z)],\alpha^{\ast}(a^{\ast})\rangle-\langle[\alpha(z),\alpha(x)],\alpha^{\ast}(b^{\ast})\rangle+
\langle\alpha(x),\{\alpha^{\ast}(c^{\ast}),\alpha^{\ast}(b^{\ast})\}\rangle\\
&&\quad\quad-\langle\{\alpha^{\ast}(b^{\ast}),\alpha^{\ast}(c^{\ast})\},\alpha(x)\rangle-\langle\alpha(z),\{\alpha^{\ast}(a^{\ast}),\alpha^{\ast}(b^{\ast})\}\rangle-\langle\alpha(z),\{\alpha^{\ast}(b^{\ast}),\alpha^{\ast}(a^{\ast})\}\rangle\\
&&\quad\quad +\langle\{\alpha^{\ast}(c^{\ast}),\alpha^{\ast}(a^{\ast})\},\alpha(y)\rangle-\langle[\alpha(z),\alpha(y)],\alpha^{\ast}(a^{\ast})\rangle\\
&&\quad=\langle \alpha^{\ast}(a^{\ast}),[\alpha(y),\alpha(z)]\rangle-\langle\{\alpha^{\ast}(b^{\ast}),\alpha^{\ast}(a^{\ast})\},\alpha(z)\rangle+\langle \{\alpha^{\ast}(c^{\ast}),\alpha^{\ast}(a^{\ast})\},\alpha(y)\rangle\\
&&\quad\quad+\langle\{\alpha^{\ast}(a^{\ast}),\alpha^{\ast}(c^{\ast})\},\alpha(y)\rangle-\langle\{\alpha^{\ast}(b^{\ast}),\alpha^{\ast}(c^{\ast})\},\alpha(x)\rangle+\langle [\alpha(y),\alpha(x)],\alpha^{\ast}(c^{\ast})\rangle\\
&&\quad\quad -\langle [\alpha(z),\alpha(x)],\alpha^{\ast}(b^{\ast})\rangle-\langle [\alpha(x),\alpha(z)],\alpha^{\ast}(b^{\ast})\rangle.
\end{eqnarray*}
Thus,  $ {B} $ is well $(\alpha\oplus\alpha^{\ast})$-invariant.
 Conversely,
if $(A\oplus A^{\ast},A,A^{\ast})$ is a Manin triple of Hom-Leibniz algebras. For all $x\in A,~a^{\ast},b^{\ast}\in A^{\ast}$,
\begin{align*}
B(b^{\ast},r_{2}(a^{\ast})x)&=B(b^{\ast},[x,a^{\ast}]')\\
&=B(\{b^{\ast},a^{\ast}\}+\{a^{\ast},b^{\ast}\},x)\\&=\langle R_{\{,\}}(a^{\ast})b^{\ast}+L_{\{,\}}(a^{\ast})b^{\ast},x\rangle\\
&=-\langle b^{\ast},(L^{\ast}_{\{,\}}+R^{\ast}_{\{,\}})(x)\rangle,
\end{align*}
which implies that $r_{2}=-L^{\ast}_{\{,\}}-R^{\ast}_{\{,\}}$. Now,
\begin{align*}\langle L_{1}(x)a^{\ast},y\rangle &=-B(y,[x,a^{\ast}]')\\
&=-B([y,a^{\ast}]'+[a^{\ast},y]',x)\\
&=-B(a^{\ast},[x,y])\\
&=-\langle a^{\ast},L_{[]}(x)y\rangle\\
&=\langle y,L^{\ast}_{[,]}(x)a^{\ast}\rangle,\end{align*}
which implies that $L_1=L^{\ast}_{[,]}$.
Similarly, $L_{2}=L^{\ast}_{\{,\}}$ and $r_1=-L^{\ast}_{[,]}-R^{\ast}_{[,]}$.
Thus, $( A,  A^{\ast}, L^{\ast}_{[,]}, -L^{\ast}_{[,]}-R^{\ast}_{[,]}, \alpha^{\ast}, L^{\ast}_{\{,\}},-L^{\ast}_{\{,\}}-R^{\ast}_{\{,\}}, \alpha)$ is a matched pair.
\end{proof}
\begin{thm}\label{homMathched pair's theorem}
Let $( A,[\cdot,\cdot], \alpha) $ be an involutive Hom-Leibniz algebra. Suppose that there is an involutive Hom-Leibniz algebra structure $"\{\cdot,\cdot\}"$ on its
dual space $  (A^{\ast},\alpha^{\ast}) $. Then, $( A,  A^{\ast}, L^{\ast}_{[,]}, -L^{\ast}_{[,]}-R^{\ast}_{[,]}, \alpha^{\ast}, L^{\ast}_{\{,\}},-L^{\ast}_{\{,\}}-R^{\ast}_{\{,\}}, \alpha)$
is a matched pair of involutive Hom-Leibniz algebras
  if and only if, for any $ x \in A$ and $ a^{\ast}, b^{\ast} \in A^{\ast},$
\begin{align}\label{infinitesimal cond.}
   & L^{\ast}_{\{,\}}(R^{\ast}_{[,]}(x)a^{\ast})\alpha(y)+[R^{\ast}_{\{,\}}a^{\ast})x,\alpha(y)]=0,\\
\label{antisymmetric cond.}
  &L^{\ast}_{\{,\}}(\alpha^{\ast}(a^{\ast}))([x,y])-[L^{\ast}_{\{,\}}(a^{\ast})x,\alpha(y)]-[\alpha(x),L^{\ast}_{\{,\}}a^{\ast})y]\nonumber\\&+L^{\ast}_{\{,\}}(L^{\ast}_{[,]}(x)a^{\ast})\alpha(y)+L^{\ast}_{\{,\}}(R^{\ast}_{[,]}(x) a^{\ast})\alpha(y)-L^{\ast}_{\{,\}}(L^{\ast}_{[,]}(y)a^{\ast})\alpha(x)\nonumber\\&-L^{\ast}_{\{,\}}(R^{\ast}_{[,]}(y)a^{\ast})\alpha(x)-R^{\ast}_{\{,\}}(L^{\ast}_{[,]}(x)(a^{\ast})\alpha(x)-R^{\ast}_{\{,\}}(R^{\ast}_{[,]}(x)a^{\ast})\alpha(x)=0.
\end{align}
\end{thm}
\begin{proof}
Obviously, \eqref{infinitesimal cond.} gives \eqref{match. pair5}, and
 \eqref{antisymmetric cond.} reduces to \eqref{match. pair4}  when
 $$ l_{A} = L^{\ast}_{[,]},\quad r_{A} = -L^{\ast}_{[,]}-R^{\ast}_{[,]},\quad
 l_{{B}} = L^{\ast}_{\{,\}},\quad
 r_{{B}} =-L^{\ast}_{\{,\}}-R^{\ast}_{\{,\}}.$$
 Now, we show that
\begin{eqnarray*}
\eqref{match. pair2} \Longleftrightarrow \eqref{match. pair5};\cr
{\mbox{and}}\quad
\eqref{match. pair1} \Longleftrightarrow  \eqref{match. pair3} \Longleftrightarrow \eqref{match. pair4}  \Longleftrightarrow \eqref{match. pair6}.
\end{eqnarray*}
Suppose  \eqref{infinitesimal cond.}  and \eqref{antisymmetric cond.} are satisfied. We have
for all $ x, y \in A, a^{\ast}, b^{\ast} \in A^{\ast}$,
\begin{eqnarray*}
\langle R^{\ast}_{[,]}(x)a^{\ast}, y \rangle &=& \langle L^{\ast}_{[,]}(y)a^{\ast}, x \rangle = -\langle   [ y, x], a^{\ast}\rangle, \cr
\langle R^{\ast}_{\{,\}}(b^{\ast})x, a^{\ast} \rangle &=& \langle L^{\ast}_{\{,\}}(a^{\ast})x, b^{\ast} \rangle =-\{\langle a^{\ast}, b^{\ast}\}, x \rangle,
\cr
\alpha^{\ast}(R^{\ast}_{[,]}(x)a^{\ast}) &=& R^{\ast}_{[,]}(\alpha(x))\alpha^{\ast}(a^{\ast}),\quad \alpha^{\ast}(L^{\ast}_{[,]}(x)a^{\ast})= L^{\ast}_{[,]}(\alpha(x))\alpha^{\ast}(a^{\ast}),
\cr
\alpha(R^{\ast}_{\{,\}}(a^{\ast})x) &=& R^{\ast}_{\{,\}}(\alpha^{\ast}(a^{\ast}))\alpha(x),\quad  \alpha(L^{\ast}_{\{,\}}(a^{\ast})x)= L^{\ast}_{\{,\}}(\alpha^{\ast}(a^{\ast}))\alpha(x),
\end{eqnarray*}
Set $\alpha(x)= z,\alpha(y)= t,$ $\alpha^{\ast}(a^{\ast})= c^{\ast}$ and $ \alpha^{\ast}(b^{\ast})= d^{\ast}.$
Then \eqref{match. pair2} $ \Longleftrightarrow $ \eqref{match. pair5} follows from
 \begin{eqnarray*}
 \langle L^{\ast}_{\{,\}}(R^{\ast}_{[,]}(x)a^{\ast})\alpha(y),b^{\ast}\rangle
&=& -\langle \{(R^{\ast}_{[,]}(x)a^{\ast}),b^{\ast}\},\alpha(y)\rangle\cr
&=&\langle \alpha^{\ast}\{R^{\ast}_{[,]}(x)a^{\ast}),b^{\ast}\},y\rangle\cr
&=&\langle \{R^{\ast}_{[,]}(\alpha(x))\alpha^{\ast}(a^{\ast}),\alpha^{\ast}(b^{\ast})\},y\rangle\cr
&=&\langle \{R^{\ast}_{[,]}(z)c^{\ast},\alpha^{\ast}(b^{\ast}\}),y\rangle;\cr
\langle[R^{\ast}_{\{,\}}(a^{\ast})x,\alpha(y)],b^{\ast}\rangle
&=&- \langle\alpha(y),L^{\ast}_{[,]}(R^{\ast}_{\{,\}}(a^{\ast})x)b^{\ast} \rangle
 \cr&=&\langle y,\alpha^{\ast}(L^{\ast}_{[,]}(R^{\ast}_{\{,\}}(\alpha^{\ast}(a^{\ast}))\alpha(x))b^{\ast})\rangle \cr
&=& \langle y,L^{\ast}_{[,]}(R^{\ast}_{\{,\}}(c^{\ast})z)\alpha^{\ast}(b^{\ast})\rangle.
\end{eqnarray*}
The statement \eqref{match. pair3} $ \Longleftrightarrow $ \eqref{match. pair6} follows from
\begin{eqnarray*}
\langle \{\alpha^{\ast}(a^{\ast}),(L^{\ast}_{[,]}(x)b^{\ast})\}, y\rangle
&=&- \langle \alpha^{\ast}(a^{\ast}), R^{\ast}_{\{,\}}L^{\ast}_{[,]}(x)b^{\ast}y\rangle  \cr&=&\langle a^{\ast}, \alpha[R^{\ast}_{\{,\}}(L^{\ast}_{[,]}(x)b^{\ast})y]\rangle\cr
&=&\langle a^{\ast}, R^{\ast}_{\{,\}}[\alpha^{\ast}(L^{\ast}_{[,]}(x)b^{\ast})]\alpha(y)\rangle\cr
&=&\langle a^{\ast}, R^{\ast}_{\{,\}}(L^{\ast}_{[,]}(\alpha(x))\alpha^{\ast}(b^{\ast}))\alpha(y)\rangle\cr
&=&\langle a^{\ast}, R^{\ast}_{\{,\}}(L^{\ast}_{[,]}(z)d^{\ast})\alpha(y)\rangle; \cr
\langle \{\alpha^{\ast}(a^{\ast}),(R^{\ast}_{[,]}(x)b^{\ast})\}, y\rangle
&=& -\langle \alpha^{\ast}(a^{\ast}), R^{\ast}_{\{,\}}(R^{\ast}_{[,]}(x)b^{\ast})y\rangle  \cr&=&\langle a^{\ast}, \alpha[R^{\ast}_{\{,\}}(R^{\ast}_{[,]}(\alpha(x))\alpha(b^{\ast})\alpha(y)]\rangle\cr
&=&\langle a^{\ast}, R^{\ast}_{\{,\}}(R^{\ast}_{[,]}(\alpha(x)\alpha^{\ast}(b^{\ast}))\alpha(y)\rangle\cr
&=&\langle a^{\ast}, R^{\ast}_{\{,\}}(R^{\ast}_{[,]}(z)d^{\ast})\alpha(y)\rangle; \cr
\langle L^{\ast}_{[,]}(L^{\ast}_{\{,\}}(b^{\ast})x)\alpha^{\ast}(a^{\ast}), y\rangle &=&- \langle [(L^{\ast}_{\{,\}}(b^{\ast})x), y], \alpha^{\ast}(a^{\ast})\rangle  \cr&=&\langle a^{\ast}, [\alpha(L^{\ast}_{\{,\}}(b^{\ast})x), \alpha(y)] \rangle \cr
&=& \langle a^{\ast}, L^{\ast}_{\{,\}}([\alpha^{\ast}(b^{\ast}))\alpha(x), \alpha(y)]\rangle  \cr&=&\langle a^{\ast}, [L^{\ast}_{\{,\}}(d^{\ast})z, \alpha(y)]\rangle;\cr
\langle R^{\ast}_{[,]}(L^{\ast}_{\{,\}}(a^{\ast})x)\alpha^{\ast}(b^{\ast}), y\rangle &=&-\langle [y, L^{\ast}_{\{,\}}(a^{\ast})x], \alpha^{\ast}(b^{\ast})\rangle  \cr&=&\langle [\alpha(y), \alpha(L^{\ast}_{\{,\}}(a^{\ast})x)], b^{\ast}\rangle \cr
&=&\langle [\alpha(y), L^{\ast}_{\{,\}}(\alpha^{\ast}(a^{\ast}))\alpha(x)], b^{\ast}\rangle  \cr&=&\langle [\alpha(y), L^{\ast}_{\{,\}}(c^{\ast})z], b^{\ast}\rangle;\cr
\langle\{ L^{\ast}(\alpha(x))(a^{\ast} , b^{\ast})\}, y\rangle
&=&- \langle \{a^{\ast}, b^{\ast}\}, [\alpha(x), y] \rangle
 \cr&=&\langle R^{\ast}_{\{,\}}(b^{\ast})([\alpha(x), y]), a^{\ast}\rangle \cr
&=& \langle R^{\ast}_{\{,\}}(\alpha^{\ast}(d^{\ast}))([z, y]), a^{\ast}\rangle; \cr
\langle \{R^{\ast}(\alpha(x))(a^{\ast}, b^{\ast})\}, y \rangle &=&- \langle \{a^{\ast}, b^{\ast}\}, [y, \alpha(x)]\rangle  \cr&=&-\langle L^{\ast}_{[,]}(a^{\ast})b^{\ast}, [y, z] \rangle \cr
 &=& \langle b^{\ast},L^{\ast}_{\{,\}}(a^{\ast})([y, z])\rangle \cr
&=&\langle b^{\ast}, L^{\ast}_{\{,\}}(\alpha^{\ast}(c^{\ast}))([y, z])\rangle;
\cr
\langle \{\alpha^{\ast}(b^{\ast}),(L^{\ast}_{[,]}(x)a^{\ast})\}, y\rangle
&=&- \langle \alpha^{\ast}(b^{\ast}), R^{\ast}_{\{,\}}(L^{\ast}_{[,]}(x)a^{\ast})y\rangle  \cr&=&\langle b^{\ast}, \alpha[R^{\ast}_{\{,\}}(L^{\ast}_{[,]}(x)a^{\ast})y]\rangle\cr
&=&\langle b^{\ast}, R^{\ast}_{\{,\}}[\alpha^{\ast}(L^{\ast}_{[,]}(x)a^{\ast})]\alpha(y)\rangle\cr
&=&\langle b^{\ast}, R^{\ast}_{\{,\}}[L^{\ast}_{[,]}(\alpha(x))\alpha^{\ast}(a^{\ast})]\alpha(y)\rangle\cr
&=&\langle b^{\ast}, R^{\ast}_{\{,\}}(L^{\ast}_{[,]}(z)d^{\ast})\alpha(y)\rangle; \cr
\langle \{\alpha^{\ast}(b^{\ast}), (R^{\ast}_{[,]}(x)a^{\ast})\}, y\rangle &=&- \langle R^{\ast}_{\{,\}}(R^{\ast}_{\{,\}}(x)a^{\ast})y, \alpha^{\ast}(b^{\ast})\rangle  \cr&=& -\langle \{\alpha^{\ast}(b^{\ast}), (R^{\ast}_{[,]}(x)a^{\ast})\}, y\rangle\cr
&=& \langle \alpha[R^{\ast}_{\{,\}}(R^{\ast}_{\{,\}}(x)a^{\ast})y], b^{\ast}\rangle\cr
 &=&\langle R^{\ast}_{\{,\}}[R^{\ast}_{\{,\}}(\alpha(x))\alpha^{\ast}(a^{\ast})]\alpha(y), b^{\ast}\rangle\cr
&=&\langle R^{\ast}_{\{,\}}(R^{\ast}_{[,]}(z)c^{\ast})\alpha(y), b^{\ast}\rangle; \cr
\langle L^{\ast}_{[,]}(L^{\ast}_{\{,\}}(a^{\ast})x)\alpha^{\ast}(b^{\ast}), y\rangle &=&- \langle [(L^{\ast}_{\{,\}}(a^{\ast})x), y], \alpha^{\ast}(b^{\ast})\rangle  \cr&=&\langle b^{\ast}, [\alpha(L^{\ast}_{\{,\}}(a^{\ast})x), \alpha(y)] \rangle \cr
&=& \langle b^{\ast}, [L^{\ast}_{\{,\}}(\alpha^{\ast}(a^{\ast}))\alpha(x), \alpha(y)]\rangle  \cr&=&\langle b^{\ast}, [L^{\ast}_{\{,\}}(c^{\ast})z, \alpha(y)]\rangle;\cr
\langle R^{\ast}_{[,]}(L^{\ast}_{\{,\}}(b^{\ast})x)\alpha(a^{\ast}), y \rangle &=& -\langle [y, L^{\ast}_{\{,\}}(b^{\ast})x], \alpha^{\ast}(a^{\ast}) \rangle  \cr&=& \langle L^{\ast}_{[,]}(y)(\alpha^{\ast}(a^{\ast})), L^{\ast}_{\{,\}}(b^{\ast})x \rangle \cr
&=& \langle L^{\ast}_{\{,\}}(b^{\ast})x, L^{\ast}_{[,]}(y)(\alpha^{\ast}(a^{\ast})) \rangle\cr
&=&- \langle \{b^{\ast},(L^{\ast}_{[,]}(y)(\alpha^{\ast}(a^{\ast})))\}, x\rangle \cr
  &=& \langle \{\alpha^{\ast}(c^{\ast}),(L^{\ast}_{[,]}(y)(d^{\ast}))\}, x\rangle.
\end{eqnarray*}
The statements
$\eqref{match. pair1} \Longleftrightarrow  \eqref{match. pair4}$ and $\eqref{match. pair4} \Longleftrightarrow  \eqref{match. pair6}$ are proved analogously.
\end{proof}

\begin{thm}\label{bialgebra theorem}
Let $( A,[\cdot,\cdot], \alpha) $ be an involutive Hom-Leibniz algebra. Suppose that there is an involutive Hom-Leibniz algebra structure $\{\cdot,\cdot\}$ on its
dual space $  (A^{\ast},\alpha^{\ast}) $. Then, $(A,A^{\ast})$ is a Hom-Leibniz bialgebra
if and only if $( A,  A^{\ast}, L^{\ast}_{[,]}, -L^{\ast}_{[,]}-R^{\ast}_{[,]}, \alpha^{\ast}, L^{\ast}_{\{,\}},\\-L^{\ast}_{\{,\}}-R^{\ast}_{\{,\}}, \alpha)$
is a matched pair of involutive Hom-Leibniz algebras.
\end{thm}
\begin{proof}
Let $x,y\in A$ and $a^{\ast},b^{\ast}\in A^{\ast}$.
First, the left hand side of  \eqref{infinitesimal cond.}, is equal to $$L^{\ast}_{\{,\}}(R^{\ast}_{[,]}(x)a^{\ast})\alpha(y)+[R^{\ast}_{\{,\}}(a^{\ast})x,\alpha(y)].$$ Furthermore, we have
\begin{align*}
    &\langle L^{\ast}_{\{,\}}(R^{\ast}_{[,]}(x)a^{\ast})\alpha(y)+[R^{\ast}_{\{,\}}a^{\ast})x,\alpha(y)],b^{\ast}\rangle\\
    &\quad=\langle L^{\ast}_{\{,\}}(R^{\ast}_{[,]}(x)a^{\ast})\alpha(y),b^{\ast}\rangle+\langle[R^{\ast}_{\{,\}}(a^{\ast})x,\alpha(y)],b^{\ast}\rangle\\
        &\quad=\langle L^{\ast}_{\{,\}}(R^{\ast}_{[,]}(x)(a^{\ast})\alpha(y),b^{\ast}\rangle+\langle R_{[,]}(\alpha(y))(R^{\ast}_{\{,\}}(a^{\ast})x),b^{\ast}\rangle\\
    &\quad=-\langle\alpha(y),\{R^{\ast}_{[,]}(x)a^{\ast},b^{\ast}\}\rangle+\langle x,\{R^{\ast}_{[,]}(\alpha(y))b^{\ast},a^{\ast}\}\rangle\\
    &\quad=-\langle\Delta\alpha(y),R^{\ast}_{[,]}(x)a^{\ast}\otimes b^{\ast}\rangle+\langle \Delta x,R^{\ast}_{[,]}(\alpha(y))b^{\ast}\otimes a^{\ast}\rangle\\
    &\quad=\langle(R_{[,]}(x)\otimes id)\Delta\alpha(y),a^{\ast}\otimes b^{\ast}\rangle-\langle(R_{[,]}(\alpha(y))\otimes id)\Delta x,b^{\ast}\otimes a^{\ast}\rangle\\
    &\quad=\langle(R_{[,]}(x)\otimes id)\Delta\alpha(y)-\sigma((R_{[,]}(\alpha(y))\otimes id)\Delta x),a^{\ast}\otimes b^{\ast}\rangle.
\end{align*}
Thus, we conclude that \eqref{infinitesimal cond.} corresponds to \eqref{bialgebra1}.

Next, the left hand side of \eqref{antisymmetric cond.} is equal to
   \begin{align*}
  &L^{\ast}_{\{,\}}(\alpha^{\ast}(a^{\ast}))([x,y])-[L^{\ast}_{\{,\}}(a^{\ast})x,\alpha(y)]-[\alpha(x),L^{\ast}_{\{,\}}a^{\ast})y]\nonumber\\&+L^{\ast}_{\{,\}}(L^{\ast}_{[,]}(x)a^{\ast})\alpha(y)+L^{\ast}_{\{,\}}(R^{\ast}_{[,]}(x) a^{\ast})\alpha(y)-L^{\ast}_{\{,\}}(L^{\ast}_{[,]}(y)a^{\ast})\alpha(x)\nonumber\\
  &-L^{\ast}_{\{,\}}(R^{\ast}_{[,]}(y)a^{\ast})\alpha(x)-R^{\ast}_{\{,\}}(L^{\ast}_{[,]}(x)a^{\ast})\alpha(x)-R^{\ast}_{\{,\}}(R^{\ast}_{[,]}(x)a^{\ast})\alpha(x).\end{align*} Furthermore, we have
\begin{align*}
  &\langle L^{\ast}_{\{,\}}(\alpha^{\ast}(a^{\ast}))([x,y])-[L^{\ast}_{\{,\}}(a^{\ast})x,\alpha(y)]-[\alpha(x),L^{\ast}_{\{,\}}a^{\ast})y]\\
  &\quad\quad+L^{\ast}_{\{,\}}(L^{\ast}_{[,]}(x)a^{\ast})\alpha(y)+L^{\ast}_{\{,\}}(R^{\ast}_{[,]}(x) a^{\ast})\alpha(y)-L^{\ast}_{\{,\}}(L^{\ast}_{[,]}(y)a^{\ast})\alpha(x)\\
  &\quad\quad-L^{\ast}_{\{,\}}(R^{\ast}_{[,]}(y)a^{\ast})\alpha(x)-R^{\ast}_{\{,\}}(L^{\ast}_{[,]}(x)a^{\ast})\alpha(x)-R^{\ast}_{\{,\}}(R^{\ast}_{[,]}(x)a^{\ast})\alpha(x),b^{\ast}\rangle\\
  &\quad=\langle L^{\ast}_{\{,\}}(\alpha^{\ast}(a^{\ast}))([x,y]),b^{\ast}\rangle-\langle[L^{\ast}_{\{,\}}(a^{\ast})x,\alpha(y)],b^{\ast}\rangle-\langle[\alpha(x),L^{\ast}_{\{,\}}a^{\ast})y],b^{\ast}\rangle\\
  &\quad\quad+\langle L^{\ast}_{\{,\}}(L^{\ast}_{[,]}(x)a^{\ast})\alpha(y),b^{\ast}\rangle+\langle L^{\ast}_{\{,\}}(R^{\ast}_{[,]}(x) a^{\ast})\alpha(y),b^{\ast}\rangle-\langle L^{\ast}_{\{,\}}(L^{\ast}_{[,]}(y)a^{\ast})\alpha(x),b^{\ast}\rangle\\
  &\quad\quad-\langle L^{\ast}_{\{,\}}(R^{\ast}_{[,]}(y)a^{\ast})\alpha(x),b^{\ast}\rangle-R^{\ast}_{\{,\}}(L^{\ast}_{[,]}(x)a^{\ast})\alpha(x),b^{\ast}\rangle-\langle R^{\ast}_{\{,\}}(R^{\ast}_{[,]}(x)a^{\ast})\alpha(x),b^{\ast}\rangle\\
  &\quad=-\langle[x,y],\{\alpha^{\ast}(a^{\ast}),b^{\ast}\}\rangle-\langle x,\{a^{\ast},R^{\ast}_{\{,\}}(\alpha(y))b^{\ast}\}\rangle-\langle y,\{a^{\ast},L^{\ast}_{[,]}(\alpha(x)) b^{\ast}\}\rangle\\
  &\quad\quad-\langle\alpha(y),\{L^{\ast}_{[,]}(x)a^{\ast},b^{\ast}\}\rangle-\langle\alpha(y),\{R^{\ast}_{[,]}(x)a^{\ast},b^{\ast}\}\rangle+\langle\alpha(x),\{ L^{\ast}_{[,]}(y)a^{\ast},b^{\ast}\}\\
  &\quad\quad+\langle\alpha(x),\{R^{\ast}_{[,]}(y)a^{\ast},b^{\ast}\}\rangle+\langle\alpha(x),\{a^{\ast},L^{\ast}(y)a^{\ast}\}\rangle+\langle\alpha(x),\{b^{\ast},R^{\ast}_{[,]}(y)a^{\ast}\}\rangle\\
  &\quad=-\langle\Delta[x,y],\alpha^{\ast}(a^{\ast})\otimes b^{\ast}\rangle-\langle\Delta x,a^{\ast}\otimes R^{\ast}_{[,]}(\alpha(y))b^{\ast}\rangle-\langle\Delta y,a^{\ast}\otimes L^{\ast}_{[,]}(\alpha(x))b^{\ast}\rangle\\&\quad\quad-\langle\Delta\alpha y,L^{\ast}_{[,]}(x)a^{\ast}\otimes b^{\ast}\rangle-\langle\Delta\alpha(y),R^{\ast}_{[,]}(x)a^{\ast}\otimes b^{\ast}\rangle+\langle\Delta\alpha(x),L^{\ast}_{[,]}(y)a^{\ast}\otimes b^{\ast}\rangle\\
  &\quad\quad+\langle\Delta\alpha(x),R^{\ast}_{[,]}(y)a^{\ast}\otimes b^{\ast}\rangle+\langle\Delta\alpha(x),a^{\ast}\otimes L^{\ast}_{[,]}(y)a^{\ast}\rangle+\langle\Delta\alpha(x),b^{\ast}\otimes R^{\ast}_{[,]}(y)a^{\ast}\rangle\\
  &\quad=\langle(\alpha\otimes id)\Delta[x,y],a^{\ast}\otimes b^{\ast}\rangle+\langle(id\otimes R_{[,]}(\alpha(y)))\Delta x,a^{\ast}\otimes b^{\ast}\rangle\\
  &\quad\quad+\langle(id\otimes L_{[,]}(\alpha(x)))\Delta y,a^{\ast}\otimes b^{\ast}\rangle+\langle (L_{[,]}(x)\otimes id)\Delta\alpha(y),a^{\ast}\otimes b^{\ast}\rangle\\
  &\quad\quad+\langle(R_{[,]}(x)\otimes id)\Delta\alpha(y),a^{\ast}\otimes b^{\ast}\rangle-\langle(L_{[,]}(y)\otimes id)\Delta\alpha(x),a^{\ast}\otimes b^{\ast}\rangle\\
  &\quad\quad-\langle(R_{[,]}(y)\otimes id)\Delta\alpha(x),a^{\ast}\otimes b^{\ast}\rangle-\langle(id\otimes L_{[,]}(y))\Delta\alpha(x),b^{\ast}\otimes a^{\ast}\rangle\\
  &\quad\quad-\langle(id\otimes R_{[,]}(y))\Delta\alpha(x),b^{\ast}\otimes a^{\ast}\rangle\\
   &\quad=\langle(\alpha\otimes id)\Delta[x,y],a^{\ast}\otimes b^{\ast}\rangle+\langle(id\otimes R_{[,]}(\alpha(y)))\Delta x,a^{\ast}\otimes b^{\ast}\rangle\\
  &\quad\quad+\langle(id\otimes L_{[,]}(\alpha(x)))\Delta y,a^{\ast}\otimes b^{\ast}\rangle+\langle (L_{[,]}(x)\otimes id)\Delta\alpha(y),a^{\ast}\otimes b^{\ast}\rangle\\
  &\quad\quad+\langle(R_{[,]}(x)\otimes id)\Delta\alpha(y),a^{\ast}\otimes b^{\ast}\rangle-\langle(L_{[,]}(y)\otimes id)\Delta\alpha(x),a^{\ast}\otimes b^{\ast}\rangle\\
  &\quad\quad-\langle(R_{[,]}(y)\otimes id)\Delta\alpha(x),a^{\ast}\otimes b^{\ast}\rangle-\langle\sigma((id\otimes L_{[,]}(y))\Delta\alpha(x)),a^{\ast}\otimes b^{\ast}\rangle\\
  &\quad\quad-\langle\sigma((id\otimes R_{[,]}(y))\Delta\alpha(x)),a^{\ast}\otimes b^{\ast}\rangle\\
   &\quad=\langle(\alpha\otimes id)\Delta[x,y]+(id\otimes R_{[,]}(\alpha(y)))\Delta x+(id\otimes L_{[,]}(\alpha(x)))\Delta y\\
   &\quad\quad+ (L_{[,]}(x)\otimes id)\Delta\alpha(y)+(R_{[,]}(x)\otimes id)\Delta\alpha(y) \\ &\quad\quad-(L_{[,]}(y)\otimes id)\Delta\alpha(x)-(R_{[,]}(y)\otimes id)\Delta\alpha(x) \\ &\quad\quad-\sigma((id\otimes L_{[,]}(y))\Delta\alpha(x))-\sigma((id\otimes R_{[,]}(y))\Delta\alpha(x)),a^{\ast}\otimes b^{\ast}\rangle.
\end{align*}
Thus, we conclude that \eqref{antisymmetric cond.} corresponds to \eqref{bialgebra2}. This completes the proof.
\end{proof}
\section{Hom-Leibniz dendriform algebras}
\label{sec:homleibnizdendrifalgs}
\begin{defn} \label{def:homleibnizdendriform}
Hom-Leibniz dendriform algebras are quadruples $( A, \prec, \succ, \alpha)$ consisting of a linear space $ A$, bilinear maps $\prec, \succ:  A\times  A\rightarrow  A,$ and linear maps $\alpha:  A\rightarrow  A$ satisfying
\begin{eqnarray}
([x , y])\succ\alpha(z) &=& \alpha(x)\succ (y \succ z)-\alpha(y)\succ(x\succ z),\label{Hom-Leibniz dendriform1}\\
\alpha(x)\succ(y\prec z)&=&(x\succ y) \prec\alpha(z)+\alpha(y)\prec ([x,z]),\label{Hom-Leibniz dendriform2}\\
\alpha(x)\prec([y,z])&=&(x\prec y) \prec\alpha(z)+\alpha(y)\succ (x \prec z),\label{Hom-Leibniz dendriform3}
\end{eqnarray}
where
\begin{eqnarray}\label{associative-dendriform}
[x,y] = x \prec y + x \succ y.
\end{eqnarray}
A Hom-Leibniz dendriform algebra is called multiplicative if
the linear map $\alpha:A\rightarrow A$ is multiplicative with respect to $\prec$ and $\succ$, that is,
for all $x,y \in A$,
$$
\alpha (x \prec y) =\alpha (x) \prec \alpha (y), \quad \alpha (x \succ y) =\alpha (x) \succ \alpha (y).
$$
\end{defn}

\begin{defn}
Let $( A, \prec, \succ, \alpha)$ and $( A', \prec',  \succ', \alpha')$ be two Hom-Leibniz dendriform algebras. A linear map $f:  A\rightarrow  A'$ is a Hom-Leibniz dendriform algebra morphism if
\begin{eqnarray*}
\prec'\circ(f\otimes f)= f\circ\prec,\quad  \succ'\circ(f\otimes f)= f\circ\succ \mbox{ and } f\circ\alpha= \alpha'\circ f.
\end{eqnarray*}
\end{defn}
\begin{prop}
Let $(A, \prec,\succ,\alpha)$ be a Hom-Leibniz dendriform algebra. Then the binary operation
$[\cdot,\cdot]: A\otimes A\rightarrow A$, given for all $x, y \in A$ by
$[x,y] = x \prec y + x \succ y,$
defines a Hom-Leibniz algebra, which is called the sub-adjacent Hom-Leibniz algebra of $(A, \prec,\succ,\alpha)$, and
$(A, \prec,\succ,\alpha)$ is called a compatible Hom-Leibniz dendriform algebra structure on $(A, [\cdot,\cdot],\alpha)$.
\end{prop}
\begin{proof}
For all $x, y, z\in A$,
\begin{align*}
&[\alpha(x),[y,z]]-[[x,y],\alpha(z)]-[\alpha(y),[x,z]]\\
&=[\alpha(x),y\prec z+y\succ z]-[x\prec y+x\succ y,\alpha(z)]-[\alpha(y),x\prec z+x\succ z]\\
&=\alpha(x)\prec(y\prec z+y\succ z)+\alpha(x)\succ(y\prec z+y\succ z) \\
&\quad -(x\prec y+ x\succ y)\prec\alpha(z)-(x\prec y+x\succ y)\succ\alpha(z)\\
&\quad -\alpha(y)\prec(x\prec z+ x\succ z)-\alpha(y)\succ(x\prec z+x\succ z)\\
&=\underbrace{\Big(-[x,y]\succ\alpha(z)+\alpha(x)\succ(y\succ z)-\alpha(y)\succ(x\succ z)\Big)}_{\text{$=0$ by \eqref{Hom-Leibniz dendriform1}}} \\
&\quad +\underbrace{\Big(\alpha(x)\succ(y\prec z)-(x\succ y)\prec\alpha(z)-\alpha(y)\succ([x,z])\Big)}_{\text{$=0$ by \eqref{Hom-Leibniz dendriform2}}}\\
&\quad +\underbrace{\Big(\alpha(x)\prec([y,z])-(x\prec y)\prec\alpha(z)-\alpha(y)\succ(x\prec z)\Big)}_{\text{$=0$ by \eqref{Hom-Leibniz dendriform3}}}=0.
\qedhere \end{align*}
\end{proof}
\begin{ex}\label{example-7}
Let $V$ be a linear space. On the linear space $\mathfrak{gl}(V)\oplus V$, define two binary operations $\prec,~\succ:(\mathfrak{gl}(V)\oplus V)\otimes (\mathfrak{gl}(V)\oplus V)\rightarrow \mathfrak{gl}(V)\oplus V$ and a linear map $\delta_\beta:\mathfrak{gl}(V)\oplus V\rightarrow \mathfrak{gl}(V)\oplus V$ by
\begin{eqnarray*}
(A+u)\prec(B+v)&=&AB+Av,\\ (A+u)\succ(B+v)&=&-BA,\\
\delta_\beta(A+u)&=&Ad_\beta(A)+\beta(u),\quad
\forall A,B\in\mathfrak{gl}(V),~ u,v\in V,
\end{eqnarray*}
with the adjoint action $Ad_\beta(A)=\beta A\beta^{-1}$ for all $\beta\in\mathfrak{gl}(V).$
Then $(\mathfrak{gl}(V)\oplus V,\prec,\succ,\delta_\beta)$ is a Hom-Leibniz dendriform algebra. Its sub-adjacent Hom-Leibniz algebra is exactly the one underlying a Hom-omni-Lie algebra \cite{ArmakanSilvestrov2020:colorhomliehomleibnomnihomlie,ShengXiong:LMLA2015:OnHomLiealg}.
\end{ex}
\begin{thm} \label{thm:homLeibDenr:ytwist}
Let $\mathcal{A}=(A, \prec ,\succ)$ be a Leibniz dendriform algebra and
$\alpha :A\rightarrow A$ be a Leibniz dendriform algebras
morphism of  $\mathcal{A}.$ With bilinear maps $\prec_{\alpha}: A
\times A\rightarrow A$ and $\succ_\alpha: A
\times A\rightarrow A$ defined for all $x, y\in A$ by
$
x\prec _{\alpha}y=\alpha (x\prec y)$ and $x\succ_{\alpha}y=\alpha(x\succ y),$
the quadruple  $\mathcal{A}_\alpha=(A_\alpha=A, \prec _{\alpha},\succ_\alpha, \alpha)$ is a Hom-Leibniz dendriform algebra, called $\alpha$-twist (Yau twist) of $\mathcal{A}$. Moreover, assume that $\mathcal{A}'=(A', \prec',\succ')$ is another Leibniz dendriform algebra, and
$\alpha':\mathcal{A}'\rightarrow \mathcal{A}'$ is a Leibniz dendriform algebras morphism.
Let $f:\mathcal{A}\rightarrow \mathcal{A}'$ be a
Leibniz dendriform algebras morphism satisfying $f\circ \alpha =\alpha
'\circ f$. Then, $f:\mathcal{A}_{\alpha }\rightarrow \mathcal{A}'_{\alpha'}$ is a
Hom-Leibniz dendriform algebras morphism.
\end{thm}
\begin{proof}
Being a Leibniz dendriform algebras
morphism, $\alpha:A\rightarrow A$ is a linear map which is multiplicative with respect to $\prec$ and $\succ$, that is,
$$ \forall\ x,y \in A: \quad
\alpha (x \prec y) =\alpha (x) \prec \alpha (y), \quad \alpha (x \succ y) =\alpha (x) \succ \alpha (y).
$$
The equalities \eqref{Hom-Leibniz dendriform1}, \eqref{Hom-Leibniz dendriform2} and \eqref{Hom-Leibniz dendriform3} in $\mathcal{A}_{\alpha}$
are proved as follows:
\begin{align*}
([x,y]_\alpha)\succ_{\alpha}\alpha(z) = & (x\prec_\alpha y+x\succ_\alpha y)\succ_{\alpha}\alpha(z)\\
=&
(\alpha (x\prec y)+\alpha(x\succ y))\succ_{\alpha}\alpha(z) \\
=&\alpha((\alpha (x\prec y)+\alpha(x\succ y))\succ\alpha(z)) \\
^\text{(multiplicativity)}=&(\alpha^2 (x\prec y)+\alpha^2(x\succ y))\succ\alpha^2(z)
\\
^\text{(multiplicativity)}=&
(\alpha^2 (x)\prec \alpha^2 (y)+\alpha^2(x)\succ \alpha^2 (y))\succ\alpha^2(z)
\\
=&[\alpha^{2}(x),\alpha^{2}(y)]\succ\alpha^{2}(z)\\
^\text{($\mathcal{A}$ is Leibniz dendriform)} =&\alpha^{2}(x)\succ(\alpha^{2}(y)\succ\alpha^{2}(z))-\alpha^{2}(y)\succ(\alpha^{2}(x)\succ\alpha^{2}(z))\\
^\text{(multiplicativity)}=&\alpha^{2}(x)\succ\alpha(\alpha(y)\succ\alpha(z))-\alpha^{2}(y)\succ\alpha(\alpha(x)\succ\alpha(z))\\
=&\alpha(x)\succ_{\alpha}(\alpha(y)\succ\alpha(z))-\alpha(y)\succ_{\alpha}(\alpha(x)\succ\alpha(z))\\
=&\alpha(x)\succ_{\alpha}(y\succ_{\alpha} z)-\alpha(y)\succ_{\alpha}(x\succ_{\alpha} z), \\*[0,2cm] \alpha(x)\succ_{\alpha}(y\prec_{\alpha} z)=& \alpha(x \succ \alpha (y \prec z)) \\
^\text{(multiplicativity)}=& \alpha^2(x \succ(y\prec z)) \\
=& \alpha^2((x\succ y)\prec z))+\alpha^2(x \succ(y\prec z))-\alpha^2((x\succ y) \prec z )\\
=& \alpha^2((x\succ y)\prec z))+\alpha^2(x \succ(y\prec z)-(x\succ y) \prec z )\\
^\text{($\mathcal{A}$ is Leibniz dendriform)} =& \alpha^2((x\succ y)\prec z))+\alpha^2(y\prec [x,z])\\
^\text{(multiplicativity)} =& \alpha(\alpha(x\succ y)\prec \alpha(z))+\alpha(\alpha(y)\prec \alpha([x,z]))\\
=& \alpha(\alpha(x\succ y)\prec \alpha(z))+\alpha(\alpha(y)\prec \alpha(x\prec z+x\succ z))\\
=& \alpha(\alpha(x\succ y)\prec \alpha(z))+\alpha(\alpha(y)\prec (\alpha(x\prec z)+\alpha(x\succ z)))\\
=&\alpha(\alpha(x\succ y)\prec \alpha(z))+\alpha(\alpha(y)\prec (x\prec_{\alpha} z+x\succ_{\alpha} z))\\
=&(x\succ_{\alpha} y) \prec_{\alpha}\alpha(z)+\alpha(y)\prec_{\alpha} ([x,z]_{\alpha}), \\*[0,2cm]
\alpha(x)\prec_\alpha ([y,z]_\alpha)=&\alpha(x)\prec_\alpha (y\prec_{\alpha} z+y\succ_{\alpha} z) \\
^\text{(multiplicativity)}=&\alpha^2(x\prec (y\prec z+y\succ z)) \\
^\text{(multiplicativity)}=&\alpha^2(x\prec [y,z]) \\
^\text{($\mathcal{A}$ is Leibniz dendriform)}=&\alpha^2((x\prec y) \prec z + y\succ (x \prec z)) \\
=&\alpha^2((x\prec y) \prec z)+\alpha^2(y\succ (x \prec z)) \\
^\text{(multiplicativity)}=&\alpha(\alpha(x\prec y) \prec \alpha(z))+\alpha(\alpha(y)\succ \alpha(x \prec z)) \\
=&(x\prec_\alpha y) \prec_\alpha \alpha(z)+\alpha(y)\succ_\alpha (x \prec_\alpha z).
\end{align*}

The second assertion follows from
\begin{align*}
    f(x\prec_{\alpha} y)&=f(\alpha(x\prec y))
    =\alpha'( f(x\prec' y)) = \alpha' (f(x) \prec' f(y))
    =f(x)\prec'_{\alpha'} f(y), \\
    f(x\succ_{\alpha} y)&=f(\alpha(x\succ y))
    =\alpha' (f(x\succ' y)) = \alpha' (f(x) \succ' f(y))
    =f(x)\succ'_{\alpha'} f(y).
\qedhere
\end{align*}
\end{proof}

\begin{prop}\label{prop:isma:HomLeibdendriform}
Let $\mathcal{A}=(A, \prec ,\succ,\alpha)$ be a Hom-Leibniz dendriform algebra and
$\alpha' :A\rightarrow A$ be a Hom-Leibniz dendriform algebras
morphism of $\mathcal{A}.$ With bilinear maps $\prec_{\alpha'}: A
\times A\rightarrow A$ and $\succ_{\alpha'}: A
\times A\rightarrow A$ defined for all $x, y\in A$ by
$
x\prec _{\alpha'}y=\alpha' (x\prec y)$ and $x\succ_{\alpha'}y=\alpha'(x\succ y),$
the quadruple  $\mathcal{A}_{\alpha'}=(A, \prec _{\alpha'},\succ_{\alpha'}, \alpha\circ\alpha')$ is a Hom-Leibniz dendriform algebra.

\end{prop}
\begin{proof}
Being a Hom-Leibniz dendriform algebras
morphism, $\alpha':A\rightarrow A$ is a linear map which is multiplicative with respect to $\prec$ and $\succ$. Then \eqref{Hom-Leibniz dendriform1}, \eqref{Hom-Leibniz dendriform2} and \eqref{Hom-Leibniz dendriform3} in $\mathcal{A}_{\alpha'}$
are proved for all $x,y,z\in A$ as follows:
\begin{align*}
([x,y]_{\alpha'})\succ_{\alpha'}\alpha\alpha'(z)=&(x\prec_{\alpha'}y+x\succ_{\alpha'} y)\succ_{\alpha'}\alpha\alpha'(z)\\
=&(\alpha'(x\prec y)+\alpha'(x\succ y))\succ_{\alpha'}\alpha\alpha'(z)\\
=&\alpha'((\alpha'(x\prec y)+\alpha'(x\succ y))\succ\alpha\alpha'(z))\\
^\text{($\alpha'$ morphism)}=&(\alpha'^{2}(x\prec y)+\alpha'^{2}(x\succ y))\succ\alpha'\alpha\alpha'(z)\\
^\text{($\alpha'$ morphism)} =&(\alpha'^{2}(x)\prec\alpha'^{2}(y)+\alpha'^{2}(x)\succ\alpha'^{2}(y))\succ\alpha'\alpha\alpha'(z)\\
^\text{($\alpha$, $\alpha'$ commute)} =&(\alpha'^{2}(x)\prec\alpha'^{2}(y)+\alpha'^{2}(x)\succ\alpha'^{2}(y))\succ\alpha\alpha'^{2}(z)\\
=&[\alpha'^{2}(x),\alpha'^{2}(y)]\succ\alpha\alpha'^{2}(z)\\
^\text{($\mathcal{A}$ Hom-Leib. dendr.)} =&\alpha\alpha'^{2}(x)\succ(\alpha'^{2}(y)\succ\alpha'^{2}(z))-\alpha\alpha'^{2}(y)\succ (\alpha'^{2}(x)\succ\alpha'^{2}(z))\\
^\text{($\alpha'$ morphism)}=&\alpha\alpha'^{2}(x)\succ\alpha'(\alpha'(y)\succ\alpha'(z))
-\alpha\alpha'^{2}(y)\succ\alpha'(\alpha'(x)\succ\alpha'(z))\\
^\text{($\alpha$, $\alpha'$  commute)} =&\alpha'\alpha\alpha'(x)\succ\alpha'(\alpha'(y)\succ\alpha'(z))
-\alpha'\alpha\alpha'(y)\succ\alpha'(\alpha'(x)\succ\alpha'(z))\\
=&\alpha\alpha'(x)\succ_{\alpha'}(\alpha'(y)\succ\alpha'(z))
-\alpha\alpha'(y)\succ_{\alpha'}(\alpha'(x)\succ\alpha'(z))\\
=&\alpha\alpha'(x)\succ_{\alpha'}(y\succ_{\alpha'} z)
-\alpha\alpha'(y)\succ_{\alpha'}(x\succ_{\alpha'} z),\\*[0,2cm]
\alpha\alpha'(x)\succ_{\alpha'}(y\prec_{\alpha'} z)
=&\alpha\alpha'(x)\succ_{\alpha'}(\alpha'(y\prec z))\\
=&\alpha'(\alpha\alpha'(x)\succ\alpha'(y\prec z))\\
^\text{($\alpha'$ morphism)}=&\alpha'\alpha\alpha'(x)\succ\alpha'^{2}(y\prec z)\\
^\text{($\alpha'$ morphism)}=&\alpha'\alpha\alpha'(x)\succ(\alpha'^{2}(y)\prec \alpha'^{2}(z))\\
^\text{($\alpha$, $\alpha'$ commute)}=&\alpha\alpha'^{2}(x)\succ(\alpha'^{2}(y)\prec \alpha'^{2}(z))\\
^\text{($\mathcal{A}$ Hom-Leib. dendr.)} =&(\alpha'^{2}(x)\succ\alpha'^{2}(y))\prec\alpha\alpha'^{2}(z)
+\alpha\alpha'^{2}(y)\prec([\alpha'^{2}(x),\alpha'^{2}(z)])\\
=&(\alpha'^{2}(x)\succ\alpha'^{2}(y))\prec\alpha\alpha'^{2}(z)\\
&+\alpha\alpha'^{2}(y)\prec(\alpha'^{2}(x)\prec\alpha'^{2}(z)+\alpha'^{2}(x)\succ\alpha'^{2}(z))\\
^\text{($\alpha$, $\alpha'$  commute)} =&(\alpha'^{2}(x)\succ\alpha'^{2}(y))\prec\alpha'\alpha\alpha'(z)\\
&+\alpha'\alpha\alpha'(y)\prec(\alpha'^{2}(x)\prec\alpha'^{2}(z)+\alpha'^{2}(x)\succ\alpha'^{2}(z))\\
^\text{($\alpha'$ morphism)}= &\alpha'(\alpha'(x)\succ\alpha'(y))\prec\alpha'\alpha\alpha'(z)\\
&+\alpha'\alpha\alpha'(y)\prec\alpha'(\alpha'(x)\prec\alpha'(z)+\alpha'(x)\succ\alpha'(z))\\
=&(x\succ_{\alpha'} y) \prec_{\alpha'}\alpha\alpha'(z)+\alpha\alpha'(y)\prec_{\alpha'}(x\prec_{\alpha'}z+x\succ_{\alpha'} z)\\
=&(x\succ_{\alpha'} y) \prec_{\alpha'}\alpha\alpha'(z)+\alpha\alpha'(y)\prec_{\alpha'}([x,z]_{\alpha'}),\\*[0,2cm]
\alpha\alpha'(x)\prec_{\alpha'}([y,z]_{\alpha'})
=&\alpha\alpha'(x)\prec_{\alpha'}(y\prec_{\alpha'} z+ y\succ_{\alpha'} z)\\
=&\alpha\alpha'(x)\prec_{\alpha'}(\alpha'(y\prec z)+ \alpha'(y\succ z))\\
=&\alpha'(\alpha\alpha'(x)\prec(\alpha'(y\prec z)+ \alpha'(y\succ z)))\\
^\text{($\alpha'$ morphism)}=&\alpha'\alpha\alpha'(x)\prec(\alpha'^{2}(y\prec z)+ \alpha'^{2}(y\succ z))\\
^\text{($\alpha'$ morphism)}=&\alpha'\alpha\alpha'(x)\prec(\alpha'^{2}(y)\prec \alpha'^{2}(z)+ \alpha'^{2}(y)\succ \alpha'^{2}(z))\\
=&\alpha'\alpha\alpha'(x)\prec([\alpha'^{2}(y), \alpha'^{2}(z)])\\
^\text{($\alpha$, $\alpha'$ commute)}=&\alpha\alpha'^{2}(x)\prec([\alpha'^{2}(y), \alpha'^{2}(z)])\\
^\text{($\mathcal{A}$ Hom-Leib. dendr.)} =&(\alpha'^{2}(x)\prec\alpha'^{2}(y))\prec\alpha\alpha'^{2}(z)
+\alpha\alpha'^{2}(y)\succ(\alpha'^{2}(x)\prec\alpha'^{2}(z))\\
^\text{($\alpha$, $\alpha'$ commute)}
=&(\alpha'^{2}(x)\prec\alpha'^{2}(y))\prec\alpha'\alpha\alpha'(z)
+\alpha'\alpha\alpha'(y)\succ(\alpha'^{2}(x)\prec\alpha'^{2}(z))\\
^\text{($\alpha'$ morphism)} =&\alpha'(\alpha'(x)\prec\alpha'(y))\prec\alpha'\alpha\alpha'(z)
+\alpha'\alpha\alpha'(y)\succ\alpha'(\alpha'(x)\prec\alpha'(z))\\
=&(\alpha'(x)\prec\alpha'(y))\prec_{\alpha'}\alpha\alpha'(z)
+\alpha\alpha'(y)\succ_{\alpha'}(\alpha'(x)\prec\alpha'(z))\\
=&(x\prec_{\alpha'}y)\prec_{\alpha'}\alpha\alpha'(z)
+\alpha\alpha'(y)\succ_{\alpha'}(x\prec_{\alpha'}z).
\qedhere \end{align*}
\end{proof}

\begin{cor} \label{cor:isma:HomLeibdend}
If $\mathcal{A}=(A, \prec,\succ,\alpha)$ is a multiplicative Hom-Leibniz dendriform
algebra, then for any positive integer $n$,
\begin{enumerate}
\item
The $n{\rm th}$ derived Hom-Leibniz dendriform algebra of type $1$ of $\mathcal{A}$ is
defined by
$$\mathcal{A}_{1}^{n}=(A,\prec^{(n)}=\alpha^{n}\circ\prec,\succ^{(n)}=\alpha^{n}\circ\succ,\alpha^{n+1}).$$
\item
The $n{\rm th}$ derived Hom-Leibniz dendriform algebra of type $2$ of $A$ is
defined by
$$\mathcal{A}_{2}^{n}=(A,\prec^{(2^n-1)}=\alpha^{2^n-1}\circ\prec,\succ^{(2^n-1)}=\alpha^{2^n-1}\circ\succ,\alpha^{2^n}).$$
\end{enumerate}
\end{cor}
\begin{proof}
Apply Proposition \ref{prop:isma:HomLeibdendriform} with $\alpha'=\alpha^{n}$ and
$\alpha'=\alpha^{2^n-1}$ respectively.
\end{proof}

\begin{ex}
There is a three-dimensional multiplicative Hom-Leibniz dendriform algebras $(A, \prec,\succ,\alpha)$ with multiplications tables for a basis $\{e_1, e_2, e_3\}$:
\begin{center}
\begin{tabular}{c|cccc}
$\prec$ & $e_1$ & $e_2$ & $e_3$ \\ \hline
$e_1$ & $0$ & $0$ & $-e_2$  \\
$e_2$ & $0$ & $0$ & $0$  \\
$e_3$ & $-e_2$ & $0$ & $2e_2$ \\
\end{tabular}
\hspace{1 cm}
\begin{tabular}{c|cccc}
$\succ$ & $e_1$ & $e_2$ & $e_3$ \\ \hline
$e_1$ & $0$ & $0$ & $0$  \\
$e_2$ & $0$ & $0$ & $0$  \\
$e_3$ & $-e_2$ & $0$ & $2e_2$
\end{tabular}
\end{center}
$$\begin{array}{llll}
    \alpha(e_1)=-\frac{p}{2}e_1, &\alpha(e_2)=\frac{p^{2}}{2}e_2,&\alpha(e_3)=pe_3,
\end{array}$$
with a parameter $p\in \mathbb{K}\setminus\{0\}$.

Then there are Hom-Leibniz dendriform algebras $\mathcal{A}_{1}^{n}$ and $\mathcal{A}_{2}^{n}$ with multiplications
tables respectively:
\begin{center}
\begin{tabular}{c|cccc}
$\prec^{(n)}$ & $e_1$ & $e_2$ & $e_3$ \\ \hline
$e_1$ & $0$ & $0$ & $-(-\frac{p^{2}}{2})^{n}e_2$  \\
&  &  &   \\
$e_2$ & $0$ & $0$ & $0$  \\
&  &  &   \\
$e_3$ & $-(-\frac{p^{2}}{2})^{n}e_2$ & $0$ & $2p^{2n} e_2$
\end{tabular}
\hspace{1 cm}
\begin{tabular}{c|cccc}
$\succ^{(n)}$ & $e_1$ & $e_2$ & $e_3$ \\ \hline
$e_1$ & $0$ & $0$ & $0$  \\
&  &  &   \\
$e_2$ & $0$ & $0$ & $0$  \\
&  &  &   \\
$e_3$ & $-(-\frac{p^{2}}{2})^{n}e_2$ & $0$ & $2p^{2n} e_2$
\end{tabular}
\end{center}

$$\alpha^{n+1}(e_1)=\left(-\frac{p}{2}\right)^{n+1}e_1, \quad \alpha^{n+1}(e_2)=\left(\frac{p^{2}}{2}\right)^{n+1}e_2, \quad \alpha^{n+1}(e_3)=p^{n+1}e_3,$$
\vspace{0,1 cm}
\begin{center}
\begin{tabular}{c|cccc}
$\prec^{(2^{n}-1)}$ & $e_1$ & $e_2$ & $e_3$ \\ \hline
$e_1$ & $0$ & $0$ & $(\frac{p^{2}}{2})^{2^{n}-1}e_2$  \\
&  &  &   \\
$e_2$ & $0$ & $0$ & $0$  \\
&  &  &   \\
$e_3$ & $(\frac{p^{2}}{2})^{2^{n}-1}e_2$ & $0$ & $2(p^{2})^{2^{n}-1}e_2$
\end{tabular}
\hspace{0,1 cm}
\begin{tabular}{c|cccc}
$\succ^{(2^{n}-1)}$ & $e_1$ & $e_2$ & $e_3$ \\ \hline
$e_1$ & $0$ & $0$ & $0$  \\
&  &  &   \\
$e_2$ & $0$ & $0$ & $0$  \\
&  &  &   \\
$e_3$ & $(\frac{p^{2}}{2})^{2^{n}-1}e_2$ & $0$ & $2(p^{2})^{2^{n}-1}e_2$
\end{tabular}
\end{center}

$$\alpha^{2^{n}}(e_1)=\left(\frac{p}{2}\right)^{2^{n}}e_1, \quad \alpha^{2^{n}}(e_2)=\left(\frac{p^{2}}{2}\right)^{2^{n}}e_2, \quad \alpha^{2^{n}}(e_3)=p^{2^{n}}e_3.
$$
\end{ex}

Next we define the notion of bimodule of a Hom-Leibniz dendriform algebra. 
\begin{defn}
Let $\mathcal{A}=(A, \prec,\succ, \alpha)$ be a Hom-Leibniz dendriform algebra, and $(V, \beta)$ be a Hom-module. Let $ l_\prec, r_\prec, l_\succ, r_\succ: A \rightarrow gl(V) $ be four linear maps. The sextuple $(l_\prec, r_\prec, l_\succ, r_\succ, \beta, V)$ is called a bimodule of 
$\mathcal{A}$ if for all $ x, y \in  A, v \in V $,
\begin{eqnarray}
l_{\succ}([x,y])\beta(v)&=&l_\succ(\alpha(x))l_\succ(y)v-l_\succ(\alpha(y))l_\succ(x)v,\label{def:bimodhomdendif:1}\\
r_\succ(\alpha(y))l(x)v&=&l_\succ(\alpha(x))r_\succ(y)v-r_\succ(x\succ y)\beta(v),\label{def:bimodhomdendif:2}\\
r_{\succ}(\alpha(y))r(x)v&=&r_\succ(x\succ y)\beta(v)-l_\succ(\alpha(x))r_\succ(y)v,\label{def:bimodhomdendif:3}\\
l_\succ(\alpha(x))l_\prec(y)v&=& l_\prec(x\succ y)\beta(v) +l_\prec(\alpha(y))l(x)v,\label{def:bimodhomdendif:4}\\
l_\succ(\alpha(x))r_\prec(y)v&=&r_\prec(\alpha(y))l_\succ(x)v+r_\prec([x,y])\beta(v),\label{def:bimodhomdendif:5}\\
r_\succ(x\prec y)\beta(v)&=&r_\prec(\alpha(y))r_\succ(x)v+l_\prec(\alpha(x))r(y)v,\label{def:bimodhomdendif:6}\\
l_\prec(\alpha(x))l(y)v&=&l_\prec(x\prec y)\beta(v)+l_\succ(\alpha(y))l_\prec(x)v,\label{def:bimodhomdendif:7}\\
l_\prec(\alpha(x))r(y)v&=&r_{\prec}(\alpha(y))l_{\prec}(x)v+r_{\succ}(x\prec y)\beta(v),\label{def:bimodhomdendif:8}\\
r_\prec([x,y])\beta(v)&=&r_{\prec}(\alpha(y))r_{\prec}(x)v+l_\succ(\alpha(x))r_\prec(y)v,\label{def:bimodhomdendif:9}\\
\beta(l_\prec(x)v)&=&l_\prec(\alpha(x))\beta(v)\label{def:bimodhomdendif:10}\\
\beta(r_\prec(x)v)&=&l_\prec(\alpha(x))\beta(v)\label{def:bimodhomdendif:11}\\
\beta(l_\succ(x)v)&=&l_\succ(\alpha(x))\beta(v)\label{def:bimodhomdendif:12}\\
\beta(r_\succ(x)v)&=&l_\succ(\alpha(x))\beta(v)\label{def:bimodhomdendif:13}
\end{eqnarray}
where $[x,y]=x\prec y+x\succ y,~~l=l_\prec+l_\succ$ and $r=r_\prec+r_\succ.$
\end{defn}
\begin{prop}\label{bimodHom-Leibniz dendriform}
Let $(l_\prec, r_\prec,l_\succ, r_\succ,\beta, V)$ be a bimodule of a Hom-Leibniz-dendri\-form algebra
$(A, \prec,\succ, \alpha)$. Then, the direct sum $A \oplus V$ of linear spaces is a
Hom-Leibniz dendriform algebra with multiplication given for all $ x_{1}, x_{2} \in  A, v_{1}, v_{2} \in V$ by
\begin{eqnarray*}
(x_{1} + v_{1}) \prec' (x_{2} + v_{2}) & = & x_{1} \prec x_{2} + (l_\prec(x_{1})v_{2} + r_\prec(x_{2})v_{1}),\cr
(x_{1} + v_{1}) \succ' (x_{2} + v_{2}) & = & x_{1} \succ x_{2} + (l_\succ(x_{1})v_{2} + r_\succ(x_{2})v_{1}),\cr
(\alpha\oplus\beta)(x_{1} + v_{1}) & = & \alpha(x_{1}) + \beta(v_{1}).
\end{eqnarray*}
\end{prop}
\begin{proof}
We prove the axiom \eqref{Hom-Leibniz dendriform1} in $A\oplus V$ as the others relations are proved analogously. For any
$x_{1},x_{2},x_{3}\in A$ and $v_1, v_2, v_3\in V$,
\begin{align*}
&([x_1+v_1,x_2+v_2]')\succ'(\alpha+\beta)(x_3+v_3)\\
&\quad -(\alpha+\beta)(x_1+v_1)\succ'((x_2+v_2)\succ'(x_3+v_3))\\
&\quad +(\alpha+\beta)(x_2+v_2)\succ'((x_1+v_1)\succ'(x_3+v_3))\\
&=([x_1,x_2]l(x_1)v_2+r(x_2)v_1)\succ'(\alpha+\beta)(x_3+v_3)\\
&\quad-(\alpha+\beta)(x_1+v_1)\succ'(x_2\succ x_3+l_\succ(x_2)v_3+r_\succ(x_3)v_2)\\
&\quad+(\alpha+\beta)(x_2+v_2)\succ'(x_1\succ x_3+l_\succ(x_1)v_3+r_\succ(x_3)v_1)\\
&=[x_1,x_2]\succ\alpha(x_3)+l_\succ([x_1,x_2])\beta(v_3)+r_\succ(\alpha(x_3))l(x_1)v_2\\
&\quad+r_\succ(\alpha(x_3))r(x_2)v_1-\alpha(x_1)\succ(x_2\succ x_3)-l_\succ(\alpha(x_1))l_\succ(x_2)v_3\\
&\quad-l_\succ(\alpha(x_1))r_\succ'x_3)v_2-r_\succ(x_2\succ x_3)\beta(v_1)+\alpha(x_2)\succ(x_1\succ x_3)\\
&\quad+l_\succ(\alpha(x_2))l_\succ(x_1)v_3+l_\succ(\alpha(x_2))r_\succ(x_3)v_1+r_\succ(x_1\succ x_3)\beta(v_2).
 \end{align*}
By \eqref{def:bimodhomdendif:1}-\eqref{def:bimodhomdendif:3},
and \eqref{Hom-Leibniz dendriform1} in $A$,
\begin{align*}
([x_1+v_1,x_2+v_2]')\succ'(\alpha+\beta)&(x_3+v_3)-(\alpha+\beta)(x_1+v_1)\succ'((x_2+v_2)\succ'(x_3+v_3))\\
&+(\alpha+\beta)(x_2+v_2)\succ'((x_1+v_1)\succ'(x_3+v_3))=0.
\qedhere \end{align*}
\end{proof}
We denote such a Hom-Leibniz dendriform algebra by $(A\oplus V, \prec',\succ', \alpha+ \beta),$
or $A\times_{l_\prec, r_\prec, l_\succ, r_\succ,\alpha, \beta} V.$
\begin{ex}
Let $(A,\prec,\succ,\alpha)$ be a Hom-Leibniz dendriform algebra, and
\begin{align*}
\forall \quad & (x,y)\in A\times A: \\
& L_{\prec}(x)y=x\prec y, \quad R_{\prec}(x)y=y\prec x, \\
& L_{\succ}(x)y=x\succ y,\quad R_{\succ}(x)y=y\succ x.
\end{align*}
Then, $(L_{\prec},R_{\prec},L_{\succ}, R_{\succ},\alpha,A)$ is called a regular
bimodule of $(A,\prec,\succ,\alpha)$.
\end{ex}

\begin{prop}\label{prop of bimodules}
Let $(l_{\prec}, r_{\prec}, l_{\succ}, r_{\succ}, \beta, V)$ be a bimodule of a Hom-Leibniz-dendri\-form algebra $( A,
\prec, \succ, \alpha)$. Let $( A, [\cdot,\cdot], \alpha)$ be the associated Hom-Leibniz algebra. Then,
\begin{enumerate}[label=\upshape{\arabic*)}]
\item\label{propr1} $ (l_{\succ}, r_{\prec}, \beta, V) $ and $ (l_{\prec} + l_{\succ}, r_{\prec} + r_{\succ}, \beta, V)$ are bimodules of $ ( A, [\cdot,\cdot], \alpha). $
\item\label{propr2} For any bimodule $(l, r, \beta, V)$ of $( A, [\cdot,\cdot], \alpha)$,
$(0, l, r, 0, \beta, V)$ is a bimodule \\ of $( A, \prec,\succ,\alpha). $
\item\label{propr3} $(0,l_{\prec} + l_{\succ},  r_{\prec} + r_{\succ},0, \beta, V)$ and $(0,l_{\succ},  r_{\prec},0, \beta, V)$ are bimodules \\
 of $( A, \prec, \succ, \alpha).$
\item\label{propr4}  The Hom-Leibniz dendriform algebras $  A \times_{l_{\prec}, r_{\prec}, l_{\succ}, r_{\succ}, \alpha, \beta} V $ and \\  $  A \times_{0,l_{\prec} +  l_{\succ} , r_{\prec} + r_{\succ},0, \alpha, \beta} V $ have the same associated
 Hom-Leibniz algebra\\ $ A \times_{l_{\prec} +  l_{\succ}, r_{\prec} + r_{\succ}, \alpha, \beta} V.$
 \end{enumerate}
\end{prop}
\begin{proof}
\ref{propr1}
    We prove that $ (l_{\succ}, r_{\prec}, \beta, V) $ and $ (l_{\prec} + l_{\succ}, r_{\prec} + r_{\succ}, \beta, V)$ respectively satisfies the axioms \eqref{Cond:HomBimod:lpb}- \eqref{Cond:HomBimod:lar}. For any $x,y\in A$ and $v\in V$,
    \begin{alignat*}{2}
    &l_{\succ}(\alpha(x))l_\succ(y)v-l_\succ(\alpha(y))l_\succ(x)v-l_\succ([x,y])\beta(v)=0, & \quad \text{(by \eqref{def:bimodhomdendif:1})}\\
  & l_\succ(\alpha(x))r_\prec(y)v-r_\prec(\alpha(y))l_\succ(x)v-r_\prec([x,y])\beta(v)=0, & \quad
  \text{(by \eqref{def:bimodhomdendif:5})}\\
  & r_\prec([x,y])\beta(v)-r_\prec(\alpha(y))r_\prec(x)v-l_\succ(\alpha(x))r_\prec(y)v=0. & \quad
  \text{(by \eqref{def:bimodhomdendif:9})}
       \end{alignat*}
       Then $(l_{\succ}, r_{\prec}, \beta, V) $ is a bimodule of $( A, [\cdot,\cdot], \alpha).$
       \begin{align*}
       &(l_\prec+l_\succ)(\alpha(x))(l_\prec+l_\succ)(y)v-(l_\prec+l_\succ)([x,y])\beta(v)\\
       &\quad\quad-(l_\prec+l_\succ)(\alpha(y))(l_\prec+l_\succ)(x)v\\
       &\quad=
       \overbrace{l_\prec(\alpha(x))l_\prec(y)v+l_\prec(\alpha(x))l_\succ(y)v}^{*}+l_\succ(\alpha(x))l_\prec(y)v +l_\succ(\alpha(x))l_\succ(y)v\\
       &\quad\quad \overbrace{-l_\prec([x,y])\beta(v)}^{**}-l_\succ([x,y])\beta(v)\\
       &\quad\quad \overbrace{-l_\prec(\alpha(y))l_\prec(x)v-l_\prec(\alpha(y))l_\succ(x)v}^{***}-l_\succ(\alpha(y))l_\prec(x)v-l_\succ(\alpha(y))l_\succ(x)v\\
       &\quad=\underbrace{\Big(\overbrace{l_\prec(\alpha(x))l(y)v}^{*}
       \overbrace{-l_\prec(x\prec y)\beta(v)}^{**}-l_\succ(\alpha(y))l_\prec(x)v\Big)}_{\text{= 0 by \eqref{def:bimodhomdendif:7}}}\\
       &\quad\quad+ \underbrace{\Big(l_\succ(\alpha(x))l_\prec(y)v
       \overbrace{-l_\prec(x\succ y) \beta(v)}^{**} \overbrace{-l_\prec(\alpha(y))l(x)v}^{***}\Big)}_{ \text{=0 by \eqref{def:bimodhomdendif:4}}}\\
       &\quad\quad+
       \underbrace{\Big(l_\succ(\alpha(x))l_\succ(y)v-l_\succ([x,y])\beta(v)-l_\succ(\alpha(y))l_\succ(x)v\Big)}_{\text{= 0 by \eqref{def:bimodhomdendif:1}}}=0,\\
       &(l_\prec+l_\succ)(\alpha(x))(r_\prec+r_\succ)(y)v-(r_\prec+r_\succ)(\alpha(y))(l_\prec+l_\succ)(x)v\\
       &\quad\quad-(r_\prec+r_\succ)([x,y])\beta(v)\\
       &\quad=
       \overbrace{l_\prec(\alpha(x))r_\prec(y)v+l_\prec(\alpha(x))r_\succ(y)v}^{*}+l_\succ(\alpha(x))r_\prec(y)v
       +l_\succ(\alpha(x))r_\succ(y)v\\
       &\quad\quad-r_\prec(\alpha(y))l_\prec(x)v-r_\prec(\alpha(y))l_\succ(x)v
       \overbrace{-r_\succ(\alpha(y))l_\prec(x)v-r_\succ(\alpha(y))l_\succ(x)v}^{***}\\
       &\quad\quad-r_\prec([x,y])\beta(v)\overbrace{-r_\succ([x,y])\beta(v)}^{**}\\
       &\quad=\underbrace{\Big(\overbrace{l_\prec(\alpha(x))r(y)v}^{*}-r_\prec(\alpha(y))l_\prec(x)v
       \overbrace{-r_\succ(x\prec y)\beta(v)}^{**}\Big)}_{\text{=0 by \eqref{def:bimodhomdendif:8}}}\\
       &\quad\quad+\underbrace{\Big(l_\succ(\alpha(x))r_\succ(y)v \overbrace{-r_\succ(\alpha(y))l(x)v}^{***}
       \overbrace{-r_\succ(x\succ y)\beta(v)}^{**}\Big)}_{\text{=0 by \eqref{def:bimodhomdendif:2}}}\\
       &\quad\quad+\underbrace{\Big(l_\succ(\alpha(x))r_\prec(y)v-r_\prec(\alpha(y))l_\succ(x)v
       -r_\prec([x,y])\beta(v)\Big)}_{\text{=0 by \eqref{def:bimodhomdendif:5}}}=0,\\
      & (r_\prec+r_\succ)([x,y])\beta(v)-(r_\prec+r_\succ)(\alpha(y))(r_\prec+r_\succ)(x)v\\
      &\quad\quad-(l_\prec+l_\succ)(\alpha(x))(r_\prec+r_\succ)(y)v\\
      &\quad=r_\prec([x,y])\beta(v)+\overbrace{r_\succ([x,y])\beta(v)}^{**}\\
      &\quad\quad-r_\prec(\alpha(y))r_\prec(x)v-r_\prec(\alpha(y))r_\succ(x)v
      \overbrace{-r_\succ(\alpha(y))r_\prec(x)v-r_\succ(\alpha(y))r_\succ(x)v}^{***}\\
      &\quad\quad \overbrace{-l_\prec(\alpha(x))r_\prec(y)v-l_\prec(\alpha(x))r_\succ(y)v}^{*}-l_\succ(\alpha(x))r_\prec(y)v-l_\succ(\alpha(x))r_\succ(y)v\\
      &\quad=\underbrace{\Big(r_\prec([x,y])\beta(v)-r_\prec(\alpha(y))r_{\prec}(x)v
      -l_\succ(\alpha(x))r_\prec(y)v\Big)}_{\text{=0 by \eqref{def:bimodhomdendif:9}}}\\
      &\quad\quad +\underbrace{\Big(\overbrace{r_\succ(x\prec y)\beta(v)}^{**}-r_\prec(\alpha(y))r_\succ(x)v \overbrace{-l_\prec(\alpha(x))r(y)v}^{*}\Big)}_{\text{=0 by \eqref{def:bimodhomdendif:6}}}\\
      &\quad\quad+\underbrace{
      \Big(\overbrace{r_\succ(x\succ y)\beta(v)}^{**} \overbrace{-r_\succ(\alpha(y))r(x)v}^{***}-l_\succ(\alpha(x))r_\succ(y)v\Big)}_{\text{=0 by \eqref{def:bimodhomdendif:3}}}=0.
\end{align*}
   Then $(l_\prec+l_{\succ}, r_{\prec}+l_\succ, \beta, V) $ is a bimodule of $( A, [\cdot,\cdot], \alpha).$

\ref{propr2}
In this case, we prove only the axioms \eqref{def:bimodhomdendif:1}, \eqref{def:bimodhomdendif:5} and \eqref{def:bimodhomdendif:9}, the others axioms satisfied automatically. For any $x,y\in A$ and $v\in V$, we have
 \begin{alignat*}{2}
&l([x,y])\beta(v)=l(\alpha(x))l(y)v-l(\alpha(y))l(x)v, & \quad \text{(by~\eqref{Cond:HomBimod:lpb})}\\
&l(\alpha(x))r(y)v=r(\alpha(y))l(x)v+r([x,y])\beta(v), & \quad \text{(by~\eqref{Cond:HomBimod:rpb})}\\
&r([x,y])\beta(v)=r(\alpha(y))r(x)v+l(\alpha(x))r(y)v. & \quad \text{(by~\eqref{Cond:HomBimod:lar})}
 \end{alignat*}
 Then $(0, l, r, 0, \beta, V)$ is a bimodule of $( A, \prec,\succ,\alpha).$

\ref{propr3}
It follows directly by \ref{propr1} and \ref{propr2}, after applying \ref{propr2} with $l=l_\succ,~r=r_\prec$ and $l=l_\prec+l_\succ,~r=r_\prec+r_\succ$ respectively.

\ref{propr4}
 Let $(A\oplus V,[\cdot,\cdot]'_1,\alpha+\beta)$ and $(A\oplus V,[\cdot,\cdot]'_2,\alpha+\beta)$ the Hom-Leibniz algebras associated respectively to Hom-Leibniz dendriform algebras $  A \times_{l_{\prec}, r_{\prec}, l_{\succ}, r_{\succ}, \alpha, \beta} V $ and  $  A \times_{0,l_{\prec} +  l_{\succ} , r_{\prec} + r_{\succ},0, \alpha, \beta} V $.
 For all $x_1,x_2\in A$, $v_1,v_2\in V$,
 \begin{align*}
    & [x_1+v_1,x_2+v_2]'_1=(x_1+v_1)\prec'_1(x_2+v_2)+(x_1+v_1)\succ'_1(x_2+v_2)\\&\quad=x_1\prec x_2+x_1\succ x_2+l_\prec(x_1)v_2+l_\succ(x_1)v_2+r_\prec(x_2)v_1)+r_\succ(x_2)v_1\\
     &\quad=[x_1,x_2]+(l_\prec+l_\succ)(x_1)v_2+(r_\prec+r_\succ)(x_2)v_1\\
      &\quad=x_1\prec x_2+(l_\prec+l_\succ)(x_1)v_2+x_1\succ x_2+(r_\prec+r_\succ)(x_2)v_1\\
      &\quad =(x_1+v_1)\prec'_2(x_2+v_2)+(x_1+v_1)\succ'_2(x_2+v_2)\\
      &\quad =[x_1+v_1,x_2+v_2]'_2.
 \end{align*}
Hence, $[\cdot,\cdot]'_1=[\cdot,\cdot]'_2$.
\end{proof}
\begin{prop}
If $f:(A_1,\prec_1,\succ_1,\alpha)\longrightarrow(A_2,\prec_2,\succ_2,\beta)$ is a Hom-Leibniz dendriform algebras morphism, then
$(l_{\prec_1},r_{\prec_1},l_{\succ_1},r_{\succ_1},\beta,A_2)$
becomes a bimodule of $ (A_1,\prec_1,\succ_1,\alpha)$ via $f$, that is, for all $(x,y)\in A_1\times A_2$,
$l_{\prec_1}(x)y=f(x)\prec_2 y,$ \ $r_{\prec_1}(x)y=y \succ_2 f(x),$\ $l_{\succ_1}(x)y=f(x)\succ_2 y,$ \ $r_{\succ_1}(x)y=y \succ_2 f(x).$
\end{prop}
\begin{proof}
We prove the axiom \eqref{def:bimodhomdendif:1}, as the others being proved similarly. For any
$x,y\in A_1$ and $z\in A_2$,
\begin{align*}
    l_{\succ_1}([x,y]_1)\beta(z)&=f([x,y]_1)\succ_2\beta(z)=[f(x),f(y)]_2\succ_2\beta(z)\\
   \mbox{(by~\eqref{Hom-Leibniz dendriform1})} &=\beta(f(x))\succ_2(f(y)\succ_2 z)-\beta(f(y))\succ_2(f(x)\succ_2 z)\\
    &=f(\alpha(x))\succ_2(f(y)\succ_2 z)-f(\alpha(y))\succ_2(f(x)\succ_2 z)\\
    &=l_{\succ_1}(\alpha(x))(f(y)\succ_2 z)-l_{\succ_1}(\alpha(y))(f(x)\succ_2 z)\\
    &=l_{\succ_1}(\alpha(x))l_{\succ_1}(y)z-l_{\succ_1}(\alpha(y))l_{\succ_1}(x)z.
\qedhere \end{align*}
\end{proof}
\begin{thm}\label{mamm}
Let $ \mathcal{A}=(A,\prec,\succ,\alpha)$ be a
Hom-Leibniz dendriform algebra, and let
$\mathcal{V}=(l_{\prec_1},r_{\prec_1},l_{\prec},r_{\succ},\beta,V)$ be a bimodule of $ A$. 
Let $\alpha'$ be a Hom-Leibniz dendriform algebra morphism of $ \mathcal{A}$ \textup{(}in particular, $\alpha$ and $\alpha'$ commute\textup{)},
and let $\beta'$ be a linear map on $V$ such that the maps $\beta$ and $\beta'$
commute. Furthermore, suppose that
$$\left\{
   \begin{array}{lllllll}
      \beta'\circ l_\prec=(l_\prec\circ\alpha')\beta',\\
       \beta'\circ r_\prec=(r_\prec\circ\alpha')\beta',\\
      \beta'\circ l_\succ=(l_\succ\circ\alpha')\beta',\\
       \beta'\circ r_\succ=(r_\succ\circ\alpha')\beta',
   \end{array}
 \right.$$
$ \mathcal{A}_{\alpha'}$ is the Hom-Leibniz dendriform algebra
$(A,\prec_{\alpha'}, \succ_{\alpha'},\alpha\alpha')$,
$$
\widetilde{l}_{\prec}=(l_{\prec}\circ\alpha')\beta',~\widetilde{r}_{\prec}=(r_{\prec}\circ\alpha')\beta,~\widetilde{l}_{\succ}
=(l_{\succ}\circ\alpha')\beta',~\widetilde{r}_{\succ}=(r_{\succ}\circ\alpha')\beta'$$ and
$\mathcal{V}_{\beta'}=(\widetilde{l}_{\prec},\widetilde{r}_{\prec},\widetilde{l}_{\succ},\widetilde{r}_{\succ},\beta\beta',V)$,
Then, $\mathcal{V}_{\beta'}$ is a bimodule of $ \mathcal{A}_{\alpha'}$.
\end{thm}
\begin{proof}
We prove only the axiom \eqref{def:bimodhomdendif:1}, as the  others being proved similarly. For any
$x,y\in A$ and $v\in V$,
\begin{align*}
    \widetilde{l_\succ}([x,y]_{\alpha'})\beta\beta'(v)&=\widetilde{l_\succ}([\alpha'(x),\alpha'(y)])\beta\beta'(v)\\
    &=l_\succ([\alpha'^{2}(x),\alpha'^{2}(y)])\beta'\beta\beta'(v)\\
    &=l_\succ([\alpha'^{2}(x),\alpha'^{2}(y)])\beta\beta'^{2}(v)\\
    &=l_\succ(\alpha\alpha'^{2}(x))l_\succ(\alpha'^{2}(y))\beta'^{2}(v)-l_{\succ}(\alpha\alpha'^{2}(y))l_\succ(\alpha'^{2}(x))\beta'^{2}(v)\\
    &=\widetilde{l_\succ}(\alpha\alpha'(x))\widetilde{l_\succ}(y)v-\widetilde{l_\succ}(\alpha\alpha'(y))\widetilde{l_\succ}(x)v.
\qedhere \end{align*}
\end{proof}
\begin{cor}
Let $ \mathcal{A}=(A,\prec,\succ,\alpha)$ be a multiplicative Hom-Leibniz dendriform algebra, and
$(l_{\prec},r_{\prec},l_{\succ},r_{\succ},\beta,V)$ a bimodule of
$ \mathcal{A}$. Then, $\mathcal{V}_{\beta^{q}}$ is a bimodule of $\mathcal{A}_{\alpha^{p}}$ for any
nonnegative integers $p$ and $q$.
\end{cor}
\begin{thm}
Let $\mathcal{A}=(A, \prec_{A},\succ_A, \alpha)$ and $\mathcal{B}=(B,  \prec_{B}, \succ_B,\beta)$
 be two Hom-Leibniz-dendri\-form algebras. If there are linear maps
$ l_{\prec_A},r_{\prec_A} ,l_{\succ_A},r_{\succ_A} : A \rightarrow gl(B),$
and $ l_{\prec_B},r_{\prec_B},l_{\succ_B},r_{\succ_B}: B \rightarrow gl(A)$
such that $( l_{\prec_{ A}},   r_{\prec_{ A}},  l_{\succ_{ A}},  r_{\succ_{ A}}, \beta,  B)$ is a bimodule of $ \mathcal{A},$ and  $(l_{\prec_{ B}},   r_{\prec_{ B}},  l_{\succ_{ B}},  r_{\succ_{ B}}, \alpha,  A)$ is a bimodule  of $ \mathcal{B},$ satisfying the following  relations
for $ x, y \in  A,$ $a, b \in  B $, and $ l_{ A} = l_{\prec_{ A}} +
 l_{\succ_{ A}},$ $r_{ A} =  r_{\prec_{ A}} +  r_{\succ_{ A}},$ $l_{ B} =
 l_{\prec_{ B}} +  l_{\succ_{ B}},$ $r_{ B} =  r_{\prec_{ B}} +  r_{\succ_{ B}}:$
\begin{eqnarray} \label{eq35}
l_{\succ_A}(r_B(a)x)\beta(b)+l_A(x)a\succ_B\beta(b)=l_{\succ_A}(\alpha'(x))(a\succ_B b)\nonumber\\-\beta(a)\succ_B l_{\succ_A}(x)b-r_{\succ_A}(r_{\succ_B}(b)x)\beta(a),
\\
\label{eq36}
r_A(x)a\succ_B\beta(b)+l_{\succ_A}(l_B(a)x)\beta(b)=\beta(a)\succ_B(l_{\succ_A}(x)b)\nonumber\\
+r_{\succ_A}(r_{\succ_B}(b)x)\beta(a)-l_{\succ_A}(\alpha(x))(a\succ_B b),
\\
\label{eq37}
r_{\succ_A}(\alpha(x))([a,b])=\beta(a)\succ_{B}(r_{\succ_A}(x)b)+r_{\succ_A}(l_{\succ_B}(b)x)\beta(a)\nonumber\\
-\beta(b)\succ_B(r_{\succ_A}(x)a)-r_{\succ_A}(l_{\succ_B}(a)x)\beta(b),
\\
\label{eq38}
l_{\succ_A}(\alpha(x))(a\prec_B b)=(l_{\succ_A}(x)a)\prec_B \beta(b)+l_{\prec_A}(r_{\succ_B}(a)x)\beta(b)\nonumber\\
+\beta(a)\prec(l_A(x)b)+r_{\prec_A}(r_B(b)x)\beta(a),
\\
\label{eq39}
\beta(a)\succ_B(l_{\prec_A}(x)b)+r_{\succ_A}(r_{\prec_B}(b)x)\beta(a)=(r_{\succ_A}(x)a)\prec_B\beta(b)\nonumber\\
+l_{\prec_A}(l_{\succ_B}(a)x)\beta(b)+l_{\prec_A}(\alpha(x))([a,b]),
\\
\label{eq40}
\beta(a)\succ_B(r_{\prec_A}(x)b)+r_{\succ_A}(l_{\prec_B}(b)x)\beta(a)=r_{\prec_A}(\alpha(x))(a\succ_B b)\nonumber\\+\beta(b)\prec_B(r_A(x)a)+r_{\prec_A}(l_B(a)x)\beta(b),
\\
\label{eq41}
l_{\prec}(\alpha(x))([a,b])=(l_{\prec_A}(x)a)\prec_B\beta(b)+l_{\prec_A}(r_{\prec_B}(a)x)\beta(b)\nonumber\\
+\beta(a)\succ_B(l_{\prec_A}(x)b)+r_{\succ_A}(r_{\prec_B}(b)x)\beta(a),
\\
\label{eq42}
\beta(a)\prec_B(l_A(x)b)+r_{\prec_A}(r_B(b)x)\beta(a)=(r_{\prec_A}(x)a)\prec_B \beta(b)\nonumber\\+l_{\prec_A}(l_{\prec_B}(a)x)\beta(b)+l_{\succ_A}(\alpha(x))(a\prec_B b),
\\
\label{eq43}
\beta(a)\prec_B(r_A(x)b)+r_{\prec_A}(l_B(b)x)\beta(a)=r_{\prec_A}(\alpha(x))(a\prec_B b)\nonumber\\+\beta(b)\succ_B(r_{\prec_A}(x)a)+r_{\succ_A}(l_{\prec_B}(a)x)\beta(b),
\\
 \label{eq44}
l_{\succ_B}(r_A(x)a)\alpha(y)+l_B(a)x\succ_A\alpha(y)=l_{\succ_B}(\beta'a))(x\succ_A y)\nonumber\\-\alpha(x)\succ_A l_{\succ_B}(a)y-r_{\succ_B}(r_{\succ_A}(y)a)\alpha(x)
\\
\label{eq45}
r_B(a)x\succ_A\beta(y)+l_{\succ_B}(l_A(x)a)\alpha(y)=\alpha(y)\succ_A(l_{\succ_B}(a)y)\nonumber\\
+r_{\succ_B}(r_{\succ_A}(y)a)\beta(x)-l_{\succ_B}(\beta(a))(x\succ_A y)
\\
\label{eq46}
r_{\succ_B}(\beta(a))([x,y])=\alpha(x)\succ_{A}(r_{\succ_B}(a)y)+r_{\succ_B}(l_{\succ_A}(y)a)\alpha(x)\nonumber\\
-\alpha(y)\succ_A(r_{\succ_B}(a)x)-r_{\succ_B}(l_{\succ_A}(x)a)\alpha(y),
\\
\label{eq47}
l_{\succ_B}(\beta(a))(x\prec_A y)=(l_{\succ_B}(a)x)\prec_A \alpha(y)+l_{\prec_B}(r_{\succ_A}(x)a)\alpha(y)\nonumber\\
+\alpha(x)\prec(l_B(a)y)+r_{\prec_B}(r_A(y)a)\alpha(x),
\\
\label{eq48}
\alpha(x)\succ_A(l_{\prec_B}(a)y)+r_{\succ_B}(r_{\prec_A}(y)a)\alpha(x)=(r_{\succ_B}(a)x)\prec_A\alpha(y)\nonumber\\
+l_{\prec_B}(l_{\succ_A}(x)a)\alpha(y)+l_{\prec_B}(\beta(a))([x,y]),
\\
\label{eq49}
\alpha(x)\succ_A(r_{\prec_B}(a)y)+r_{\succ_B}(l_{\prec_A}(y)a)\alpha(x)=r_{\prec_B}(\beta(a))(x\succ_A y)\nonumber\\+\alpha(y)\prec_A(r_B(a)x)+r_{\prec_B}(l_A(x)a)\alpha(y),
\\
\label{eq50}
l_{\prec}(\beta(a))([x,y])=(l_{\prec_B}(a)x)\prec_A\alpha(y)+l_{\prec_B}(r_{\prec_A}(x)a)\alpha(y)\nonumber\\
+\beta(x)\succ_A(l_{\prec_B}(a)y)+r_{\succ_B}(r_{\prec_A}(y)a)\alpha(x),
\\
\label{eq51}
\alpha(x)\prec_A(l_B(a)y)+r_{\prec_B}(r_A(y)a)\alpha(x)=(r_{\prec_B}(a)x)\prec_A \alpha(y)\nonumber\\+l_{\prec_B}(l_{\prec_A}(x)a)\alpha(y)+l_{\succ_B}(\beta(a))(x\prec_A y),
\\
\label{eq52}
\alpha(x)\prec_A(r_B(a)y)+r_{\prec_B}(l_A(y)a)\alpha(x)=r_{\prec_B}(\beta(a))(x\prec_A y)\nonumber\\+\alpha(y)\succ_A(r_{\prec_B}(a)x)+r_{\succ_B}(l_{\prec_A}(x)a)\alpha(y),
\end{eqnarray}
then, there is a Hom-Leibniz dendriform algebra structure on the direct sum $  A \oplus  B $ of the underlying linear spaces of
 $  A $ and $  B $ given for any $ x, y \in  A, a, b \in  B $ by
\begin{eqnarray*}
(x + a) \prec( y + b )&=&(x \prec_{ A} y + r_{\prec_{ B}}(b)x + l_{\prec_{ B}}(a)y) + \cr &&
 (l_{\prec_{ A}}(x)b + r_{\prec_{ A}}(y)a + a \prec_{ B} b), \cr
(x + a) \succ ( y + b )&=&(x \succ_{ A} y + r_{\succ_{ B}}(b)x + l_{\succ_{ B}}(a)y) + \cr &&
(l_{\succ_{ A}}(x)b + r_{\succ_{ A}}(y)a + a \succ_{ B} b),\cr
(\alpha\oplus\beta)(x + a)&=&\alpha(x) + \beta(a).
\end{eqnarray*}
\end{thm}
\begin{proof}
It  is obtained in a similar way as for Theorem \ref{thm:matchedpairs}.
\end{proof}
\begin{defn}
Let $ ( A,   \prec_{ A},  \succ_{ A}, \alpha) $ and $  ( B,   \prec_{ B},  \succ_{ B}, \beta) $
be two Hom-Leibniz-dendri\-form algebras. Suppose there are linear maps
$ l_{  \prec_{ A}}, r_{  \prec_{ A}}, l_{ \succ_{ A}}, r_{ \succ_{ A}} :  A \rightarrow gl( B) $
and $ l_{  \prec_{ B}}, r_{  \prec_{ B}}, l_{ \succ_{ B}}, r_{ \succ_{ B}} :  B \rightarrow gl( A) $
 such that $(l_{  \prec_{ A}}, r_{  \prec_{ A}}, l_{ \succ_{ A}}, r_{ \succ_{ A}}, \beta)$ is a bimodule of $  A, $
and $(l_{  \prec_{ B}}, r_{  \prec_{ B}}, l_{ \succ_{ B}}, r_{ \succ_{ B}}, \alpha)$ is a bimodule of $  B $.
If \eqref{eq35}-\eqref{eq52} are satisfied, then $$( A,  B, l_{  \prec_{ A}},
r_{  \prec_{ A}}, l_{ \succ_{ A}}, r_{ \succ_{ A}}, \beta, l_{  \prec_{ B}}, r_{  \prec_{ B}}, l_{ \succ_{ B}},
 r_{ \succ_{ B}}, \alpha)$$
 is called a matched pair of Hom-Leibniz dendriform algebras.
\end{defn}

\begin{cor}\label{match dendriform-associative algebras}
If $( A,  B, l_{  \prec_{ A}}, r_{  \prec_{ A}}, l_{ \succ_{ A}},
r_{ \succ_{ A}}, \beta,
 l_{  \prec_{ B}}, r_{  \prec_{ B}}, l_{ \succ_{ B}}, r_{ \succ_{ B}},
 \alpha) $ is a matched pair of Hom-Leibniz dendriform algebras, then
 \begin{eqnarray*}( A,  B, l_{  \prec_{ A}} + l_{ \succ_{ A}},
r_{  \prec_{ A}} + r_{ \succ_{ A}},\beta,
l_{  \prec_{ B}} + l_{ \succ_{ B}},  r_{  \prec_{ B}} + r_{ \succ_{ B}}, \alpha )\end{eqnarray*} is a matched pair of the associated
Hom-Leibniz algebras $ ( A, [\cdot,\cdot]_{ A}, \alpha) $ and  $( B, [\cdot,\cdot]_{ B}, \beta)$.
\end{cor}
\begin{proof}
Let $(A, B, l_{\prec_{A}}, r_{\prec_{A}}, l_{\succ_{A}}, r_{\succ_{A}}, \beta,
 l_{\prec_{B}}, r_{\prec_{B}}, l_{\succ_{B}}, r_{\succ_{B}}, \alpha)$ be a matched pair of a Hom-Leibniz dendriform algebras $(A, \prec_{A}, \succ_{A}, \alpha)$ and $(B, \prec_{B}, \succ_{B},\beta)$. In view of Proposition \ref{prop of bimodules}, the linear maps $l_{\prec_{A}} + l_{\succ_{A}}, r_{\prec_{A}} + r_{\succ_{A}}:A\rightarrow gl(B)$ and  $l_{\prec_{B}} + l_{\succ_{B}},  r_{\prec_{B}} + r_{\succ_{B}}:B\rightarrow gl(A)$ are bimodules of the underlying Hom-Leibniz algebras $(A,\ast_A, \alpha)$ and $(B,\ast_B,\beta)$, respectively. Therefore,
 \eqref{match. pair1}, \eqref{match. pair3} are equivalent to \eqref{eq35}-\eqref{eq43} and \eqref{match. pair4}, \eqref{match. pair6} are equivalent to \eqref{eq44}-\eqref{eq52}.
\end{proof}
\begin{defn}
Let $( A, [\cdot,\cdot], \alpha)$ be a Hom-Leibniz algebra, and $(l, r, \beta, V)$ be a bimodule. 
A linear map $ T : V \rightarrow  A $
is called an $ \mathcal{O} $-operator associated to $(l, r, \beta, V)$,  if $ T $ satisfies
\begin{align*}\alpha T &= T\beta, \\ 
[T(u),T(v)] &= T(l(T(u))v + r(T(v))u) \mbox { for all } u, v \in V.
\end{align*}
\end{defn}

\begin{ex}
Let $( A, [\cdot,\cdot], \alpha)$ be a multiplicative Hom-Leibniz algebra. Then,
 the identity map $id$ is an $ \mathcal{O} $-operator associated to the bimodule $(L,0, \alpha)$ or $(0,R, \alpha)$.
\end{ex}
\begin{ex}
Let ${\mathcal A}=( A, [\cdot,\cdot], \alpha)$ be a multiplicative Hom-Leibniz algebra.
A linear map $ K :  A \rightarrow  A $ is called a Rota-Baxter operator on ${\mathcal A}$ of weight zero if $ K $
satisfies
\begin{align*}
\alpha\circ K&=K\circ\alpha,\\
[K(x),K(y)] &= K([K(x), y] + [x, K(y)]) \mbox { for all } x, y \in  A.
\end{align*}
In fact, a Rota-Baxter operator on $  {\mathcal A} $ is just an $ \mathcal{O} $-operator associated to the regular bimodule $(L, R, \alpha)$.
\end{ex}
\begin{thm}
Let $( A, [\cdot,\cdot], \alpha)$ be a Hom-Leibniz algebra, and $(l, r, \beta, V) $ be a bimodule.
Let $ T : V \rightarrow  A $ be an $ \mathcal{O} $-operator associated to $(l, r, \beta, V)$. Then, there exists a Hom-Leibniz dendriform
algebra structure on $ V $ given for all $u, v \in V$ by
\begin{eqnarray*}
 u \prec v =r(T(v))u ,\quad  u \succ v = l(T(u))v. 
\end{eqnarray*}
So, there is an associated Hom-Leibniz algebra structure on $ V $ given by the equation \eqref{associative-dendriform},  and $ T $
is a homomorphism of Hom-Leibniz algebras. Moreover, $ T(V) = \lbrace { T(v) \setminus v \in V }  \rbrace  \subset  A $ is a Hom-Leibniz
subalgebra of $  A, $ and there is an induced Hom-Leibniz dendriform algebra structure on $ T(V) $ given for all $ u, v \in V $ by
\begin{eqnarray}
T(u) \prec T(v) = T(u \prec v), \quad T(u) \succ T(v) = T(u \succ v).\end{eqnarray}
Its corresponding associated Hom-Leibniz algebra structure on $ T(V) $ given by the equation \eqref{associative-dendriform} is
just the Hom-Leibniz subalgebra structure of $  A, $ and $ T $ is a homomorphism of Hom-Leibniz dendriform algebras.
\end{thm}
\begin{proof}
For any $x, y, z\in V,$ 
\begin{eqnarray*}
&&([x,y])\succ \beta(z)=\beta(x)\succ(y\succ z)-\beta(y)\succ(x\succ z)\\
&&=(r(T(y))x+l(T(x)y)\succ \beta(z)-\beta(x)\succ(l(T(y))z)-\beta(y)\succ(l(T(x))z)\\
&&=l(T(r(T(y))x)\beta(z)+l(T(l(T(x))y)\beta(z)-l(T(\beta(x))l(T(y))z-l(T(\beta(y))l(T(x))z\\
&&=l(T(r(T(y))x+l(T(x))y)\beta(z)-l(\alpha T(x))l(T(y))z-l(\alpha T(y))l(T(x))z\\
&&=l([T(x),T(y)])\beta(z)-l([T(x),T(y)])\beta(z)=0.
\end{eqnarray*}
The two other axioms are proved similarly.
\end{proof}
\begin{cor}
Let $T:V\rightarrow A$ be an $\mathcal{O}$-operator on a Hom-Leibniz algebra $(A,[\cdot,\cdot],\alpha)$ with respect to a representation $(l,r, \beta,V)$. Let $[u,v]_C=u\prec v+u\succ v,~$ for all $u,v\in V.$ Then, $T$ is a morphism from the Hom-Leibniz algebra
$(V, [\cdot,\cdot]_C, \beta)$ to $(A,[\cdot,\cdot],\alpha)$.
\end{cor}
\begin{proof}
For all $u,v\in V$,
\begin{equation*}
    T([u,v]_C)=T(u\prec v+u\succ v)=T(u)\prec T(v)+T(u)\succ T(v)=[T(u),T(v)].
\qedhere \end{equation*}
\end{proof}
\begin{cor}\label{dendriform-invertible operator}
Let $( A, [\cdot,\cdot], \alpha)$ be a multiplicative Hom-Leibniz algebra. Then, there is a compatible multiplicative Hom-Leibniz dendriform algebra structure on $ A$
if and only if there exists an invertible $ \mathcal{O} $-operator of $ ( A, [\cdot,\cdot], \alpha)$.
\end{cor}
\begin{proof}
If $ T $ is an invertible $ \mathcal{O}- $operator associated to a bimodule $(l, r, \beta, V)$, then,
the compatible multiplicative Hom-Leibniz dendriform algebra structure on $  A $ is given by
\begin{eqnarray*}
x \succ y = T(l(x)T^{-1}(y)), \quad x \prec y = T(r(y)T^{-1}(x)) \ \mbox { for all } x, y \in  A.
\end{eqnarray*}
Conversely, let $( A, \prec, \succ, \alpha)$ be a Hom-Leibniz dendriform algebra, and $( A,[\cdot,\cdot], \alpha)$ be
 the associated multiplicative Hom-Leibniz algebra. Then, the identity map $
 id$ is an $ \mathcal{O}-$operator
associated to the bimodule $(L_{\succ}, R_{\prec}, \alpha)$ of $( A, [\cdot,\cdot], \alpha)$.
\end{proof}
\begin{prop}
Let $(l_{\prec}, r_{\prec}, l_{\succ}, r_{\succ}, \beta, V)$ be a bimodule of a  Hom-Leibniz-dendri\-form algebra $( A,
\prec, \succ, \alpha)$, and let $( A, [\cdot,\cdot], \alpha)$ be the associated involutive Hom-Leibniz algebra. Let $ l^{\ast}_{\prec},r^{\ast}_{\prec},l^{\ast}_{\succ}, r^{\ast}_{\succ} :  A
\rightarrow gl(V^{\ast}) $ be the linear maps given by
\begin{align*}
\langle l^{\ast}_{\prec}(x)a^{\ast}, y \rangle &=- \langle l_{\prec}(x)y, a^{\ast} \rangle, \quad
\langle r^{\ast}_{\prec}(x)a^{\ast},  y \rangle = -\langle r_{\prec}(x)y, a^{\ast} \rangle, \cr
\langle l^{\ast}_{\succ}(x)a^{\ast},  y \rangle &= -\langle l_{\succ}(x)y, a^{\ast} \rangle, \quad
\langle r^{\ast}_{\succ}(x)a^{\ast}, y \rangle = -\langle r_{\succ}(x)y, a^{\ast} \rangle.
\end{align*}
\begin{enumerate}[label=\upshape{\arabic*)}]
\item $(0,l^{\ast}_{\prec} + l^{\ast}_{\succ}, r^{\ast}_{\prec} + r^{\ast}_{\succ},  0, \beta^{\ast}, V^{\ast}),~(0,l^{\ast}_{\succ}, -l^{\ast}_{\succ}-r^{\ast}_{\prec},0, \beta^{\ast}, V^{\ast})$
 and $(0,l^{\ast}_{\succ}, r^{\ast}_{\prec},  0, \beta^{\ast}, V^{\ast})$ are bimodules
 of $ ( A,[\cdot,\cdot], \alpha); $
\item $(l^{\ast}_{\prec} +  l^{\ast}_{\succ}, r^{\ast}_{\prec} + r^{\ast}_{\succ}, \beta^{\ast}, V^{\ast}),~(l^{\ast}_{\succ}, -l^{\ast}_{\succ}-r^{\ast}_{\prec}, \beta^{\ast}, V^{\ast}),$ and
$(l^{\ast}_{\succ}, r^{\ast}_{\prec}, \beta^{\ast}, V^{\ast})$ are bimodules of $( A, [\cdot,\cdot], \alpha). $
\end{enumerate}
\end{prop}
\begin{ex}
If $( A, \prec, \succ, \alpha)$ is an involutive Hom-Leibniz dendriform algebra, then
\begin{eqnarray*}
(L_{\prec},
R_{\prec}, L_{\succ}, R_{\succ}, \alpha,  A),\ (0,L_{\succ}, R_{\prec},  0, \alpha,  A),\ 
 (0,L_{\prec} + L_{\succ},  R_{\prec} + R_{\succ},  0, \alpha,  A)
\end{eqnarray*}
are bimodules of $( A, \prec, \succ, \alpha)$. On the other hand,
$$(0,L^{\ast}_{\succ},  -L^{\ast}_{\succ}-R^{\ast}_{\prec},  0, \alpha^{\ast},  A^{\ast}),~~ (0, L^{\ast}_{\prec} + L^{\ast}_{\succ},R^{\ast}_{\prec}
+ R^{\ast}_{\succ} ,  0, \alpha^{\ast},  A^{\ast}),$$
are bimodules of $( A, \prec, \succ, \alpha)$  too.
\end{ex}
\begin{thm}\label{dendriform-symplectic theorem}
Let $( A, [\cdot,\cdot], \alpha)$ be an involutive Hom-Leibniz algebra, and let $\omega $ be a non-degenerate skew-symmetric bilinear form on $ A$ such that the following identity \textup{(}invariance condition\textup{)} is satisfied for all $x, y, z\in A,$
\begin{eqnarray}
\omega([\alpha(x),\alpha(y)], \alpha(z)) + \omega([\alpha(y),\alpha(z)], \alpha(x)) + \omega([\alpha(z),\alpha(x)], \alpha(y))=0.
\end{eqnarray}
Then,  there exists a compatible Hom-Leibniz dendriform algebra structure $ \prec, \succ $ on $(  A, \alpha)$ given by
\begin{eqnarray} \label{dendriform-symplectic equation}
\omega(x \prec y, z) = \omega(y, \{z , x\}), \ \ \ \omega(x \succ y, z) = \omega(x, \{y , z\}) \mbox { for all } x, y \in  A.
\end{eqnarray}
\end{thm}
\begin{proof}
Define a linear map $ T :  A \rightarrow  A^{\ast} $ by $ \langle T(x), y \rangle = \omega(x, y) $ for all $ x, y \in  A $.
Then, $ T $ is invertible and $ T^{-1} $ is an $  \mathcal{O}-$operator of the involutive Hom-Leibniz algebra $( A, [\cdot,\cdot], \alpha)$ associated to the
bimodule $(L^{\ast}, -L^{\ast}-R^{\ast}, \alpha^{\ast})$. By Corollary \ref{dendriform-invertible operator}, there is a compatible Hom-Leibniz dendriform algebra structure
$ \prec, \succ $ on $( A, [\cdot,\cdot])$ given by
\begin{eqnarray*}
x \prec y = T^{-1}L^{\ast}(x)T(y), \ \ x \succ y =  T^{-1}(-L^{\ast}-R^{\ast})(y)T(x)
\end{eqnarray*}
for all $ x, y \in  A $, which gives exactly \eqref{dendriform-symplectic equation}.
\end{proof}

\begin{thm}
Let $( A, \prec_{ A}, \succ_{ A}, \alpha)$ be an involutive Hom-Leibniz dendriform algebra, and $( A, [\cdot,\cdot], \alpha)$ be
the associated involutive Hom-Leibniz algebra. Suppose  there is an involutive Hom-Leibniz dendriform algebra  structure
$"\prec_{ A^{\ast}}, \succ_{ A^{\ast}}, \alpha^{\ast}"$ on its dual space $ A^{\ast},$ and
 $ ( A^{\ast}, \{\cdot,\cdot\}, \alpha^{\ast})$ is the associated Hom-Leibniz algebra. Then, is a standard Manin triple of involutive Hom-Leibniz algebras associated to $( A, [\cdot,\cdot], \alpha)$ and $( A^{\ast}, \{\cdot,\cdot\}, \alpha^{\ast})$
if and only if the octuple
$( A,  A^{\ast}, L^{\ast}_{\succ_{ A}}, -L^{\ast}_{\succ_{ A}}-R^{\ast}_{\prec_{ A}}, \alpha^{\ast}, L^{\ast}_{\succ_{ A^{\ast}}},
-L^{\ast}_{\succ_{ A^{\ast}}}-R^{\ast}_{\prec_{ A^{\ast}}}, \alpha)$ is a matched pair of Hom-Leibniz algebras.
\end{thm}
\begin{proof}
The conclusion can be obtained by a similar proof as in Theorem \ref{Frobenius theorem}.
Then, if  $( A,  A^{\ast}, L^{\ast}_{\succ_{ A}}, -L^{\ast}_{\succ_{ A}}-R^{\ast}_{\prec_{ A}}, \alpha^{\ast}, L^{\ast}_{\succ_{ A^{\ast}}},
-L^{\ast}_{\succ_{ A^{\ast}}}-R^{\ast}_{\prec_{ A^{\ast}}}, \alpha)$ is a matched pair of the involutive Hom-Leibniz algebras $( A, [\cdot,\cdot], \alpha)$ and $( A^{\ast}, \{\cdot,\cdot\}, \alpha^{\ast})$,
it is straightforward to show that the bilinear form \eqref{quadreq} is an  $(\alpha_{1}\oplus\alpha^{\ast}_{1})$-invariant on the Hom-Leibniz algebra
$  A \bowtie^{ L^{\ast}_{\succ_{ A}}, -L^{\ast}_{\succ_{ A}}-R^{\ast}_{\prec_{ A}}, \alpha^{\ast}}_{L^{\ast}_{\succ_{ A^{\ast}}},
-L^{\ast}_{\succ_{ A^{\ast}}}-R^{\ast}_{\prec_{ A^{\ast}}}, \alpha}  A^{\ast} $ given by
\begin{eqnarray*}
(x + a^{\ast})\ast_{ A \oplus  A^{\ast}} (y + b^{\ast}) &=&([x , y] + L^{\ast}_{\succ_{ A^{\ast}}}(a^{\ast})y
+(- L^{\ast}_{\succ_{ A^{\ast}}}- R^{\ast}_{\prec_{ A^{\ast}}})(b^{\ast})x)\cr
&&+ (\{a^{\ast} , b^{\ast}\} + L^{\ast}_{\succ_{ A}}(x)b^{\ast} +(- L^{\ast}_{\succ_{ A}}- R^{\ast}_{\prec_{ A}})(y)a^{\ast}).
\end{eqnarray*}
In fact, we have
\begin{eqnarray*}
&&\omega(\alpha(x) + \alpha^{\ast}(a^{\ast}), \alpha(y) + \alpha^{\ast}(b^{\ast})]', \alpha(z) + \alpha^{\ast}(c^{\ast}))\cr
&& \quad + \omega(\alpha(y) + \alpha^{\ast}(b^{\ast}), \alpha(z) + \alpha^{\ast}(c^{\ast})]', \alpha(x) + \alpha^{\ast}(a^{\ast})] \cr
&& \quad +\omega(\alpha(z) + \alpha^{\ast}(c^{\ast}), \alpha(x) + \alpha^{\ast}(a^{\ast})]', \alpha(y) + \alpha^{\ast}(b^{\ast})) \cr
&&= -\langle [\alpha(x) , \alpha(y)] + L^{\ast}_{\succ_{ A^{\ast}}}(\alpha^{\ast}(a^{\ast}))\alpha(y) -
(L^{\ast}_{\prec_{ A^{\ast}}}+L^{\ast}_{\prec_{ A^{\ast}}})(\alpha^{\ast}(b^{\ast}))\alpha(x), \alpha^{\ast}(c^{\ast})\rangle \cr
&& \quad + \langle \{\alpha^{\ast}(a^{\ast})
,\alpha^{\ast}(b^{\ast})\} +L^{\ast}_{\succ_{ A}}(\alpha(x))\alpha^{\ast}(b^{\ast}) -( L^{\ast}_{\prec_{ A}}+R^{\ast}_{\succ_{ A}})(\alpha(y))\alpha^{\ast}(a^{\ast}), \alpha(z)\rangle \cr
&&\quad-\langle [\alpha(y) , \alpha(z)] + L^{\ast}_{\succ_{ A^{\ast}}}(\alpha^{\ast}(a^{\ast}))\alpha(z) -
(L^{\ast}_{\prec_{ A^{\ast}}}+L^{\ast}_{\prec_{ A^{\ast}}})(\alpha^{\ast}(c^{\ast}))\alpha(y), \alpha^{\ast}(a^{\ast})\rangle \cr
&& \quad + \langle \{\alpha^{\ast}(b^{\ast})
,\alpha^{\ast}(c^{\ast})\} +L^{\ast}_{\succ_{ A}}(\alpha(y))\alpha^{\ast}(c^{\ast})
 -( L^{\ast}_{\prec_{ A}}+R^{\ast}_{\succ_{ A}})(\alpha(z))\alpha^{\ast}(b^{\ast}), \alpha(x)\rangle \cr
&&\quad-\langle [\alpha(x) , \alpha(y)] + L^{\ast}_{\succ_{ A^{\ast}}}(\alpha^{\ast}(a^{\ast}))\alpha(y) -
(L^{\ast}_{\prec_{ A^{\ast}}}+L^{\ast}_{\prec_{ A^{\ast}}})(\alpha^{\ast}(b^{\ast}))\alpha(x), \alpha^{\ast}(c^{\ast})\rangle \cr
&&\quad + \langle \{\alpha^{\ast}(a^{\ast})
,\alpha^{\ast}(b^{\ast})\} +L^{\ast}_{\succ_{ A}}(\alpha(x))\alpha^{\ast}(b^{\ast})
 -( L^{\ast}_{\prec_{ A}}+R^{\ast}_{\succ_{ A}})(\alpha(y))\alpha^{\ast}(a^{\ast}), \alpha(z)\rangle \cr
 &&=-\langle[\alpha(x),\alpha(y)],\alpha^{\ast}(c^{\ast})\rangle
+\langle\alpha(y),\alpha^{\ast}(a^{\ast})\succ \alpha^{\ast}(c^{\ast})-\langle\alpha(x),\alpha^{\ast}(b^{\ast})\succ \alpha^{\ast}(c^{\ast})\rangle\cr&&\quad-\langle\alpha(x),\alpha^{\ast}(c^{\ast})\prec \alpha^{\ast}(b^{\ast})\rangle+\langle\{\alpha^{\ast}(a^{\ast}),\alpha^{\ast}(b^{\ast})\},\alpha(z)\rangle-\langle \alpha(x)\succ\alpha(z),\alpha^{\ast}(b^{\ast})\rangle\cr&&\quad+\langle\alpha(y)\succ\alpha(z),\alpha^{\ast}(a^{\ast})\rangle+\langle\alpha(z)\prec\alpha(y),\alpha^{\ast}(a^{\ast})\rangle-\langle[\alpha(y),\alpha(z)],\alpha^{\ast}(a^{\ast})\rangle
\cr
&&\quad+\langle\alpha(z),\alpha^{\ast}(b^{\ast})\succ \alpha^{\ast}(a^{\ast})-\langle\alpha(y),\alpha^{\ast}(c^{\ast})\succ \alpha^{\ast}(a^{\ast})\rangle-\langle\alpha(y),\alpha^{\ast}(a^{\ast})\prec \alpha^{\ast}(c^{\ast})\rangle\cr&&\quad+\langle\{\alpha^{\ast}(b^{\ast}),\alpha^{\ast}(c^{\ast})\},\alpha(x)\rangle-\langle \alpha(y)\succ\alpha(x),\alpha^{\ast}(c^{\ast})\rangle+\langle\alpha(z)\succ\alpha(x),\alpha^{\ast}(b^{\ast})\rangle\cr&&\quad+\langle\alpha(x)\prec\alpha(z),\alpha^{\ast}(b^{\ast})\rangle-\langle[\alpha(z),\alpha(x)],\alpha^{\ast}(b^{\ast})\rangle
+\langle\alpha(x),\alpha^{\ast}(c^{\ast})\succ \alpha^{\ast}(b^{\ast})\cr
&&\quad-\langle\alpha(z),\alpha^{\ast}(a^{\ast})\succ \alpha^{\ast}(b^{\ast})\rangle-\langle\alpha(y),\alpha^{\ast}(b^{\ast})\prec \alpha^{\ast}(a^{\ast})\rangle+\langle\{\alpha^{\ast}(c^{\ast}),\alpha^{\ast}(a^{\ast})\},\alpha(y)\rangle\cr&&\quad-\langle \alpha(z)\succ\alpha(y),\alpha^{\ast}(a^{\ast})\rangle +\langle\alpha(x)\succ\alpha(y),\alpha^{\ast}(c^{\ast})\rangle+\langle\alpha(y)\prec\alpha(x),\alpha^{\ast}(c^{\ast})\rangle =0.
\end{eqnarray*}
Conversely, if there exists a standard Manin triple of involutive Hom-Leibniz algebras associated to $( A, [\cdot,\cdot], \alpha)$ and $( A^{\ast}, \{\cdot,\cdot\}, \alpha^{\ast})$, then, the octuple
$( A,  A^{\ast}, L^{\ast}_{\succ_{ A}}, -L^{\ast}_{\succ_{ A}}-R^{\ast}_{\prec_{ A}}, \alpha^{\ast}, L^{\ast}_{\succ_{ A^{\ast}}},
-L^{\ast}_{\succ_{ A^{\ast}}}-R^{\ast}_{\prec_{ A^{\ast}}}, \alpha)$
is a matched pair of the involutive Hom-Leibniz algebras  given by
\begin{align} \label{match. pairr1}
&L^{\ast}_{\succ_{A}}(\alpha(x))(\{a,b\})=L^{\ast}_{\succ_{A}}((-L^{\ast}_{\succ_{ A^{\ast}}}-R^{\ast}_{\prec_{ A^{\ast}}})(a)x)\beta(b)\nonumber\\
&\quad +(-L^{\ast}_{\succ_{ A}}-R^{\ast}_{\prec_{ A}})((-L^{\ast}_{\succ_{ A^{\ast}}}-R^{\ast}_{\prec_{ A^{\ast}}})(b)x)\beta(a)+\{L^{\ast}_{\succ_{A}}(x)a,\beta(b)\}\nonumber\\
&\quad +\{\beta(a),L^{\ast}_{\succ_{A}}(x)b\},\\*[0,2cm]
\label{match. pairr2}
&\{\beta(a),L^{\ast}_{\succ_{A}}(x)b\}+(-L^{\ast}_{\succ_{ A}}-R^{\ast}_{\prec_{ A}})((-L^{\ast}_{\succ_{ A^{\ast}}}-R^{\ast}_{\prec_{ A^{\ast}}})(b)x)\beta(a)\nonumber\\
&\quad =\{(-L^{\ast}_{\succ_{ A}}-R^{\ast}_{\prec_{ A}})(x)a,\beta(b)\}+L^{\ast}_{\succ_{A}}(L^{\ast}_{\succ_{A^{\ast}}}(a)x)\beta(b)\nonumber\\
&\quad +L^{\ast}_{\succ_{A}}(\alpha(x))([a,b]),
\\*[0,2cm]
\label{match. pairr3}
&\{\beta(a),(-L^{\ast}_{\succ_{ A}}-R^{\ast}_{\prec_{ A}})(x)b\}+(-L^{\ast}_{\succ_{ A}}-R^{\ast}_{\prec_{ A}})(L^{\ast}_{\succ_{A^{\ast}}}(b)x)\beta(a)\nonumber\\
&\quad =(-L^{\ast}_{\succ_{ A}}-R^{\ast}_{\prec_{ A}})(\alpha(x))([a,b])+\{\beta(b),(-L^{\ast}_{\succ_{ A}}-R^{\ast}_{\prec_{ A}})(x)a\}\nonumber\\
&\quad +(-L^{\ast}_{\succ_{ A}}-R^{\ast}_{\prec_{ A}})(L^{\ast}_{\succ_{A^{\ast}}}(a)x)\beta(b)
\\*[0,2cm]
\label{match. pairr4}
&L^{\ast}_{\succ_{A^{\ast}}}(\beta(a))(\{x,y\})=L^{\ast}_{\succ_{A^{\ast}}}((-L^{\ast}_{\succ_{ A}}-R^{\ast}_{\prec_{ A}})(x)a)\alpha(y)\nonumber\\
&\quad +(-L^{\ast}_{\succ_{ A^{\ast}}}-R^{\ast}_{\prec_{ A^{\ast}}})((-L^{\ast}_{\succ_{ A}}-R^{\ast}_{\prec_{ A}})(x)a)\alpha(x)+\{L^{\ast}_{\succ_{A^{\ast}}}(a)x,\alpha(y)\}\nonumber\\
&\quad +\{\alpha(x),L^{\ast}_{\succ_{A^{\ast}}}(a)y\},
\\*[0,2cm]
\label{match. pairr5}
&\{\alpha(x),L^{\ast}_{\succ_{A^{\ast}}}(a)y\}+(-L^{\ast}_{\succ_{ A^{\ast}}}-R^{\ast}_{\prec_{ A^{\ast}}})((-L^{\ast}_{\succ_{ A}}-R^{\ast}_{\prec_{ A}})(y)a)\alpha(x)\nonumber\\
&\quad =\{(-L^{\ast}_{\succ_{ A^{\ast}}}-R^{\ast}_{\prec_{ A^{\ast}}})(a)x,\alpha(y)\}
+L^{\ast}_{\succ_{A^{\ast}}}(L^{\ast}_{\succ_{A}}(x)a)\alpha(y)\nonumber\\
&\quad +L^{\ast}_{\succ_{A^{\ast}}}(\beta(a))([x,y]),
\\*[0,2cm]
\label{match. pairr6}
&\{\alpha(x),(-L^{\ast}_{\succ_{ A^{\ast}}}-R^{\ast}_{\prec_{ A^{\ast}}})(a)y\}+(-L^{\ast}_{\succ_{ A^{\ast}}}-R^{\ast}_{\prec_{ A}})(L^{\ast}_{\succ_{A^{\ast}}}(y)a)\alpha(x)\nonumber\\
&\quad =(-L^{\ast}_{\succ_{ A^{\ast}}}-R^{\ast}_{\prec_{ A^{\ast}}})(\beta(a))([x,y])+\{\alpha(y),(-L^{\ast}_{\succ_{ A^{\ast}}}-R^{\ast}_{\prec_{ A^{\ast}}})(a)x\}\nonumber\\
&\quad +(-L^{\ast}_{\succ_{ A^{\ast}}}-R^{\ast}_{\prec_{ A^{\ast}}})(L^{\ast}_{\succ_{A}}(x)a)\alpha(y),
\end{align}
which completes the proof. \end{proof}
\begin{cor} Let $ ( A, \prec, \succ, \alpha) $ be an involutive Hom-Leibniz dendriform algebra, and the triple                            $(L^{\ast}_{\succ}, -L^{\ast}_{\succ} -R^{\ast}_{\prec}, \alpha^{\ast})$ be
the bimodule of the associated involutive Hom-Leibniz algebra $( A, [\cdot,\cdot], \alpha)$. Then,
 $ (T( A) =  A \times _{L^{\ast}_{\succ},-L^{\ast}_{\succ} -R^{\ast}_{\prec}, \alpha, \alpha^{\ast}}  A^{\ast}, \omega)$ is a standard Manin triple of involutive Hom-Leibniz algebras.
\end{cor}

\section{BiHom-Leibniz algebra}
\label{sec:bihomleibniz}
\begin{defn}
A BiHom-Leibniz algebra is a quadruple $(A,[\cdot,\cdot],\alpha,\beta)$ consisting of a linear space $A$, bilinear map $[.,.]:A\times A\rightarrow A$, linear maps $\alpha,\beta:A\rightarrow A$ satisfying, for all $x,y,z\in A$,
\begin{alignat}{2}
    &\alpha\circ\beta=\beta\circ\alpha, & \quad \text{(commutativity)}\\
    &[\alpha\beta(x),[y,z]]=[[\beta(x),y],\beta(z)]+[\beta(y),[\alpha(x),z]]. & \quad \text{(BiHom-Leibniz identity)} \label{cond BiHom-Leibn}
\end{alignat}
If, in addition, $\alpha$ and $\beta$ satisfies the multiplicativity properties
\begin{equation}
    \alpha([x,y])=[\alpha(x),\alpha(y)],~  \beta([x,y])=[\beta(x),\beta(y)],
\end{equation}
then $(A, [\cdot, \cdot], \alpha,\beta)$ is said to be multiplicative.
\end{defn}

\begin{defn}
A BiHom-module is a triple $(M,\alpha,\beta )$, where $M$ is a $\mathbb{K}$-linear space,
and $\alpha,\beta : M\rightarrow M$ are two linear maps.
\end{defn}
\begin{defn}
Let $(A,[\cdot,\cdot],\alpha_1,\alpha_2)$ be a BiHom-Leibniz algebra and let $(V,\beta_1,\beta_2)$ be a BiHom-module. Let $l,r:A\rightarrow gl(V)$ be two linear maps. The quintuple $(l,r,\beta_1,\beta_2,V)$ is called a bimodule of $A$ if for all $ x, y \in  A, v \in V $,
\begin{eqnarray}
\label{Cond.biHom-bimod 1}
l(\alpha_1\alpha_2(x))l(y)v&=&l([\alpha_2(x),y])\beta_2(v)+l(\alpha_2(y))l(\alpha_1(x))v, \\
\label{Cond.biHom-bimod 2}
l(\alpha_1\alpha_2(x))r(y)v&=&r(\alpha_2(y))l(\alpha_2(x))v+r([\alpha_1(x),y])\beta_2(v),\\
\label{Cond.biHom-bimod 3}
r([x,y])\beta_1\beta_2(v)&=&r(\alpha_2(y))r(x)\beta_2(v)+l(\alpha_2(x))r(y)\beta_1(v),\\
\label{Cond.biHom-bimod 4}
\beta_1(l(x)v) &=&l(\alpha_1(x))\beta_1(v),\\
\label{Cond.biHom-bimod 5}
\beta_1(r(x)v) &=& r(\alpha_1(x))\beta_1(v),\\
\label{Cond.biHom-bimod 6}
\beta_2(l(x)v) &=& l(\alpha_2(x))\beta_2(v),\\
\label{Cond.biHom-bimod 7}
\beta_2(r(x)v) &=& l(\alpha_2(x))\beta_2(v).\end{eqnarray}
\end{defn}
\begin{prop}
Let $(l, r, \beta_{1}, \beta_{2}, V)$ be a bimodule of a BiHom-Leibniz algebra\\ $( A, [\cdot,\cdot], \alpha_{1}, \alpha_{2})$. Then, the direct sum $ A \oplus V$ of linear spaces is turned into a BiHom-Leibniz algebra  by defining multiplication in $ A \oplus V $ for all $ x_{1}, x_{2} \in   A, v_{1}, v_{2} \in V$ by
\begin{eqnarray*}
&&[x_{1} + v_{1} , x_{2} + v_{2}]'=[x_{1} , x_{2}] + (l(x_{1})v_{2} + r(x_{2})v_{1}),\cr
&&(\alpha_{1}\oplus\beta_{1})(x_{1} + v_{1})=\alpha_{1}(x_{1}) + \beta_{1}(v_{1}), (\alpha_{2}\oplus\beta_{2})(x_{1} + v_{1})=\alpha_{2}(x_{1}) + \beta_{2}(v_{1}).
\end{eqnarray*}
\end{prop}
\begin{proof}
For all $v_{1}, v_{2}, v_{3}\in V$ and $x_{1}, x_{2},x_{3}\in  A$,
\begin{align*}\label{condit. du BiHomBimod.}
&[(\alpha_1+\beta_1)(\alpha_2+\beta_2)(x_1+v_1),[x_2+v_2,x_3+v_3]']'\\
&\quad\quad-[[(\alpha_2+\beta_2)(x_1+v_1),x_2+v_2]',(\alpha_2+\beta_2)(x_3+v_3)]'\\
&\quad\quad-[(\alpha_2+\beta_2)(x_2+v_2),[(\alpha_1+\beta_1)(x_1+v_1),x_3+v_3]']'\\
&\quad=[\alpha_1\alpha_2(x_1)+\beta_1\beta_2(v_1),[x_2,x_3]+l(x_2)v_3+r(x_3)v_2]'\\
&\quad\quad -[[\alpha_2(x_1),x_2]+l((\alpha_2(x_1))v_2+r(x_2)\beta_2(v_1),\alpha_2(x_3)+\beta_2(v_3)]'\\
&\quad\quad-[\alpha_2(x_2)+\beta_2(v_2),[\alpha_1(x_1),x_3]+l(\alpha_1(x_1))v_3+r(x_3)\beta_1(v_1)]\\
&\quad=[\alpha_1\alpha_2(x_1),[x_2,x_3]]+l(\alpha\alpha_2(x_1))l(x_2)v_3+l(\alpha_1\alpha_2(x_1))r(x_3)v_2\\
&\quad\quad+r([x_2,x_3])\beta_1\beta_2(v_1)-[[\alpha_2(x_1),x_2],\alpha_2(x_3)]-l([\alpha_2(x_1),x_2])r(x_2)\beta_2(v_1)\\
&\quad\quad-r(\alpha_2(x_3))r(x_2)\beta_2(v_1)
-[\alpha_2(x_2),[\alpha_1(x_1),x_3]]-l(\alpha_2(x_2))l(\alpha_1(x_1))v_3\\
&\quad\quad-l(\alpha_2(x_2))r(x_3)\beta_1(v_1)-r([\alpha_1(x_1),x_3])\beta_2(v_2).
\end{align*}
Then, by \eqref{cond BiHom-Leibn} and \eqref{Cond.biHom-bimod 1}-\eqref{Cond.biHom-bimod 3},
\begin{multline*}
[(\alpha_1+\beta_1)(\alpha_2+\beta_2)(x_1+v_1),[x_2+v_2,x_3+v_3]']' \\
-[[(\alpha_2+\beta_2)(x_1+v_1),x_2+v_2]',(\alpha_2+\beta_2)(x_3+v_3)]' \\
-[(\alpha_2+\beta_2)(x_2+v_2),[(\alpha_1+\beta_1)(x_1+v_1),x_3+v_3]']'=0,
\end{multline*}
which completes the proof.
\end{proof}
We denote such a BiHom-Leibniz algebra by $( A \oplus V, [\cdot,\cdot]', \alpha_{1} + \beta_{1}, \alpha_{2} + \beta_{2}),$
or $ A \times_{l, r, \alpha_{1}, \alpha_{2}, \beta_{1}, \beta_{2}} V.$

\begin{ex}\label{bmodex}
Let $( A, [\cdot,\cdot], \alpha, \beta)$ be a multiplicative BiHom-Leibniz algebra. Then the quadruples
$ (L, 0, \alpha, \beta)$ and $(L, R, \alpha, \beta) $ are bimodules of $( A, [\cdot,\cdot], \alpha, \beta)$.
\end{ex}
\begin{prop}
Let $\mathcal{V}=(l, r, \beta_{1}, \beta_{2}, V)$ be bimodule of a multiplicative BiHom-Leibniz algebra  $( A, [\cdot,\cdot], \alpha_{1}, \alpha_{2})$. Then, the quintuples $\mathcal{V}_1=(l\circ\alpha^{n}_{1}, r\circ\alpha^{n}_{1}, \beta_{1}, \beta_{2}, V)$ and $\mathcal{V}_2=(l\circ\alpha^{n}_{2}, r\circ\alpha^{n}_{2}, \beta_{1}, \beta_{2}, V)$ are  bimodules of $( A, [\cdot,\cdot], \alpha_{1}, \alpha_{2}),$
for any non-negative integer $n$, or if moreover $\alpha_1$ and $\alpha_2$ are invertible, then for any integer $n$.
\end{prop}
\begin{proof}
For all $x,y\in A,$ $v\in V$ and non-negative integer $n$, the defining equalities \eqref{Cond.biHom-bimod 1}, \eqref{Cond.biHom-bimod 2} and \eqref{Cond.biHom-bimod 3} in $\mathcal{V}_{1}$
are proved as follows:
\begin{align*}
(l\circ\alpha_{1}^{n})(\alpha_1\alpha_2(x))(l\circ\alpha_1^{n})(y)v=&l(\alpha_1^{n+1}\alpha_2(x))l(\alpha_1^{n}(y))v\\
=&l([\alpha_1^{n}\alpha_2(x),\alpha_1^{n}(y)]\beta_2(v)+l(\alpha_2\alpha_1^{n}(y))l(\alpha_1^{n+1}(x))v\\
=&(l\circ\alpha_1^{n})([\alpha_2(x),y])\beta_2(v)\\&+(l\circ\alpha_1^{n})(\alpha_2(y))(l\circ\alpha_1^{n})(\alpha_1(x)),\\*[0,2cm]
(l\circ\alpha_1^{n})(\alpha_1\alpha_2(x))(r\circ\alpha_1^{n})(y)v=&l(\alpha_1^{n+1}\alpha_2(x))r(\alpha_1^{n}(y))v\\
=&r(\alpha_2\alpha_1^{n}(y))l(\alpha_1^{n}\alpha_2(x))v+r([\alpha_1^{n+1}(x),\alpha_1^{n}(y)]\beta_2(v)\\
=&(r\circ\alpha_1^{n})(\alpha_2(y))(l\circ\alpha_1^{n})(\alpha_2(x))v\\&+(r\circ\alpha_1^{n})([\alpha_1(x),y])\beta_2(v),\\*[0,2cm]
(r\circ\alpha_1^{n})([x,y])\beta_1\beta_2(v)=&r([\alpha_1^{n}(x),\alpha_1^{n}(y)])\beta_1\beta_2(v)\\
=&r(\alpha_2\alpha_1^{n}(y))r(\alpha_1^{n}(x))\beta_2(v)+l(\alpha_1^{n}\alpha_2(x))r(\alpha_1^{n}(y))\beta_2(v)\\
=&(r\circ\alpha_1^{n})(\alpha_2(y))(r\circ\alpha_1^{n})(x)\beta_2(v)\\
&+(l\circ\alpha_1^{n})(\alpha_2(x))(r\circ\alpha_1^{n})(y)\beta_2(v).
\end{align*}
Similarly it can be shown that for any non-negative integer $n$, $\mathcal{V}_2$ is a bimodule of $( A, [\cdot,\cdot], \alpha_{1}, \alpha_{2})$.
If moreover $\alpha_1$ and $\alpha_2$ are invertible, then in the same way it follows that
$\mathcal{V}_1=(l\circ\alpha^{n}_{1}, r\circ\alpha^{n}_{1}, \beta_{1}, \beta_{2}, V)$ and $\mathcal{V}_2=(l\circ\alpha^{n}_{2}, r\circ\alpha^{n}_{2}, \beta_{1}, \beta_{2}, V)$ are bimodules of
$( A, [\cdot,\cdot], \alpha_{1}, \alpha_{2}),$
for any integer $n,$ since $\alpha_1^{-1}$ and $\alpha_2^{-1}$ are BiHom-Leibniz algebras morphisms, and $\alpha_{i}^{j}\circ\alpha_{k}^{l}=\alpha_{k}^{l}\circ\alpha_{i}^{j},$
for all $(i,k)\in\{1,2\}^{2}$ and $(j,l)\in\{1,-1\}^{2}.$
\end{proof}
\begin{ex}
Let $( A,[\cdot,\cdot], \alpha_{1}, \alpha_{2})$ be a multiplicative BiHom-Leibniz algebra, then, $(L\circ\alpha^{n}_{1}, R\circ\alpha^{n}_{1}, \alpha_{1}, \alpha_{2},  A)$ and $(L\circ\alpha^{n}_{2}, R\circ\alpha^{n}_{2}, \alpha_{1}, \alpha_{2},  A)$ are bimodules of $( A,[\cdot,\cdot], \alpha_{1}, \alpha_{2})$ for any non-negative integer $n$, or for any integer $n$ if $\alpha_1$ and $\alpha_2$ are invertible.
\end{ex}

\begin{thm}
Let $( A, [\cdot,\cdot], \alpha_{1}, \alpha_{2}) $ and $(  B, \{\cdot,\cdot\}, \beta_{1}, \beta_{2})$ be two biHom-Leibniz
algebras. Suppose there exist linear maps $l_{  A}, r_{  A} :  A
\rightarrow gl(  B),$ and $l_{   B}, r_{   B} :    B \rightarrow
gl(  A) $ such that $ (l_{  A}, r_{  A},  \beta_{1}, \beta_{2},   B)$ is a bimodule of $ A, $ and $ (l_{   B}, r_{   B}, \alpha_{1}, \alpha_{2},  A)$ is a bimodule of $  B,$ satisfying, for any $ x, y \in   A, a,b \in B$ the following conditions:
\begin{eqnarray} \label{bimatch. pair1}
&l_{A}(\alpha_1\alpha_2(x))(\{a,b\})=l_A(r_B(a)\alpha_2(x))\beta_2(b)+r_A(r_B(b)\alpha_1(x))\beta_2(a)\nonumber\\
&+\{l_A(\alpha_2(x))a,\beta_2(b)\}+\{\beta_2(a),l_A(\alpha_1(x))b\},\\
\label{bimatch. pair2}
&\{\beta_1\beta_2(a),l_A(x)b\}+r_A(r_B(b)x)\beta_1\beta_2(a)=\{r_A(x)\beta_2(a),\beta_2(b)\}\nonumber\\
&+l_A(l_B(\beta_2(a))x)\beta_2(b)+l_A(\alpha_2(x))([\beta_1(a),b]),\\
\label{bimatch. pair3}
&\{\beta_1\beta_2(a),r_A(x)b\}+r_A(l_B(b)x)\beta_1\beta_2(a)=r_A(\alpha_2(x))(\{\beta_2(a),b\})\nonumber\\
&+\{\beta_2(b),r_A(x)\beta_1(a)\}+r_A(l_B(\beta_1(a))x)\beta_2(b),\\
 \label{bimatch. pair4}
&l_{B}(\beta_1\beta_2(a))([x,y])=l_B(r_A(x)\beta_2(a))\alpha_2(y)+r_B(r_A(y)\beta_1(a))\alpha_2(x)\nonumber\\
&+[l_B(\beta_2(a))x,\alpha_2(y)]+[\alpha_2(x),l_B(\beta_1(a))y],\\
\label{bimatch. pair5}
&[\alpha_1\alpha_2(x),l_B(a)y]+r_B(r_A(y)a)\alpha_1\alpha_2(x)=[r_B(a)\alpha_2(x),\alpha_2(y)]\nonumber\\
&+l_B(l_A(\beta_2(x))a)\alpha_2(y)+l_B(\beta_2(a))([\alpha_1(x),y]),\\
\label{bimatch. pair6}
&[\alpha_1\alpha_2(x),r_B(a)y]+r_B(l_A(y)a)\alpha_1\alpha_2(x)=r_B(\beta_2(a))([\alpha_2(x),y])\nonumber\\
&+[\alpha_2(y),r_B(a)\alpha_1(x)]+r_B(l_A(\alpha_1(x))a)\alpha_2(y).
\end{eqnarray}
 Then, there is a BiHom-Leibniz algebra
 structure on the direct sum $ A\oplus  B$ of
the underlying linear spaces of $ A$ and $  B$ given for all $ x, y \in   A, a,b \in B$ by
\begin{equation}\label{bimatch. pair product}
\begin{array}{l}
[x + a,y + b]' = ([x, y] + l_{B}(a)y + r_{ B}(b)x) + (\{a, b\} +  l_{  A}(x)b +  r_{  A}(y)a),\cr
(\alpha_{1}\oplus\beta_{1})(x + a) = \alpha_{1}(x) + \beta_{1}(a), (\alpha_{2}\oplus\beta_{2})(x + a)=\alpha_{2}(x) + \beta_{2}(a).
\end{array}
\end{equation}
\end{thm}
We denote this BiHom-Leibniz algebra by
$( A\bowtie   B,[\cdot,\cdot]' , \alpha_{1} + \beta_{1}, \alpha_{2} + \beta_{2})$ or $ A \bowtie^{l_{A}, r_{A}, \beta_{1}, \beta_{2}}_{l_{B}, r_{B}, \alpha_{1}, \alpha_{2}}B.$

\begin{defn}
Let $ ( A,[\cdot,\cdot], \alpha_{1}, \alpha_{2})$ and $(B, \{\cdot,\cdot\}, \beta_{1}, \beta_{2})$ be two BiHom-Leibniz
algebras. Suppose  there exist linear maps
 $ l_{  A}, r_{  A} :   A \rightarrow gl(   B),$ and
$ l_{   B}, r_{   B} :    B \rightarrow gl(  A) $ such that
$ (l_{  A}, r_{  A}, \beta_{1}, \beta_{2}) $ is a bimodule of $ A,$ and $ (l_{   B}, r_{   B}, \alpha_{1}, \alpha_{2}) $
is a bimodule of $    B $. Then,
 $ (  A,    B, l_{  A}, r_{  A}, \beta_{1}, \beta_{2}, l_{   B}, r_{   B}, \alpha_{1}, \alpha_{2}) $
 is called a matched pair of BiHom-Leibniz algebras,  if the conditions \eqref{bimatch. pair1} - \eqref{bimatch. pair6} are satisfied.
\end{defn}
\begin{ex}
We consider the multiplicative BiHom-Leibniz algebra $(A, [\cdot,\cdot],\alpha_1,\alpha_2)$. By Example \ref{bmodex}, we have $(L,0,\alpha_1,\alpha_2,A)$ and $(L,R,\alpha,\alpha_2,A)$ are bimodules of $A$. Then $ (  A,    A, L, 0, \alpha_{1}, \alpha_{2}, L, R, \alpha_{1}, \alpha_{2}) $ is a matched pair of BiHom-Leibniz algebras if and only if for all $x,y,z\in A$, $[[\alpha_2(x),y],\alpha_2(z)]=0$
\end{ex}
\begin{ex}
Consider the three-dimensional BiHom-Leibniz algebras $(A, [\cdot,\cdot],\alpha_1,\alpha_2)$ and $(A, \{\cdot,\cdot\},\beta_1,\beta_2)$ with multiplications tables for a basis $\{e_1, e_2, e_3\}$:
\begin{center}
\begin{tabular}{c|cccc}
$[\cdot,\cdot]$ & $e_1$ & $e_2$ & $e_3$ \\ \hline
$e_1$ & $e_3$ & $-e_3$ & $0$  \\
$e_2$ & $-e_3$ & $3e_3$ & $0$  \\
$e_3$ & $0$ & $0$ & $0$ \\
\end{tabular}
$$\begin{array}{llll}
    \alpha_1(e_1)=-e_1, &\alpha_1(e_2)=-e_2,&\alpha_1(e_3)=e_3,\\
    \alpha_2(e_1)=-2e_1, &\alpha_2(e_2)=-2e_2, &\alpha_2(e_3)=4 e_3,
\end{array}$$
\hspace{1 cm}
\begin{tabular}{c|cccc}
$\{\cdot,\cdot\}$ & $e_1$ & $e_2$ & $e_3$ \\ \hline
$e_1$ & $-e_3$ & $0$ & $0$  \\
$e_2$ & $0$ & $2e_3$ & $0$  \\
$e_3$ & $0$ & $0$ & $0$ \\
\end{tabular}
\end{center}
$$\begin{array}{llll}
    \beta_1(e_1)=\frac{1}{2}e_1, &\beta_1(e_2)=\frac{1}{2}e_2, &\beta_1(e_3)=\frac{1}{4}e_3,\\
    \beta_2(e_1)=\sqrt{2}e_1, &\beta_2(e_2)=\sqrt{2}e_2, &\beta_2(e_3)=2e_3.
\end{array}$$
It is clear that $ (L, R, \beta_{1}, \beta_{2}) $ is a bimodule of $ (A,[\cdot,\cdot],\alpha_1,\alpha_2),$ and $ (\mathcal{L}, \mathcal{R}, \alpha_{1}, \alpha_{2}) $
is a bimodule of $(A,\{\cdot,\cdot\},\beta_1,\beta_2)$, where
\begin{align*}
\forall \quad & (x,y)\in A\times A: \\
& L(x)y=[x, y], \quad R(x)y=[y,x], \\
& \mathcal{L}(x)y=\{x,y\},\quad \mathcal{R}(x)y=\{y, x\}.
\end{align*}
It is easy to see that the conditions \eqref{bimatch. pair1} - \eqref{bimatch. pair6} are satisfied for $l_A=L,~r_A=R,~l_B=\mathcal{L},~r_B=\mathcal{R}$ and $B=A$. Then $ (  A,    A, L, R, \beta_{1}, \beta_{2}, \mathcal{L}, \mathcal{R}, \alpha_{1}, \alpha_{2}) $ is a matched pair of BiHom-Leibniz algebras.
\end{ex}
\section{BiHom-Leibniz dendriform algebras}
\label{sec:bileibnizdendriformalgs}
\begin{defn}
A BiHom-Leibniz dendriform algebra is a quintuple $( A, \prec, \succ, \alpha, \beta)$ consisting of a linear space $ A$ on which the operations $\prec, \succ:  A\otimes  A\rightarrow  A$ and $\alpha, \beta:  A\rightarrow  A$ are linear maps satisfying
\begin{eqnarray}
&&\alpha\circ\beta=\beta\circ\alpha,\\
\label{BiH-Leib 1}
&&([\beta(x),y])\succ\beta(z)=\alpha\beta(x)\succ(y\succ z)-\beta(y)\succ(\alpha(x)\succ z),\label{first cond}\\
&&\alpha\beta(x)\succ(y\prec z)=(\beta(x)\succ y)\prec\beta(z)+\beta(y)\prec([\alpha(x),z]),\label{second cond}\\
&&\alpha\beta(x)\prec([y,z])=(\beta(x)\prec y)\prec\beta(z)+\beta(y)\succ(\alpha(x)\prec z),\label{third cond}
\end{eqnarray}
where $ [x , y] = x \prec y + x \succ y $.
\end{defn}
\begin{defn}
Let $( A, \prec, \succ, \alpha, \beta)$ and $( A', \prec',  \succ', \alpha', \beta')$ be BiHom-Leibniz dendriform algebras. A linear map $f:  A\rightarrow  A'$ is a BiHom-Leibniz dendriform algebra morphism if
\begin{eqnarray*}
\prec'\circ(f\otimes f)= f\circ\prec,\ \succ'\circ(f\otimes f)= f\circ\succ, f\circ\alpha= \alpha'\circ f \mbox{ and } f\circ\beta= \beta'\circ f.
\end{eqnarray*}
\end{defn}
\begin{prop}
Let $( A, \prec, \succ, \alpha, \beta)$ be a BiHom-Leibniz dendriform algebra.

Then, $( A, [\cdot,\cdot]=\prec+\succ, \alpha, \beta)$ is a BiHom-Leibniz algebra.
\end{prop}
\begin{proof}
For all $x,y,z\in A$,
\begin{align*}
    &[\alpha\beta(x),[y,z]]-[[\beta(x),y],\beta(z)]-[\beta(y),[\alpha(x),z]]\\
    &=[\alpha\beta(x),y\prec z+y\succ z]-[\beta(x)\prec y+\beta(x)\succ y,\beta(z)]\\
    &\quad-[\beta(y),\alpha(x)\prec z+\alpha(x)\succ z]\\
    &=[\alpha\beta(x),y\prec z+y\succ z]-[\beta(x)\prec y+\beta(x)\succ y,\beta(z)]\\
    &\quad-[\beta(y),\alpha(x)\prec z+\alpha(x)\succ z]\\
    &=\alpha\beta(x)\prec(y\prec z)+\alpha\beta(x)\succ(y\prec z)+\alpha\beta(x)\prec(y\succ z)\\
    &\quad+\alpha\beta(x)\succ(y\succ z)-(\beta(x)\prec y)\prec\beta(z)-(\beta(x)\prec y)\succ\beta(z)\\
    &\quad-(\beta(x)\succ y)\prec\beta(z)-(\beta(x)\succ y)\succ\beta(z)-\beta(y)\prec(\alpha(x)\prec z)\\
    &\quad-\beta(y)\succ(\alpha(x)\prec z)-\beta(y)\prec(\alpha(x)\succ z)-\beta(y)\succ (\alpha(x)\prec z)\\
    &=\underbrace{\Big(\alpha\beta(x)\succ(y\succ z)-\beta(y)\succ(\alpha(x)\succ z)-(\beta(x)\prec y+\beta(x)\succ y)\beta(z)\Big)}_{\text{$=0$ by \eqref{first cond}}}\\
    &\quad+\underbrace{\Big(\alpha\beta(x)\succ(y\prec z)-(\beta(x)\succ y)\prec\beta(z)-\beta(y)\prec(\alpha(x)\prec z+\alpha(x)\succ z)\Big)}_{\text{$=0$ by \eqref{second cond}}}\\
    &\quad+\underbrace{\Big(\alpha\beta(x)\prec(y\prec z+y\succ z)-(\beta(x)\prec y)\prec\beta(z)-\beta(y)\succ(\alpha(x)\prec z)\Big)}_{\text{$=0$ by \eqref{third cond}}}=0.
\qedhere \end{align*}
\end{proof}
We call $ ( A, [\cdot,\cdot], \alpha, \beta)$ the BiHom-Leibniz algebra of $( A, \prec, \succ, \alpha, \beta),$ and the quintuple  $( A, \prec, \succ, \alpha, \beta)$ is called a compatible BiHom-Leibniz dendriform algebra structure on the BiHom-Leibniz algebra $( A, [\cdot,\cdot], \alpha, \beta)$.

The following theorem shows that any Leibniz dendriform algebra \cite{ShengTang2017:LeibnizbialgsRotaBaxLeibnizYangBax} together with two commuting morphisms yield BiHom-Leibniz dendriform algebra, and morphisms of Leibniz dendriform algebra naturally correspond to
morphisms of BiHom-Leibniz dendriform algebra.

\begin{thm} \label{thm:BihomLeibDenr:ytwist}
Let $\mathcal{A} = (A, \prec , \succ )$ be a Leibniz dendriform algebra and $\alpha_1,\alpha_2 :A\rightarrow A$ be two morphisms of $\mathcal{A}$ such that $\alpha_1\alpha_2=\alpha_2\alpha_1$. Then,
\begin{enumerate}[label=\upshape{\arabic*)}]
\item
$\mathcal{A}_{\alpha_1,\alpha_2}:=(A, \prec _{\alpha_1,\alpha_2}=\prec\circ(\alpha_1\otimes\alpha_2), \succ_{\alpha_1,\alpha_2}=\succ\circ(\alpha_1\otimes\alpha_2), \alpha_1,\alpha_2)$ is a BiHom-Leibniz dendriform algebra;
\item
if $(A', \prec ', \succ')$ is another Leibniz dendriform algebra,
$\alpha'_1,\alpha'_2:A'\rightarrow A'$
are two commuting Leibniz dendriform algebra morphisms, and $f:A\rightarrow A'$ is a
Leibniz dendriform algebra morphism satisfying $f\circ \alpha_1 =\alpha'_1\circ f$ and $f\circ \alpha_2 =\alpha'_2\circ f$, then $f:\mathcal{A}_{\alpha_1,\alpha_2 }\rightarrow A'_{\alpha'_1,\alpha'_2}$ is a
BiHom-Leibniz dendriform algebra morphism.
\end{enumerate}
\end{thm}
\begin{proof}
We shall only prove relation \eqref{BiH-Leib 1} as the others are proved analogously.
\begin{align*}
& \mbox{For any } x, y, z \in A: \\
&([\alpha_2(x),y]_{\alpha_1,\alpha_2})\succ_{\alpha_1,\alpha_2}\alpha_2(z) \\
&\quad =\alpha_1([\alpha_1\alpha_2(x),\alpha_2(y)])\succ\alpha_2^{2}(z)
=[\alpha_{1}^{2}\alpha_2(x),\alpha_1\alpha_2(y)]\succ\alpha_{2}^{2}(z)\\
&\quad=\alpha_{1}^{2}\alpha_2(x)\succ(\alpha_1\alpha_2(y)\succ\alpha_2^{2}(z))-\alpha_1\alpha_2(y)\succ(\alpha_{1}^{2}\alpha_2(x)\succ\alpha_2^{2}(z))\\
&\quad=\alpha_{1}^{2}\alpha_2(x)\succ\alpha_2(\alpha_1(y)\succ\alpha_2(z))-\alpha_1\alpha_2(y)\succ\alpha_2(\alpha_1^{2}(x)\succ\alpha_2(z))\\
&\quad=\alpha_1\alpha_2(x)\succ_{\alpha_1,\alpha_2}(y\succ_{\alpha_1,\alpha_2}z)-\alpha_2(y)\succ_{\alpha_1,\alpha_2}(\alpha_1(x)\succ_{\alpha_1,\alpha_2} z).
\end{align*}
The second assertion follows from
\begin{align*}
    f(x\prec_{\alpha_1,\alpha_2} y)=&f(\alpha_1(x)\prec\alpha_2(y)) 
    =f(\alpha_1(x))\prec' f(\alpha_2(y))\\
    =&\alpha'_1 f(x)\prec'\alpha'_2 f(y)
    =f(x)\prec'_{\alpha_1,\alpha_2} f(y).
\end{align*}
Similarly, $ f(x\succ_{\alpha_1,\alpha_2}y)=f(x)\succ'_{\alpha'_1,\alpha'_2} f(y)$.
This completes the proof.
\end{proof}
\begin{prop}\label{isma}
Let $\mathcal{A}=(A, \prec , \succ, \alpha_1,\alpha_2)$ be a
BiHom-Leibniz dendriform algebra and $\alpha'_1,\alpha'_2: \mathcal{A}\rightarrow \mathcal{A}$ be a two BiHom-Leibniz dendriform algebra morphisms such that any two of the maps $\alpha_1,\alpha_2,\alpha'_1,\alpha'_2$
commute.
Then  $\mathcal{A}_{\alpha'_{1},\alpha'_2}=(A, \prec _{\alpha'_1,\alpha'_2}, \succ_{\alpha'_1,\alpha'_2},\alpha_1 \circ \alpha'_1,\alpha_2 \circ \alpha'_2)$ is a BiHom-Leibniz dendriform algebra.
\end{prop}
\begin{proof}
We prove the axioms \eqref{first cond}-\eqref{third cond} in $\mathcal{A}_{\alpha'_{1},\alpha'_2}$. For any $x,y,z\in A$, 
\begin{align*}
    &([\alpha_2\alpha'_2(x),y]_{\alpha'_1,\alpha'_2})\succ_{\alpha'_1,\alpha'_2}\alpha_2\alpha'_2(z)\\
    &\quad=\alpha'_{1}[\alpha'_1\alpha'_2(x),\alpha'_2(y)]\succ\alpha_2{\alpha'}_{2}^{2}(z)\\
    &\quad=[{\alpha'_{1}}^{2}\alpha'_{2}\alpha_2(x),\alpha'_1\alpha'_2(y)]\succ\alpha_2{\alpha'_{2}}^{2}(z)\\
    &\quad=\alpha_1\alpha_2{\alpha'_{1}}^{2}\alpha'_2(x)\succ(\alpha'_1\alpha'_2\succ{\alpha'_{2}}^{2}(z))-\alpha_2\alpha'_1\alpha'_2(y)\succ(\alpha_1\alpha'_1\alpha'_2(x)\succ{\alpha'_{2}}^{2}(z))\\
    &\quad=\alpha_1{\alpha'_{1}}^{2}\alpha_2\alpha'_2(x)\succ\alpha'_2(\alpha_1\succ\alpha'_2(z))-\alpha'_1\alpha_2\alpha'_2(y)\succ\alpha'_2(\alpha_1{\alpha'_{1}}^{2}(x)\succ\alpha'_2(z))\\
    &\quad=\alpha_1\alpha'_1\alpha_2\alpha'_2(x)\succ_{\alpha'_1,\alpha'_2}(y\succ_{\alpha'_1,\alpha'_2} z)-\alpha_2\alpha'_2(y)\succ_{\alpha'_1,\alpha'_2}(\alpha_1\alpha'_1(x)\succ_{\alpha'_1,\alpha'_2} z),\\*[0,1cm]
    &\alpha_1\alpha'_1\alpha_2\alpha'_2(x)\succ_{\alpha'_1,\alpha'_2}(y\prec_{\alpha'_1,\alpha'_2} z)\\
    &\quad=\alpha_1\alpha_2{\alpha'_{1}}^{2}\alpha'_2(x)\succ\alpha'_2(\alpha'_1(y)\prec\alpha'_2(z))\\
    &\quad=\alpha_1\alpha_2{\alpha'_{1}}^{2}\alpha'_2(x)\succ(\alpha'_1\alpha'_2(y)\prec{\alpha'_{2}}^{2}(z))\\
    &\quad=({\alpha'_{1}}^{2}\alpha'_{2}(x)\succ\alpha'_1\alpha'_2(y))\prec\alpha_2{\alpha'_{2}}^{2}(z)+\alpha_2\alpha'_1\alpha'_2(y)\succ([\alpha_1{\alpha'_{1}}^{2}\alpha'_{2}(x),{\alpha'_{2}}^{2}(z)])\\
    &\quad=(\alpha'_1\alpha_2\alpha'_2(x)\succ_{\alpha'_1,\alpha'_2}\alpha'_1(y))\prec\alpha_2{\alpha'_{2}}^{2}(z)+\alpha_2\alpha'_1\alpha'_2(y)\prec([\alpha_1\alpha'_1\alpha'_2(x),\alpha'_2(z)]_{\alpha'_1,\alpha'_2})\\
    &\quad=(\alpha_2\alpha'_2(x)\succ_{\alpha'_1,\alpha'_2} y)\prec_{\alpha'_1,\alpha'_2} \alpha_2\alpha'_2(z)+\alpha_2\alpha'_2(y)\prec_{\alpha'_1,\alpha'_2}([\alpha_1\alpha'_1(x),z]_{\alpha'_1,\alpha'_2}),\\
    &\alpha_1\alpha'_1\alpha_2\alpha'_2(x)\prec_{\alpha'_1,\alpha'_2}([y,z]_{\alpha'_1,\alpha'_2})\\
    &\quad=\alpha_1{\alpha'_{1}}^{2}\alpha_2\alpha'_2(x)\prec[\alpha'_{1}\alpha'_{2}(y),{\alpha'_{2}}^{2}]\\
    &\quad=(\alpha_2{\alpha'_{1}}^{2}\alpha'_2(x)\prec\alpha'_1\alpha'_2(y))\prec\alpha_2{\alpha'_{2}}^{2}(z)+\alpha_2\alpha'_1\alpha'_2(y)\succ(\alpha{\alpha'_{1}}^{2}\alpha'_2(x)\prec{\alpha'_{2}}^{2}(z))\\
    &\quad=(\alpha_2\alpha'_1\alpha'_2(x)\prec_{\alpha'_1,\alpha'_2}\alpha'_1(y))\prec\alpha_2{\alpha'_{2}}^{2}(z)+\alpha_2\alpha'_1\alpha'_2(y)\succ(\alpha_1\alpha'_1\alpha'_2(x)\succ_{\alpha'_1,\alpha'_2}\alpha'_2(z))\\
    &\quad=(\alpha_2\alpha'_2(x)\prec_{\alpha'_1,\alpha'_2} y)\prec_{\alpha'_1,\alpha'_2}\alpha_2\alpha'_2(z)+\alpha_2\alpha'_2(y)\succ_{\alpha'_1,\alpha'_2}(\alpha_1\alpha'_1(x)\prec_{\alpha'_1,\alpha'_2} z)
\end{align*}
Then  $\mathcal{A}_{\alpha'_{1},\alpha'_2}=(A, \prec _{\alpha'_1,\alpha'_2}, \succ_{\alpha'_1,\alpha'_2},\alpha_1 \circ \alpha'_1,\alpha_2 \circ \alpha'_2)$ 
is a BiHom-Leibniz dendriform algebra.
\end{proof}
\begin{ex}
There are three-dimensional BiHom-Leibniz dendriform algebras \\ 
$(A, \prec,\succ,\alpha_1,\alpha_2)$ with multiplications tables for a basis $\{e_1, e_2, e_3\}$:
\begin{center}
\begin{tabular}{c|cccc}
$\prec$ & $e_1$ & $e_2$ & $e_3$ \\ \hline
$e_1$ & $0$ & $0$ & $0$  \\
$e_2$ & $0$ & $0$ & $-2e_1$  \\
$e_3$ & $0$ & $0$ & $3e_1$ \\
\end{tabular}
\hspace{1 cm}
\begin{tabular}{c|cccc}
$\succ$ & $e_1$ & $e_2$ & $e_3$ \\ \hline
$e_1$ & $0$ & $0$ & $0$  \\
$e_2$ & $0$ & $0$ & $-2e_1$  \\
$e_3$ & $0$ & $-2e_1$ & $3e_1$ \\
\end{tabular}
\end{center}
$$\begin{array}{llll}
    \alpha_1(e_1)=\lambda_1e_1, &\alpha_1(e_2)=0, &\alpha_1(e_3)=\lambda_2e_3,\\
    \alpha_2(e_1)=0, &\alpha_2(e_2)=\mu_1 e_2, &\alpha_2(e_3)=\mu_2 e_3,
\end{array}$$
where $\lambda_1,~\lambda_2,~\mu_1$ and $\mu_2$ are parameters in $\mathbb{K}\setminus\{0\}$.
Using suitable algebra morphisms $\alpha'_{1}$ and $\alpha'_{2}$, we obtain the BiHom-Leibniz dendriform
algebras. The algebra morphisms
$\alpha'_{1},\alpha'_{2}:A\rightarrow A$ are given by
$$\begin{array}{llll}
    \alpha'_1(e_1)=\frac{p^{2}}{3}e_1, &\alpha'_1(e_2)=-\frac{2p}{3}e_2, &\alpha'_1(e_3)=pe_3,\\
    \alpha'_2(e_1)=\frac{q^{2}}{3}e_1, &\alpha'_2(e_2)=-\frac{2q}{3}e_2, &\alpha'_2(e_3)=q e_3,
\end{array}$$
where $p$ and $q$ are parameters in $\mathbb{K}\setminus\{0\}$.
Then there is a BiHom-Leibniz dendriform algebra $\mathcal{A}_{\alpha'_{1},\alpha'_2}=(A, \prec _{\alpha'_1,\alpha'_2}, \succ_{\alpha'_1,\alpha'_2},\alpha_1 \circ \alpha'_1,\alpha_2 \circ \alpha'_2)$ with multiplications
tables:
\begin{center}
\begin{tabular}{c|cccc}
$\prec_{\alpha'_1,\alpha'_2}$ & $e_1$ & $e_2$ & $e_3$ \\ \hline
$e_1$ & $0$ & $0$ & $0$  \\
$e_2$ & $0$ & $0$ & $\frac{4pq}{3}e_1$  \\
$e_3$ & $0$ & $0$ & $3pqe_1$ \\
\end{tabular}
\hspace{1 cm}
\begin{tabular}{c|cccc}
$\succ_{\alpha'_1,\alpha'_2}$ & $e_1$ & $e_2$ & $e_3$ \\ \hline
$e_1$ & $0$ & $0$ & $0$  \\
$e_2$ & $0$ & $0$ & $\frac{4pq}{3}e_1$  \\
$e_3$ & $0$ & $\frac{4pq}{3}e_1$ & $3pqe_1$ \\
\end{tabular}
\end{center}

$$\begin{array}{llll}
    \alpha_1\alpha'_1(e_1)=\frac{\lambda_{1}p^{2}}{3}e_1, & \alpha_1\alpha'_1(e_2)=0,& \alpha_1\alpha'_1(e_3)=\lambda_{2}pe_3,\\
    \alpha_2\alpha'_2(e_1)=0, & \alpha_2\alpha'_2(e_2)=-\frac{2\mu_1q}{3}e_2, & \alpha_2\alpha'_2(e_3)=\mu_{2}q e_3.
\end{array}$$
\end{ex}
\begin{cor}  \label{cor:isma}
Let $(A, \prec , \succ,\alpha_1,\alpha_2)$ be a multiplicative BiHom-Leibniz dendriform
algebra. Let $n$ be any positive integer.
\begin{enumerate}
\item \label{cor:isma:itemi}
The $nth$ derived BiHom-Leibniz dendriform algebra of type $1$ of $A$ is
defined by
$$\mathcal{A}_{1}^{n}=(A,\prec^{(n)}=\prec\circ(\alpha_1^{n}\otimes\alpha_2^{n}),\succ^{(n)}=\succ\circ(\alpha_1^{n}\otimes\alpha_2^{n})
,\alpha_1^{n+1},\alpha_2^{n+1}).$$
\item \label{cor:isma:itemii}
The $nth$ derived BiHom-Leibniz dendriform algebra of type $2$ of $A$ is
defined by
$$
\mathcal{A}_{2}^{n}=(A,\prec^{(2^n-1)}=\prec\circ(\alpha_1^{2^{n}-1}\otimes\alpha_2^{2^{n}-1}),\succ^{(2^n-1)}=\succ\circ(\alpha_1^{2^{n}-1}\otimes\alpha_2^{2^{n}-1}),\alpha_1^{2^n},\alpha_2^{2^n}).
$$
\end{enumerate}
\end{cor}
\begin{proof}
Apply Proposition \ref{isma} with $\alpha'_1=\alpha_1^{n},~\alpha'_2=\alpha_2^{n}$ and
$\alpha'_1=\alpha_1^{2^{n}-1},~\alpha'_2=\alpha_2^{2^{n}-1}$ respectively.
We prove in more details, only the axiom \eqref{first cond} in both cases $\mathcal{A}_{1}^{n}$ and $\mathcal{A}_{2}^{n}$, as others proved similarly.

\ref{cor:isma:itemi} \  For any $x,y,z\in A$,
   \begin{align*}
  &([\alpha_{2}^{n+1}(x),y]^{(n)})\succ^{(n)}\alpha_{2}^{n+1}(z)\\
  &\quad =\alpha_{1}^{n}([\alpha_{1}^{n}\alpha_{2}^{n+1}(x),\alpha_{2}^{n}(y)])\succ\alpha_{2}^{2n+1}(z)\\
  &\quad=[\alpha_{1}^{2n}\alpha_{2}^{n+1}(x),\alpha_{1}^{n}\alpha_{2}^{n}(y)])\succ\alpha_{2}^{2n+1}(z)\\
  &\quad=\alpha_{1}^{2n+1}\alpha_{2}^{n+1}(x)\succ(\alpha_{1}^{n}\alpha_{2}^{n}(y)\succ\alpha_{2}^{2n}(z))-\alpha_{1}^{n}\alpha_{2}^{n+1}(y)\succ(\alpha_{1}^{2n+1}\alpha_{2}^{n}(x)\succ\alpha_{2}^{2n}(z))\\
  &\quad=\alpha_{1}^{2n+1}\alpha_{2}^{n+1}(x)\succ(\alpha_{2}^{n}(y)\succ^{(n)}\alpha_{2}^{n}(z))-\alpha_{1}^{n}\alpha_{2}^{n+1}(y)\succ(\alpha_{1}^{n+1}\alpha_{2}^{n}(x)\succ^{(n)}\alpha_{2}^{n}(z))\\
  &\quad=\alpha_{1}^{n+1}\alpha_{2}^{n+1}(x)\succ^{(n)}(y\succ^{(n)} z)-\alpha_{2}^{n+1}(y)\succ^{(n)}(\alpha_{1}^{n+1}(x)\succ^{(n)} z).
  \end{align*}

\ref{cor:isma:itemii} \ For any $x,y,z\in A$,
\begin{align*}
  &(\alpha_{2}^{2^{n}}(x),y]^{2^{n}-1})\succ^{2^{n}-1}\alpha_2^{2^{n}}(z)\\
  &\quad=\alpha_1^{2^{n}-1}([\alpha_1^{2^{n}-1}\alpha_2^{2^{n}}(x),\alpha_2^{2^{n}-1}(y)]\succ\alpha_2^{2^{n}-1}\alpha_2^{2^{n}}(z)\\
  &\quad=([\alpha_1^{2^{n}-1}\alpha_1^{2^{n}-1}\alpha_2^{2^{n}}(x),\alpha_1^{2^{n}-1}\alpha_2^{2^{n}-1}(y)]\succ\alpha_2^{2^{n}-1}\alpha_2^{2^{n}}(z)\\
   &\quad=\alpha_1^{2^{n}-1}\alpha_1^{2^{n}}\alpha_2^{2^{n}}(x)\succ(\alpha_2^{2^{n}-1}(y)\succ^{(2^{n}-1)}\alpha_2^{2^{n}-1}(z))\\&\quad\quad-\alpha_1^{2^{n}-1}\alpha_2^{2^{n}}(y)\succ(\alpha_1^{2^{n}}\alpha_2^{2^{n}-1}(x)\succ^{(2^{n}-1)}\alpha_2^{2^{n}-1}(z))\\
  &\quad=\alpha_1^{2^{n}-1}\alpha_1^{2^{n}}\alpha_2^{2^{n}}(x)\succ(\alpha_1^{2^{n}-1}\alpha_2^{2^{n}-1}(y)\succ\alpha_2^{2^{n}-1}\alpha_2^{2^{n}-1}(z))\\&\quad\quad-\alpha_1^{2^{n}-1}\alpha_2^{2^{n}}(y)\succ(\alpha_1^{2^{n}}\alpha_1^{2^{n}-1}\alpha_2^{2^{n}-1}(x)\succ\alpha_2^{2^{n}-1}\alpha_2^{2^{n}-1}(z))\\
  &\quad=\alpha_1^{2^{n}}\alpha_2^{2^{n}}(x)\succ^{(2^{n}-1)}(y\succ^{(2^{n}-1)} z)-\alpha_2^{2^{n}}(y)\succ^{(2^{n}-1)}(\alpha_1^{2^{n}}(x)\succ^{(2^{n}-1)}z).
 \qedhere   \end{align*}
\end{proof}

\begin{ex}
There are three-dimensional multiplicative BiHom-Leibniz dendriform algebras $(A, \prec,\succ,\alpha_1,\alpha_2)$ with multiplications tables for a basis $\{e_1, e_2, e_3\}$:
\begin{center}
\begin{tabular}{c|cccc}
$\prec$ & $e_1$ & $e_2$ & $e_3$ \\ \hline
$e_1$ & $0$ & $-2e_3$ & $0$  \\
$e_2$ & $-2e_3$ & $3e_3$ & $0$  \\
$e_3$ & $0$ & $0$ & $0$
\end{tabular}
\hspace{1 cm}
\begin{tabular}{c|cccc}
$\succ$ & $e_1$ & $e_2$ & $e_3$ \\ \hline
$e_1$ & $0$ & $-2e_3$ & $0$  \\
$e_2$ & $-2e_3$ & $3e_3$ & $0$  \\
$e_3$ & $0$ & $0$ & $0$
\end{tabular}
\end{center}
$$\begin{array}{llll}
    \alpha_1(e_1)=-\frac{2p}{3}e_1, &\alpha_1(e_2)=pe_2, &\alpha_1(e_3)=\frac{p^{2}}{3}e_3,\\
    \alpha_2(e_1)=-\frac{2q}{3}e_1, &\alpha_2(e_2)=qe_2, &\alpha_2(e_3)=\frac{q^{2}}{3} e_3,
\end{array}$$
where $p$ and $q$ are parameters in $\mathbb{K}\setminus\{0\}$.
Then there are BiHom-Leibniz dendriform algebras $\mathcal{A}_{1}^{n}$ and $\mathcal{A}_{2}^{n}$ with multiplications tables respectively:
\begin{center}
\begin{tabular}{c|cccc}
$\prec^{(n)}$ & $e_1$ & $e_2$ & $e_3$ \\ \hline
$e_1$ & $0$ & $-2(\frac{-2pq}{3})^{n}e_3$ & $0$  \\
&  &  &   \\
$e_2$ & $-2(\frac{-2pq}{3})^{n}e_3$ & $3(pq)^{n}e_3$ & $0$  \\
&  &  &   \\
$e_3$ & $0$ & $0$ & $0$
\end{tabular}
\vspace{1 cm}

\begin{tabular}{c|cccc}
$\succ^{(n)}$ & $e_1$ & $e_2$ & $e_3$ \\ \hline
$e_1$ & $0$ & $-2(\frac{-2pq}{3})^{n}e_3$ & $0$  \\
&  &  &   \\
$e_2$ & $-2(\frac{-2pq}{3})^{n}e_3$ & $3(pq)^{n}e_3$ & $0$  \\
&  &  &   \\
$e_3$ & $0$ & $0$ & $0$
\end{tabular}

$$\begin{array}{llll}
    \alpha_1^{n+1}(e_1)=(-\frac{2p}{3})^{n+1}e_1, &\alpha_1^{n+1}(e_2)=p^{n+1}e_2,&\alpha_1^{n+1}(e_3)=(\frac{p^{2}}{3})^{n+1}e_3,\\
    \alpha_2^{n+1}(e_1)=(-\frac{2q}{3})^{n+1}e_1, &\alpha_2^{n+1}(e_2)=q^{n+1}e_2, &\alpha_2^{n+1}(e_3)=(\frac{q^{2}}{3})^{n+1} e_3,
\end{array}$$
\end{center}
\vspace{0,3cm}
\begin{center}
\begin{tabular}{c|cccc}
$\prec^{(2^{n}-1)}$ & $e_1$ & $e_2$ & $e_3$ \\ \hline
$e_1$ & $0$ & $2(\frac{2pq}{3})^{2^{n}-1}e_3$ & $0$  \\
&  &  &   \\
$e_2$ & $2(\frac{2pq}{3})^{2^{n}-1}e_3$ & $3(pq)^{2^{n}-1}e_3$ & $0$  \\
&  &  &   \\
$e_3$ & $0$ & $0$ & $0$
\end{tabular}

\vspace{1 cm}
\begin{tabular}{c|cccc}
$\succ^{(2^{n}-1)}$ & $e_1$ & $e_2$ & $e_3$ \\ \hline
$e_1$ & $0$ & $2(\frac{2pq}{3})^{2^{n}-1}e_3$ & $0$   \\
&  &  &   \\
$e_2$ & $2(\frac{2pq}{3})^{2^{n}-1}e_3$ & $3(pq)^{2^{n}-1}e_3$ & $0$  \\
&  &  &   \\
$e_3$ & $0$ & $0$ & $0$
\end{tabular}
\end{center}

$$\begin{array}{llll}
    \alpha_1^{2^{n}}(e_1)=(\frac{2p}{3})^{2^{n}}e_1, &\alpha_1^{2^{n}}(e_2)=p^{2^{n}}e_2, &\alpha_1^{2^{n}}(e_3)=(\frac{p^{2}}{3})^{2^{n}}e_3,\\
    \alpha_2^{2^{n}}(e_1)=(\frac{2q}{3})^{2^{n}}e_1, &\alpha_2^{2^{n}}(e_2)=q^{2^{n}}e_2, &\alpha_2^{2^{n}}(e_3)=(\frac{q^{2}}{3})^{2^{n}} e_3.
\end{array}$$
\vspace{0,1cm}
\end{ex}

Next we define notion of bimodule of BiHomLeibniz dendriform algebra. 
\begin{defn}
Let $\mathcal{A}=( A, \prec, \succ, \alpha_{1}, \alpha_{2})$ be a BiHom-Leibniz dendriform algebra,  and $V$  be a linear space.
Let $l_{\prec}, r_{\prec}, l_{\succ}, r_{\succ} :  A \rightarrow gl(V),$ and $\beta_{1}, \beta_{2}: V \rightarrow V$ be six linear maps. Then, $( l_{\prec}, r_{\prec}, l_{\succ}, r_{\succ}, \beta_{1}, \beta_{2}, V)$ is called a bimodule of $\mathcal{A}$
 if for any $ x, y \in  A, v\in V$ and $ [x , y] = x \prec y + x \succ y, l = l_{\prec} + l_{\succ}, r = r_{\prec} + r_{\succ} $:
\begin{eqnarray}
\label{BiH-mod1}
l_{\succ}([\alpha_2(x),y])\beta_2(v)&=&l_{\succ}(\alpha_1\alpha_2(x)l_\succ(y)v-l_\succ(\alpha_2(y))l_\succ((\alpha_1(x))v,\\
\label{BiH-mod2}
r_\succ(\alpha_2(y))l(\alpha_2(x))v&=&l_\succ(\alpha_1\alpha_2(x))r_\succ(y)v-r_\succ(\alpha_1(x)\succ y)\beta_2(v),\\
\label{BiH-mod3}
r_\succ(\alpha_2(y))r(x)\beta_2(v)&=&r_\succ(x\succ y)\beta_1\beta_2(v)-l_\succ(\alpha_2(x))r_\succ(y)\beta_1(v),\\
\label{BiH-mod4}
l_\succ(\alpha_1\alpha_2(x))l_\prec(y)v&=&l_\prec(\alpha_2(x)\succ y)\beta_2(v)+l_\prec(\alpha_2(y))l(\alpha_1(x))v,\\
\label{BiH-mod5}
l_\succ(\alpha_1\alpha_2(x))r_\prec(y)v&=&r_\prec(\alpha_2(y))l_\succ(\alpha_2(x))\beta_1(v)+r_\prec([\alpha_1(x),y])\beta_2(v),\\
\label{BiH-mod6}
r_\succ(x\prec y)\beta_1\beta_2(v)&=&r_\prec(\alpha_2(y))r_\succ(x)\beta_2(v)+l_\prec(\alpha_2(x))r(y)\beta_1(v),\\
\label{BiH-mod7}
l_\prec(\alpha_1\alpha_2(x))r(y)v&=&l_\prec(\alpha_2(x)\prec y)\beta_2(v)+l_\succ(\alpha_2(y))l_\prec(\alpha_1(x))v,\\
\label{BiH-mod8}
l_\prec(\alpha_1\alpha_2(x))r(y)v&=&r_\prec(\alpha_2(y))l_\prec(\alpha_2(x))v+r_\succ(\alpha_1(x)\prec y)\beta_2(v),\\
\label{BiH-mod9}
r_\prec([x,y])\beta_1\beta_2(v)&=&r_\prec(\alpha_2(y))r_\prec(x)\beta_2(v)+l_\succ(\alpha_2(x))r_\prec(y)\beta_1(v),\\
\label{BiH-mod10}
\beta_1(l_\prec(x)v)&=&l_\prec(\alpha_1(x))\beta_1(v),\quad
\beta_1(r_\prec(x)v)=l_\prec(\alpha_1(x))\beta_1(v),\\
\label{BiH-mod12}
\beta_1(l_\succ(x)v)&=&l_\succ(\alpha_1(x))\beta_1(v),\quad
\beta_1(r_\succ(x)v)=l_\succ(\alpha_1(x))\beta_1(v),\\
\label{BiH-mod13}
\beta_2(l_\prec(x)v)&=&l_\prec(\alpha_2(x))\beta_2(v),\quad
\beta_2(r_\prec(x)v)=l_\prec(\alpha_2(x))\beta_2(v)\\
\label{BiH-mod14}
\beta_2(l_\succ(x)v)&=&l_\succ(\alpha_2(x))\beta_2(v),\quad
\beta_2(r_\succ(x)v)=l_\succ(\alpha_2(x))\beta_2(v).
\end{eqnarray}
\end{defn}
\begin{prop}
Let $(l_{\prec}, r_{\prec}, l_{\succ}, r_{\succ}, \beta_{1}, \beta_{2}, V)$ be a bimodule of a BiHom-Leibniz dendriform algebra $( A,\prec, \succ, \alpha_{1}, \alpha_{2}).$
On the direct sum $ A\oplus V $ of the linear spaces, there is a BiHom-Leibniz dendriform algebra structure given for $x, y \in  A, u, v \in V$ by
\begin{eqnarray*}
(x + u) \prec (y + v) &=& x \prec y + l_{\prec}(x)v + r_{\prec}(y)u, \cr
(x + u) \succ (y + v) &=& x \succ y + l_{\succ}(x)v + r_{\succ}(y)u.
\end{eqnarray*}
We denote it by $  A \times_{l_{\prec},r_{\prec}, l_{\succ}, r_{\succ}, \alpha_{1}, \alpha_{2}, \beta_{1}, \beta_{2}} V$.
\end{prop}
\begin{proof}
Let $v_{1}, v_{2}, v_{3}\in V$ and $x_{1}, x_{2},x_{3}\in  A.$ Setting and computing
\begin{align*}
&([(\alpha_2+\beta_2)(x_1+v_1),x_2+v_2]')\succ'(\alpha_2+\beta_2)(x_3+v_3)\\
&\quad\quad-(\alpha_1+\beta_1)(\alpha_2+\beta_2)(x_1+v_1)\succ'((x_2+v_2)\succ'(x_3+v_3))\\
&\quad\quad+((\alpha_2+\beta_2)(x_2+v_2)\succ'((\alpha_1+\beta_1)(x_1+v_1)\succ x_3+v_3)\\
&\quad=[\alpha_2(x_1),x_2]\succ\alpha_2(x_3)+l_\succ([\alpha_2(x_1),x_2])\beta_2(v_3)\\&\quad\quad +r_{\succ}(\alpha_2(x_3))l_\prec(\alpha_2(x_1))v_2+r_\succ(\alpha_2(x_3))r_\prec(x_2)\beta_2(v_3)\\
&\quad\quad+r_\succ(\alpha_2(x_3))l_\succ(\alpha_2(x_1))v_2+r_\succ(\alpha_2(x_3))r_\succ(x_2)\beta_2(v_1)\\
&\quad\quad-\alpha_1\alpha_2(x_1)\succ(x_2\succ x_3)-l_\succ(\alpha_1\alpha_2(x_1))l_\succ(x_2)v_3\\&\quad\quad-l_\succ(\alpha_1\alpha_2(x_1))r_\succ(x_3)v_2-r_\succ(x_2\succ x_3)\beta_1\beta_2(v_1)\\
&\quad\quad+\alpha_2(x_2)\succ(\alpha_1(x_1)\succ x_3)+l_\succ(\alpha_2(x_2))l_\succ(\alpha_1(x_1))v_3\\
&\quad\quad+l_\succ(\alpha_2(x_2))r_\succ(x_3)\beta_1(v_1)+r_\succ(\alpha_1(x_1)\succ x_3)\beta_2(v_2).
\end{align*}
By \eqref{BiH-mod1}-\eqref{BiH-mod3},
and \eqref{BiH-Leib 1} in $A$,
\begin{align*}
& ([(\alpha_2+\beta_2)(x_1+v_1),x_2+v_2]')\succ'(\alpha_2+\beta_2)(x_3+v_3) \\
& \quad -(\alpha_1+\beta_1)(\alpha_2+\beta_2)(x_1+v_1)\succ'((x_2+v_2)\succ'(x_3+v_3))
\\
& \quad +((\alpha_2+\beta_2)(x_2+v_2)\succ'((\alpha_1+\beta_1)(x_1+v_1)\succ x_3+v_3)=0.
\qedhere \end{align*}
\end{proof}
\begin{prop}\label{proposition BiHbimodule}
Let $(l_{\prec}, r_{\prec}, l_{\succ}, r_{\succ}, \beta_{1}, \beta_{2}, V)$ be a bimodule of a BiHom-Leibniz dendriform algebra $( A, \prec, \succ, \alpha_{1}, \alpha_{2})$, and let $(A, [\cdot,\cdot], \alpha_1,\alpha_2)$ be the associated BiHom-Leibniz algebra. Then,
\begin{enumerate}[label=\upshape{\arabic*)}]
\item $(l_{\succ}, r_{\prec}, \beta_{2}, \beta_{1}, V) $ and $ (l_{\prec} + l_{\succ}, r_{\prec} + r_{\succ}, \beta_{1}, \beta_{2}, V)$ are bimodules \\ of $ ( A, [\cdot,\cdot], \alpha_{1}, \alpha_{2}); $
\item for any bimodule $(l, r, \beta_{1}, \beta_{2}, V)$ of $( A, [\cdot,\cdot], \alpha_{1}, \alpha_{2})$, $(0,l,r,0, \beta_{1}, \beta_{2}, V)$ is a bimodule of $( A, \prec, \succ, \alpha_{1}, \alpha_{2}).$
\item $(0,l_{\prec} + l_{\succ}, r_{\prec} + r_{\succ}, 0, \beta_{1}, \beta_{2}, V)$ and $(0,l_{\succ},  r_{\prec}, 0, \beta_{1}, \beta_{2}, V)$ are bimodules
 of\\ $( A, \prec, \succ, \alpha_{1}, \alpha_{2});$
\item  the BiHom-Leibniz dendriform algebras
\begin{eqnarray*}
  A \times_{l_{\prec}, r_{\prec}, l_{\succ}, r_{\succ}, \alpha_{1}, \alpha_{2}, \beta_{1}, \beta_{2}}V \mbox{ and }   A \times_{ 0,l_{\prec} +  l_{\succ} , r_{\prec} + r_{\succ}, 0, \alpha_{1}, \alpha_{2}, \beta_{1}, \beta_{2}} V
 \end{eqnarray*} have the same associated
BiHom-Leibniz algebra $ A \times_{l_{\prec} +  l_{\succ}, r_{\prec} + r_{\succ}, \alpha_{1}, \alpha_{2}, \beta_{1}, \beta_{2}} V.$
 \end{enumerate}
\end{prop}
\begin{ex}
Let $( A, \prec, \succ, \alpha, \beta)$ be a BiHom-Leibniz dendriform algebra. Suppose that $( A, [\cdot,\cdot], \alpha, \beta)$ is the associated BiHom-Leibniz algebra.  Then, $ (L_\prec+L_{\succ}, R_{\prec}+R_\succ, \alpha, \beta,  A)$ and $ (L_{\succ}, R_{\prec}, \alpha, \beta,  A)$ are bimodule of $( A,[\cdot,\cdot], \alpha, \beta)$.
\end{ex}
\begin{thm}
Let $( A, \prec_{ A}, \succ_{ A}, \alpha_{1}, \alpha_{2})$ and $( B , \prec_{ B }, \succ_{ B }, \beta_{1}, \beta_{2})$
 be two BiHom-Leibniz dendriform algebras. If there are linear maps
$ l_{\prec_{ A}},   r_{\prec_{ A}},  l_{\succ_{ A}},  r_{\succ_{ A}} :  A \rightarrow gl( B ),$
and $ l_{\prec_{ B }},   r_{\prec_{ B }},  l_{\succ_{ B }},  r_{\succ_{ B }} :  B  \rightarrow gl( A)$
such that $( l_{\prec_{ A}},   r_{\prec_{ A}},  l_{\succ_{ A}},
r_{\succ_{ A}}, \beta_{1}, \beta_{2},  B )$ is a bimodule of $ A,$ and
$(l_{\prec_{ B }},   r_{\prec_{ B }},  l_{\succ_{ B }},  r_{\succ_{ B }},
\alpha_{1}, \alpha_{2},  A)$ is a bimodule  of $ B ,$
and for any $ x, y \in  A,$ $a, b \in  B  $ and $ l_{ A} = l_{\prec_{ A}} +
 l_{\succ_{ A}}, r_{ A} =  r_{\prec_{ A}} +  r_{\succ_{ A}}, l_{ B } =
 l_{\prec_{ B }} +  l_{\succ_{ B }} , r_{ B } =  r_{\prec_{ B }} +  r_{\succ_{ B }} $,
\begin{eqnarray}
\label{bieq35}
l_A(\alpha_2(x))a\succ_B\beta_2(b)+l_{\succ_A}(r_B(\alpha_2(x)))\beta_2(b)=l_{\succ_A}(\alpha_1\alpha_2(x))(a\succ b)\nonumber\\-\beta_2(a)\succ_B(l_{\succ_A}(\alpha_1(x))b)-r_{\succ_A}(r_{\succ_B}(b)\alpha_1(x))\beta_2(a),
\\
\label{bieq36}
r_A(x)\beta_2(a)\succ_B\beta_2(b)+l_{\succ_A}(l_B(\beta_2(a))x)\beta_2(b)=\beta_1\beta_2(a)\succ_B(l_{\succ_A}(x)b)\nonumber\\+r_A(r_B(b)x)\beta_1\beta_2(a)-l_{\succ_A}(\alpha_2(x))(\beta_1(a)\succ_B b),\\
\label{bieq37}
r_{\succ_A}(\alpha_2(x))([\beta_2(a),b])=\beta_1\beta_2(a)\succ_B(r_{\succ_A}(x))b+r_A(r_B(b)x)\beta_1\beta_2(b)\nonumber\\-\beta_2(a)\succ_B(r_{\succ_A}(x)\beta_1(b))-r_A(l_{\succ_B}(\beta_1(b))x)\beta_2(a),\\
 \label{bieq38}
l_{\succ_A}(\alpha_1\alpha_2(x))(a\prec b)=(l_{\succ_A}(\alpha_2(x))a)\prec_B \beta_2(b)+l_{\prec_A}(r_{\succ_B}(a)\alpha_2(x))\beta_2(b)\nonumber\\+\beta_2(a)\prec_B(l_A(\alpha_1(x))b)+r_{prec_A}(r_B(b)\alpha_1(x))\beta_2(a), \\
\label{bieq39}
\beta_1\beta_2(a)\succ_B(l_{\prec_A}(x)b)+r_{\succ_A}(r_B(b)x)\beta_1\beta_2(a)=(r_{\succ_A}(x)\beta_2(a))\prec_B\beta_2(b)\nonumber\\+l_A(l_{\succ_B}(\beta_2(a))x)\beta_2(b)+l_{\prec_A}(\alpha_2(x))([\beta_1(a),b]),\\
\label{bieq40}
\beta_1\beta_2(a)\succ_B(r_{\prec_A}(x)b)+r_A(l_B(b)x)\beta_1\beta_2(a)=r_{\prec_A}(\alpha_2(x))(\beta_2(a)\succ_B b)\nonumber\\+r_{\prec_A}(l_B(\beta_1(b)x)\beta_2(a)+\beta_2(a)\prec_B(r_A(x)\beta_1(b)),
\\
\label{bieq41}
l_{\prec_A}(\alpha_1\alpha_2(x))([a,b])=l_{\prec_A}(r_{\prec_B}(a)x)\beta_2(b)+(l_{\prec_A}(\alpha_2(x))a)\prec_B\beta_2(b)\nonumber\\+\beta_2(a)\succ_B(l_{\prec_A}(\alpha_1(x))b)+r_A(r_B(b)\alpha_1(x))\beta_2(a),
\\
\label{bieq42}
\beta_1\beta_2(a)\prec_B(l_A(x)b)+r_{\prec_A}(r_B(b)x)\beta_1\beta_2(a)=(r_{\prec_A}(x)\beta_2(a))\prec_B\beta_2(b)\nonumber\\+l_{\prec_A}(l_{\prec_B}(\beta_2(a))x)\beta_2(b)+l_{\succ_A}(\alpha_2(x))(\beta_1(a)\prec b),\\
\label{bieq43}
\beta_1\beta_2(a)\prec_B(r_A(x)b)+l_{\prec_A}(l_B(b)x)\beta_1\beta_2(a)=r_{\prec_A}(\alpha_2(x))(\beta_2(a)\prec b)\nonumber\\+r_{\succ_A}(l_{\prec_B}(\beta_1(a))x)\beta_2(b)+\beta_2(b)\succ_B(r_{\succ_A}(x)\beta_1(a)),\\
\label{bieq44}
l_B(\beta_2(a))x\succ_A\alpha_2(y)+l_{\succ_B}(r_A(\beta_2(a)))\alpha_2(y)=l_{\succ_B}(\beta_1\beta_2(a))(x\succ y)\nonumber\\-\beta_2(x)\succ_A(l_{\succ_B}(\beta_1(a))y)-r_{\succ_B}(r_{\succ_A}(y)\beta_1(a))\alpha_2(x),
\\
\label{bieq45}
r_B(a)\alpha_2(x)\succ_A\alpha_2(y)+l_{\succ_B}(l_A(\alpha_2(x))a)\alpha_2(y)=\alpha_1\alpha_2(x)\succ_A(l_{\succ_B}(a)y)\nonumber\\+r_B(r_A(y)a)\alpha_1\alpha_2(x)-l_{\succ_B}(\beta_2(a))(\alpha_1(x)\succ_A y),\\
\label{bieq46}
r_{\succ_B}(\beta_2(a))([\alpha_2(x),y])=\alpha_1\alpha_2(x)\succ_A(r_{\succ_B}(a))y+r_B(r_A(y)a)\alpha_1\alpha_2(y)\nonumber\\-\alpha_2(x)\succ_A(r_{\succ_B}(a)\alpha_1(y))-r_B(l_{\succ_A}(\alpha_1(y))a)\alpha_2(x),\\
 \label{bieq47}
l_{\succ_B}(\beta_1\beta_2(a))(x\prec y)=(l_{\succ_B}(\beta_2(a))x)\prec_A \alpha_2(y)+l_{\prec_B}(r_{\succ_A}(x)\beta_2(a))\alpha_2(y)\nonumber\\+\alpha_2(x)\prec_A(l_B(\beta_1(a))y)+r_{prec_B}(r_A(y)\beta_1(a))\alpha_2(x), \\
\label{bieq48}
\alpha_1\alpha_2(x)\succ_A(l_{\prec_B}(a)y)+r_{\succ_B}(r_A(y)a)\alpha_1\alpha_2(x)=(r_{\succ_B}(a)\alpha_2(x))\prec_A\alpha_2(y)\nonumber\\+l_B(l_{\succ_A}(\alpha_2(x))a)\alpha_2(y)+l_{\prec_B}(\beta_2(a))([\alpha_1(x),y]),\\
\label{bieq49}
\alpha_1\alpha_2(x)\succ_A(r_{\prec_B}(a)y)+r_B(l_A(y)a)\alpha_1\alpha_2(x)=r_{\prec_B}(\beta_2(a))(\alpha_2(x)\succ_A y)\nonumber\\+r_{\prec_B}(l_A(\alpha_1(y)a)\alpha_2(x)+\alpha_2(x)\prec_A(r_B(a)\alpha_1(y)),
\\
\label{bieq50}
l_{\prec_B}(\beta_1\beta_2(a))([x,y])=l_{\prec_B}(r_{\prec_A}(x)a)\alpha_2(y)+(l_{\prec_B}(\beta_2(a))x)\prec_A\alpha_2(y)\nonumber\\+\alpha_2(x)\succ_A(l_{\prec_B}(\beta_1(a))y)+r_B(r_A(y)\beta_1(a))\alpha_2(x),
\\
\label{bieq51}
\alpha_1\alpha_2(x)\prec_A(l_B(a)y)+r_{\prec_B}(r_A(y)a)\alpha_1\alpha_2(x)=(r_{\prec_B}(a)\alpha_2(x))\prec_A\alpha_2(y)\nonumber\\+l_{\prec_B}(l_{\prec_A}(\alpha_2(x))a)\alpha_2(y)+l_{\succ_B}(\beta_2(a))(\alpha_1(x)\prec y),\\
\label{bieq52}
\alpha_1\alpha_2(x)\prec_A(r_B(a)y)+l_{\prec_B}(l_A(y)a)\alpha_1\alpha_2(x)=r_{\prec_B}(\beta_2(a))(\alpha_2(x)\prec y)\nonumber\\
+r_{\succ_B}(l_{\prec_A}(\alpha_1(x))a)\alpha_2(y)+\alpha_2(y)\succ_A(r_{\succ_B}(a)\alpha_1(x)),
\end{eqnarray}
then, there is a BiHom-Leibniz dendriform algebra structure on the direct sum $  A \oplus  B  $ of the underlying linear spaces of
 $  A $ and $  B  $ given for any $ x, y \in  A, a, b \in  B  $ by
\begin{eqnarray*}
(x + a) \prec ( y + b ) &=& (x \prec_{ A} y + r_{\prec_{ B }}(b)x + l_{\prec_{ B }}(a)y)\cr
&+&(l_{\prec_{ A}}(x)b + r_{\prec_{ A}}(y)a + a \prec_{ B } b ), \cr
(x + a) \succ ( y + b ) &=& (x \succ_{ A} y + r_{\succ_{ B }}(b)x + l_{\succ_{ B }}(a)y)\cr
&+& (l_{\succ_{ A}}(x)b + r_{\succ_{ A}}(y)a + a \succ_{ B } b ).
\end{eqnarray*}
\end{thm}
\begin{proof}
The proof  is obtained in a similar way as for Theorem \ref{thm:matchedpairs}.
\qedhere
\end{proof}

 Let $  A \bowtie^{l_{\prec_{ A}}, r_{\prec_{ A}},
l_{\succ_{ A}}, r_{\succ_{ A}}, \beta_{1}, \beta_{2}}_{l_{\prec_{ B }}, r_{\prec_{ B }}, l_{\succ_{ B }},
r_{\succ_{ B }}, \alpha_{1}, \alpha_{2}}  B  $ denote this BiHom-Leibniz dendriform algebra.
\begin{defn}
Let $ ( A, \prec_{ A}, \succ_{ A}, \alpha_{1}, \alpha_{2}) $ and $  ( B , \prec_{ B }, \succ_{ B }, \beta_{1}, \beta_{2}) $
be BiHom-Leibniz dendriform algebras. Suppose there exist linear maps
$ l_{\prec_{ A}}, r_{\prec_{ A}}, l_{\succ_{ A}}, r_{\succ_{ A}} :  A \rightarrow gl( B ),$
and $ l_{\prec_{ B }}, r_{\prec_{ B }}, l_{\succ_{ B }}, r_{\succ_{ B }} :  B  \rightarrow gl( A) $
 such that $(l_{\prec_{ A}}, r_{\prec_{ A}}, l_{\succ_{ A}}, r_{\succ_{ A}}, \beta_{1}, \beta_{2})$ is a bimodule of $  A,$
and $(l_{\prec_{ B }}, r_{\prec_{ B }}, l_{\succ_{ B }}, r_{\succ_{ B }}, \alpha_{1}, \alpha_{2})$ is a bimodule of $  B  $.
If  \eqref{bieq35} - \eqref{bieq52} are satisfied, $( A,  B , l_{\prec_{ A}},
r_{\prec_{ A}}, l_{\succ_{ A}}, r_{\succ_{ A}}, \beta_{1}, \beta_{2}, l_{\prec_{ B }}, r_{\prec_{ B }}, l_{\succ_{ B }},
 r_{\succ_{ B }}, \alpha_{1}, \alpha_{2})$ is called a matched pair of BiHom-Leibniz dendriform algebras.
\end{defn}
\begin{cor}
If $( A, B , l_{\prec_{ A}}, r_{\prec_{ A}}, l_{\succ_{ A}}, r_{\succ_{ A}}, \beta_{1}, \beta_{2},
 l_{\prec_{ B }}, r_{\prec_{ B }}, l_{\succ_{ B }}, r_{\succ_{ B }}, \alpha_{1}, \alpha_{2}) $ is an arbitrary matched pair of BiHom-Leibniz dendriform algebras, then
$(A, B , l_{\prec_{ A}} + l_{\succ_{ A}}, r_{\prec_{ A}} + r_{\succ_{ A}},
l_{\prec_{ B }} + l_{\succ_{ B }},  r_{\prec_{ B }} + r_{\succ_{ B }}, \alpha_{1} + \beta_{1}, \alpha_{2} + \beta_{2})$ is a matched pair of the associated
BiHom-Leibniz algebras $( A, [\cdot,\cdot], \alpha_{1}, \alpha_{2})$ and  $( B , \{\cdot,\cdot\}, \beta_{1}, \beta_{2})$.
\end{cor}
\begin{proof}
Let $(A, B, l_{\prec_{A}}, r_{\prec_{A}}, l_{\succ_{A}}, r_{\succ_{A}}, \beta_1,\beta_2,
 l_{\prec_{B}}, r_{\prec_{B}}, l_{\succ_{B}}, r_{\succ_{B}}, \alpha_1,\alpha_2)$ be some matched pair of BiHom-Leibniz dendriform algebras $(A, \prec_{A}, \succ_{A}, \alpha_1,\alpha_2)$ and $(B, \prec_{B}, \succ_{B},\beta_1,\beta_2)$. In view of Proposition \ref{proposition BiHbimodule}, the linear maps $l_{\prec_{A}} + l_{\succ_{A}}, r_{\prec_{A}} + r_{\succ_{A}}:A\rightarrow gl(B)$ and  $l_{\prec_{B}} + l_{\succ_{B}},  r_{\prec_{B}} + r_{\succ_{B}}:B\rightarrow gl(A)$ are bimodules of the underlying BiHom-Leibniz algebras $(A,[\cdot,\cdot], \alpha_1,\alpha_2)$ and $(B,\{\cdot,\cdot\},\beta_1,\beta_2)$, respectively. Therefore,
 \eqref{bimatch. pair1}-\eqref{bimatch. pair3} are equivalent to \eqref{bieq35}-\eqref{bieq43} and \eqref{bimatch. pair4}-\eqref{bimatch. pair6} are equivalent to \eqref{bieq44}-\eqref{bieq52}.
\end{proof}
\begin{defn}
Let $( A, [\cdot,\cdot], \alpha_{1}, \alpha_{2})$ be a BiHom-Leibniz algebra, and $(l, r, \beta_{1}, \beta_{2}, V)$ be a bimodule. A linear map $T : V \rightarrow  A$
is called an $ \mathcal{O} $-operator associated to $(l, r, \beta_{1}, \beta_{2}, V)$,  if $ T $ satisfies
\begin{align*}
& \alpha_{1} T= T\beta_{2}, \quad \alpha_{2} T= T\beta_{1}, \\
& [T(u), T(v)] = T(l(T(u))v + r(T(v))u) \mbox { for all } u, v \in V.
\end{align*}
\end{defn}

\begin{ex}
Let $( A, [\cdot,\cdot], \alpha_{1}, \alpha_{2})$ be a multiplicative BiHom-Leibniz algebra. Then,
the identity map $id$ is an $ \mathcal{O} $-operator associated to the bimodule $(L,0, \alpha_{1}, \alpha_{2})$.
\end{ex}

\begin{ex}
Let ${\mathcal{A}}=( A, [\cdot,\cdot], \alpha, \beta)$ be a multiplicative BiHom-Leibniz algebra.
A linear map $ K :  A \rightarrow  A $ is called a Rota-Baxter operator on $\mathcal{A}$ of weight zero if $ K $ satisfies
\begin{align*}
K\circ\alpha & =\alpha\circ K, \quad K\circ\beta=\beta\circ K, \\
[K(x), K(y)] &= K([K(x), y] + [x, K(y)]), \ \mbox{ for all } x, y \in  A.
\end{align*}
A Rota-Baxter operator on $ \mathcal{A} $ is an $ \mathcal{O} $-operator associated to the bimodule 
$(L, R, \alpha, \beta)$.
\end{ex}
\begin{thm}\label{theorm Leib to dendr}
Let $( A, [\cdot,\cdot], \alpha_{1}, \alpha_{2})$ be a BiHom-Leibniz algebra, and $(l, r, \beta_{1}, \beta_{2}, V) $ be a bimodule.
Let $ T : V \rightarrow  A $ be an $ \mathcal{O} $-operator associated to $(l, r, \beta_{1}, \beta_{2}, V)$. Then, there exists a BiHom-Leibniz dendriform
algebra structure on $V$ given by
\begin{eqnarray*}
 u\prec v = r(T(u))v , \quad u\succ v = l(T(v))u
\end{eqnarray*}
for all $u, v \in V$. So, there is an associated BiHom-Leibniz algebra structure on $ V $ given by the equation \eqref{associative-dendriform},  and $ T $
is a homomorphism of BiHom-Leibniz algebras. Moreover, $ T(V) = \lbrace { T(v) \setminus v \in V }  \rbrace  \subset  A $ is a BiHom-Leibniz
subalgebra of $  A, $ and there is an induced BiHom-Leibniz dendriform algebra structure on $ T(V) $ given by
\begin{eqnarray}
T(u) \succ T(v) = T(u \succ v),\quad T(u) \prec T(v) = T(u \prec v)
\end{eqnarray}
for all $ u, v \in V $. Its corresponding associated BiHom-Leibniz algebra structure on $ T(V) $ given by \eqref{associative-dendriform} is
just the BiHom-Leibniz subalgebra structure of $  A, $ and $ T $ is a homomorphism of BiHom-Leibniz dendriform algebras.
\end{thm}
\begin{proof}
For any $x, y, z\in V,$ we have
\begin{align*}
&([\beta_2(x),y])\succ\beta_2(z)-\beta_1\beta_2(x)\succ(y\succ z)+\beta_2(y)\succ(\beta_1(x)\succ z)\\
&\quad=(r(T(y))\beta_2(x)+l(T((\beta_2(x))y)\succ\beta_2(z)-\beta_1\beta_2(x)\succ(l(T(y))z)\\&\quad\quad+\beta_2(y)\succ(l(T(\beta_1(x))z)\\
&\quad=l(T(r(T(y))\beta_2(x)+l(T(\beta_2(x))y))\beta_2(z)-l(T(\beta_1\beta_2(x)))l(T(y))z\\
&\quad\quad+l(T(\beta_2(y))l(T(\beta_1(x))z\\
&\quad=l(T(r(T(y))\beta_2(x)+l(\alpha_2(T(x)))y)\beta_2(z)-l(\alpha_1\alpha_2(T(x)))l(T(y))z\\&\quad\quad+l(\alpha_2(T(y))l(\alpha_1(T(x))z\\
&\quad=l([T(\beta_2(x)),T(y)]\beta_2(z)-l([T(\beta_2(x)),T(y)]\beta_2(z)=0.
\end{align*}
The two other axioms are checked in a similar way.
\end{proof}
\begin{ex}
There is three-dimensional BiHom-Leibniz algebra $(A, [\cdot,\cdot],\alpha_1,\alpha_2)$ with multiplications tables for a basis $\{e_1, e_2, e_3\}$:
\begin{center}
\begin{tabular}{c|cccc}
$[\cdot,\cdot]$ & $e_1$ & $e_2$ & $e_3$ \\ \hline
$e_1$ & $0$ & $0$ & $0$  \\
$e_2$ & $0$ & $e_1$ & $0$  \\
$e_3$ & $0$ & $0$ & $0$ \\
\end{tabular}
\end{center}
$$\begin{array}{llll}
    \alpha_1(e_1)=e_1, &\alpha_1(e_2)=-e_2, &\alpha_1(e_3)=-2e_3,\\
    \alpha_2(e_1)=\frac{1}{4}e_1, &\alpha_2(e_2)=\frac{1}{2} e_3, &\alpha_2(e_3)=2 e_3.
\end{array}$$
Let $K$ the Rota-Baxter operator of weight zero defined, for the basis $\{e_1, e_2, e_3\}$, by
$$\begin{array}{llll}
    K(e_1)=\frac{\lambda_1}{2}e_1, &K(e_2)=\lambda_2 e_1+\lambda_1 e_2+\lambda_3 e_3,&K(e_3)=\lambda_4 e_1+\lambda_5 e_3,
\end{array}$$
where $\lambda_i$ are parameters in $\mathbb{K}\setminus\{0\}$.
Now, using the previous Theorem, there is a BiHom-Leibniz dendriform algebra structure on $A$ with
the multiplications "$\prec$" and "$\succ$" are given by $x\prec y=[y,K(x)],~~x\succ y=[K(y),x]$
that is
\begin{center}
\begin{tabular}{c|cccc}
$\prec$ & $e_1$ & $e_2$ & $e_3$ \\ \hline
$e_1$ & $0$ & $0$ & $0$  \\
$e_2$ & $0$ & $\lambda_1 e_1$ & $0$  \\
$e_3$ & $0$ & $0$ & $0$
\end{tabular}
\hspace{1 cm}
\begin{tabular}{c|cccc}
$\succ$ & $e_1$ & $e_2$ & $e_3$ \\ \hline
$e_1$ & $0$ & $0$ & $0$  \\
$e_2$ & $0$ & $\lambda_1 e_1$ & $0$  \\
$e_3$ & $0$ & $0$ & $0$
\end{tabular}
\end{center}
$$\begin{array}{llll}
    \alpha_1(e_1)=e_1, &\alpha_1(e_2)=-e_2,&\alpha_1(e_3)=-2e_3,\\
    \alpha_2(e_1)=\frac{1}{4}e_1, &\alpha_2(e_2)=\frac{1}{2} e_3, 
    &\alpha_2(e_3)=2 e_3.
\end{array}$$
\end{ex}
\begin{cor}
Let $T:V\rightarrow A$ be an $\mathcal{O}$-operator on a BiHom-Leibniz algebra $(A,[\cdot,\cdot],\alpha_1,\alpha_2)$ with respect to a representation $(l,r, \beta_1,\beta_2,V)$.
Let $[u,v]_C=u\prec v+u\succ v,~$ for all $u,v\in V.$
Then, $T$ is a morphism from the BiHom-Leibniz algebra
$(V, [\cdot,\cdot]_C, \beta_1,\beta_2)$ to $(A,[\cdot,\cdot],\alpha_1,\alpha_2).$
\end{cor}
\begin{cor}\label{bidendriform-invertible operator}
Let $( A, [\cdot,\cdot], \alpha, \beta)$ be a multiplicative BiHom-Leibniz algebra. There is a compatible multiplicative BiHom-Leibniz dendriform algebra structure on $  A $
if and only if there exists an invertible $ \mathcal{O} $-operator of $ ( A, [\cdot,\cdot], \alpha, \beta)$.
\end{cor}
\begin{proof}
In fact, if the homomorphism  $ T $ is an invertible $ \mathcal{O}- $operator associated to a bimodule $(l, r, \alpha, \beta, V)$, then
the compatible multiplicative BiHom-Leibniz dendriform algebra structure on $  A $ is given for all
$x, y \in  A$ by
\begin{eqnarray*}
x \prec y = T(l(x)T^{-1}(y)),\quad  x \succ y = T(r(y)T^{-1}(x)).
\end{eqnarray*}
Conversely, let $( A, \prec, \succ, \alpha, \beta)$ be a multiplicative BiHom-Leibniz dendriform algebra, and $( A, [\cdot,\cdot], \alpha, \beta)$
 be its associated BiHom-Leibniz algebra. Then, the identity map $id$ is an $ \mathcal{O}- $operator
associated to the bimodule $(L_{\succ}, R_{\prec}, \alpha, \beta)$ of $( A,[\cdot,\cdot], \alpha, \beta)$.
\end{proof}

\end{document}